\documentclass[12pt]{amsart}

\usepackage{amsmath,amsthm,amssymb,enumitem,verbatim,stmaryrd,xcolor,microtype,graphicx}
\usepackage[T1]{fontenc}
\usepackage[utf8]{inputenc}
\usepackage[english]{babel}
\usepackage[top=3.4cm,bottom=3.4cm,left=3.5cm,right=3.5cm]{geometry}
\usepackage[bookmarksdepth=2,linktoc=page,colorlinks,linkcolor={red!80!black},citecolor={red!80!black},urlcolor={blue!80!black}]{hyperref}
\usepackage{tikz}\usetikzlibrary{matrix,arrows,decorations.markings}
\usepackage{tikz-cd}
\pgfrealjobname{foundations}
\usepackage[mathcal]{euscript}                               
\usepackage{mathptmx}                                        
\usepackage[disable,colorinlistoftodos,bordercolor=orange,backgroundcolor=orange!20,linecolor=orange,textsize=footnotesize]{todonotes} \makeatletter \providecommand \@dotsep{5} \def\listtodoname{List of Todos} \def\listoftodos{\@starttoc{tdo}\listtodoname} \makeatother 

\allowdisplaybreaks 

\usepackage{etoolbox}
\makeatletter
\patchcmd{\@startsection}{\@afterindenttrue}{\@afterindentfalse}{}{}             
\patchcmd{\part}{\bfseries}{\bfseries\LARGE}{}{}
\patchcmd{\section}{\scshape}{\bfseries}{}{}\renewcommand{\@secnumfont}{\bfseries} 
\patchcmd{\@settitle}{\uppercasenonmath\@title}{\large}{}{}
\patchcmd{\@setauthors}{\MakeUppercase}{}{}{}
\addto{\captionsenglish}{} 
\addto{\captionsenglish}{} 
\addto{\captionsenglish}{} 
\makeatother

\usepackage{fancyhdr}

\theoremstyle{plain}
\newtheorem{thm}{Theorem}[section]
\newtheorem{cor}[thm]{Corollary}
\newtheorem{lemma}[thm]{Lemma}
\newtheorem{prop}[thm]{Proposition}
\newtheorem{thmA}{Theorem}  

\newtheorem*{thm*}{Theorem}
\newtheorem*{lem*}{Lemma}

\theoremstyle{definition}
\newtheorem{df}[thm]{Definition}
\newtheorem{rem}[thm]{Remark}

\newtheorem{ex}[thm]{Example}
\newtheorem*{df*}{Definition}
\newtheorem*{ex*}{Example}
\newtheorem*{rem*}{Remark}

\theoremstyle{remark}

\DeclareRobustCommand{\gobblefour}[5]{}    

\DeclareFontFamily{OT1}{pzc}{}                                
\DeclareFontShape{OT1}{pzc}{m}{it}{<-> s * [1.10] pzcmi7t}{}
\DeclareMathAlphabet{\mathpzc}{OT1}{pzc}{m}{it}
\DeclareSymbolFont{sfoperators}{OT1}{bch}{m}{n} \DeclareSymbolFontAlphabet{\mathsf}{sfoperators} \makeatletter\def\operator@font{\mathgroup\symsfoperators}\makeatother 
\DeclareSymbolFont{cmletters}{OML}{cmm}{m}{it}              
\DeclareSymbolFont{cmsymbols}{OMS}{cmsy}{m}{n}
\DeclareSymbolFont{cmlargesymbols}{OMX}{cmex}{m}{n}
\DeclareMathSymbol{\myjmath}{\mathord}{cmletters}{"7C}     \let\jmath\myjmath 
\DeclareMathSymbol{\myamalg}{\mathbin}{cmsymbols}{"71}     
\DeclareMathSymbol{\mycoprod}{\mathop}{cmlargesymbols}{"60}
\DeclareMathSymbol{\myalpha}{\mathord}{cmletters}{"0B}     \let\alpha\myalpha 
\DeclareMathSymbol{\mybeta}{\mathord}{cmletters}{"0C}      \let\beta\mybeta
\DeclareMathSymbol{\mygamma}{\mathord}{cmletters}{"0D}     \let\gamma\mygamma
\DeclareMathSymbol{\mydelta}{\mathord}{cmletters}{"0E}     \let\delta\mydelta
\DeclareMathSymbol{\myepsilon}{\mathord}{cmletters}{"0F}   \let\epsilon\myepsilon
\DeclareMathSymbol{\myzeta}{\mathord}{cmletters}{"10}      \let\zeta\myzeta
\DeclareMathSymbol{\myeta}{\mathord}{cmletters}{"11}       \let\eta\myeta
\DeclareMathSymbol{\mytheta}{\mathord}{cmletters}{"12}     \let\theta\mytheta
\DeclareMathSymbol{\myiota}{\mathord}{cmletters}{"13}      \let\iota\myiota
\DeclareMathSymbol{\mykappa}{\mathord}{cmletters}{"14}     \let\kappa\mykappa
\DeclareMathSymbol{\mylambda}{\mathord}{cmletters}{"15}    \let\lambda\mylambda
\DeclareMathSymbol{\mymu}{\mathord}{cmletters}{"16}        \let\mu\mymu
\DeclareMathSymbol{\mynu}{\mathord}{cmletters}{"17}        \let\nu\mynu
\DeclareMathSymbol{\myxi}{\mathord}{cmletters}{"18}        \let\xi\myxi
\DeclareMathSymbol{\mypi}{\mathord}{cmletters}{"19}        \let\pi\mypi
\DeclareMathSymbol{\myrho}{\mathord}{cmletters}{"1A}       \let\rho\myrho
\DeclareMathSymbol{\mysigma}{\mathord}{cmletters}{"1B}     \let\sigma\mysigma
\DeclareMathSymbol{\mytau}{\mathord}{cmletters}{"1C}       \let\tau\mytau
\DeclareMathSymbol{\myupsilon}{\mathord}{cmletters}{"1D}   \let\upsilon\myupsilon
\DeclareMathSymbol{\myphi}{\mathord}{cmletters}{"1E}       \let\phi\myphi
\DeclareMathSymbol{\mychi}{\mathord}{cmletters}{"1F}       \let\chi\mychi
\DeclareMathSymbol{\mypsi}{\mathord}{cmletters}{"20}       \let\psi\mypsi
\DeclareMathSymbol{\myomega}{\mathord}{cmletters}{"21}     \let\omega\myomega
\DeclareMathSymbol{\myvarepsilon}{\mathord}{cmletters}{"22}\let\varepsilon\myvarepsilon
\DeclareMathSymbol{\myvartheta}{\mathord}{cmletters}{"23}  \let\vartheta\myvartheta
\DeclareMathSymbol{\myvarpi}{\mathord}{cmletters}{"24}     \let\varpi\myvarpi
\DeclareMathSymbol{\myvarrho}{\mathord}{cmletters}{"25}    \let\varrho\myvarrho
\DeclareMathSymbol{\myvarsigma}{\mathord}{cmletters}{"26}  \let\varsigma\myvarsigma
\DeclareMathSymbol{\myvarphi}{\mathord}{cmletters}{"27}    \let\varphi\myvarphi

\DeclareMathOperator{\Hom}{Hom}
\DeclareMathOperator{\Aut}{Aut}

\DeclareMathOperator{\Sym}{Sym}
\DeclareMathOperator{\Pastures}{Pastures}

\DeclareMathOperator{\sign}{sign}
\DeclareMathOperator{\supp}{supp}

\DeclareMathOperator{\OBlpr}{{OBlpr}}
\DeclareMathOperator{\Fields}{{Fields}}
\DeclareMathOperator{\HypFields}{{HypFields}}
\DeclareMathOperator{\PartFields}{{PartFields}}
\DeclareMathOperator{\FuzzRings}{{FuzzRings}}
\DeclareMathOperator{\Tracts}{{Tracts}}

\newcommand\C{{\mathbb C}}
\newcommand\D{{\mathbb D}}

\newcommand\F{{\mathbb F}}

\renewcommand\H{{\mathbb H}}

\newcommand\K{{\mathbb K}}

\newcommand\N{{\mathbb N}}

\renewcommand\P{{\mathbb P}}
\newcommand\Q{{\mathbb Q}}
\newcommand\R{{\mathbb R}}
\renewcommand\S{{\mathbb S}}
\newcommand\T{{\mathbb T}}
\newcommand\U{{\mathbb U}}

\newcommand\W{{\mathbb W}}

\newcommand\Z{{\mathbb Z}}
\newcommand\bI{{\mathbf I}}
\newcommand\bJ{{\mathbf J}}

\newcommand\cB{{\mathcal B}}
\newcommand\cC{{\mathcal C}}
\newcommand\cD{{\mathcal D}}
\newcommand\cE{{\mathcal E}}
\newcommand\cF{{\mathcal F}}

\newcommand\cH{{\mathcal H}}

\newcommand\cP{{\mathcal P}}

\newcommand\cU{{\mathcal U}}

\newcommand\cX{{\mathcal X}}

\newcommand\Funpm{{\F_1^\pm}}
\newcommand\id{\textup{id}}
\renewcommand\max{\textup{max}}

\newcommand\octa{{\scalebox{0.7}{$\diamondsuit$}}}
\renewcommand\char{\textup{char}\; }
\renewcommand\geq{\geqslant}
\renewcommand\leq{\leqslant}
\renewcommand\check{\checkmark}
\newcommand{\gen}[1]{\langle #1 \rangle}
\newcommand{\norm}[1]{|#1|}
\newcommand{\past}[2]{#1\!\sslash\!#2}
\newcommand{\pastgen}[3]{#1\langle #2 \rangle \!\sslash\!\{ #3 \}}
\newcommand{\cross}[5]{\mathchoice{\scalebox{1.3}{$\big[$}\,\raisebox{1pt}{$\begin{matrix}{\scalebox{0.9}{$#1$}}\hspace{-5pt}&{\scalebox{0.9}{$#2$}}\\[-2pt]{\scalebox{0.9}{$#3$}}\hspace{-5pt}&{{\scalebox{0.9}{$#4$}}}\end{matrix}$}\,\scalebox{1.3}{$\big]$}_{#5}}{\big[\begin{smallmatrix}{#1}&{#2}\\{#3}&{#4}\end{smallmatrix}\big]_{#5}}{}{}}   
\newcommand{\crossinv}[5]{\mathchoice{\scalebox{1.3}{$\big[$}\,\raisebox{1pt}{$\begin{matrix}{\scalebox{0.9}{$#1$}}\hspace{-5pt}&{\scalebox{0.9}{$#2$}}\\[-2pt]{\scalebox{0.9}{$#3$}}\hspace{-5pt}&{{\scalebox{0.9}{$#4$}}}\end{matrix}$}\,\scalebox{1.3}{$\big]$}_{#5}^{-1}}{\big[\begin{smallmatrix}{#1}&{#2}\\{#3}&{#4}\end{smallmatrix}\big]_{#5}^{-1}}{}{}}   
\newcommand{\minor}[2]{\backslash #1 / #2}
\newcommand{\hyperplus}{\mathrel{\,\raisebox{-1.1pt}{\larger[-0]{$\boxplus$}}\,}}

\renewcommand{\implies}{\ensuremath{\Rightarrow}}
\renewcommand{\iff}{\ensuremath{\Leftrightarrow}}
\renewcommand\emptyset\varnothing

\title{Foundations of matroids\\[10pt] \normalsize Part 1: Matroids without large uniform minors}

\author{Matthew Baker}
\address{\rm Matthew Baker, School of Mathematics, Georgia Institute of Technology, Atlanta, USA}
\email{mbaker@math.gatech.edu}

\author{Oliver Lorscheid}
\address{\rm Oliver Lorscheid, Instituto Nacional de Matem\'atica Pura e Aplicada, Rio de Janeiro, Brazil}
\email{oliver@impa.br}

\thanks{The first author was supported in part by the NSF Research Grant DMS-1529573. Parts of this project have been carried out while the second author was hosted by the Max Planck Institute in Bonn.}

\begin{document}

\ \vspace{-15pt}

\begin{abstract}
 The {\em foundation} of a matroid is a canonical algebraic invariant which classifies, in a certain precise sense, all representations of the matroid up to rescaling equivalence. Foundations of matroids are {\em pastures}, a simultaneous generalization of partial fields and hyperfields which are special cases of both tracts (as defined by the first author and Bowler) and ordered blue fields (as defined by the second author).

 Using deep results due to Tutte, Dress--Wenzel, and Gelfand--Rybnikov--Stone, we give a presentation for the foundation of a matroid in terms of generators and relations. The generators are certain ``cross-ratios'' generalizing the cross-ratio of four points on a projective line, and the relations encode dependencies between cross-ratios in certain low-rank configurations arising in projective geometry.

 Although the presentation of the foundation is valid for all matroids, it is simplest to apply in the case of matroids {\em without large uniform minors}. i.e., matroids having no minor corresponding to five points on a line or its dual configuration. For such matroids, we obtain a complete classification of all possible foundations.

 We then give a number of applications of this classification theorem, for example:
 \begin{enumerate}
  \item We prove the following strengthening of a 1997 theorem of Lee and Scobee: every orientation of a matroid without large uniform minors comes from a dyadic representation, which is unique up to rescaling.
  \item For a matroid $M$ without large uniform minors, we establish the following strengthening of a 2017 theorem of Ardila--Rinc{\'o}n--Williams: if $M$ is positively oriented then $M$ is representable over every field with at least 3 elements.
  \item Two matroids are said to belong to the same {\em representation class} if they are representable over precisely the same pastures. We prove that there are precisely 12 possibilities for the representation class of a matroid without large uniform minors, exactly three of which are not representable over any field.
 \end{enumerate}
\end{abstract}

\maketitle

\begin{small} \tableofcontents \end{small}


\section*{Introduction}
\label{introduction}


Matroids are a combinatorial abstraction of the notion of linear independence in vector spaces.
If $K$ is a field and $n$ is a positive integer, any linear subspace of $K^n$ gives rise to a matroid; such matroids are called {\em representable} over $K$. The task of deciding whether or not certain families of matroids are representable over certain kinds of fields has occupied a plethora of papers in the matroid theory literature.

Dress and Wenzel \cite{Dress-Wenzel89,Dress-Wenzel90} introduced the {\em Tutte group} and the {\em inner Tutte group} of a matroid. These are abelian groups which, in a certain precise sense, can be used to understand representations of $M$ over all so-called {\em fuzzy rings} (which, in particular include fields). Dress and Wenzel gave several different presentations for these groups in terms of generators and relations, and Gelfand--Rybnikov--Stone \cite{{Gelfand-Rybnikov-Stone95}} subsequently gave additional presentations for the inner Tutte group of $M$. 
The Dress--Wenzel theory of Tutte groups, inner Tutte groups, and fuzzy rings is powerful but lacks simple definitions and characterizations in terms of universal properties.

In their 1996 paper \cite{Semple-Whittle96}, Semple and Whittle generalized the notion of matroid representations to {\em partial fields} (which are special cases of fuzzy rings); this allows one to consider certain families of matroids (e.g. regular or dyadic) as analogous to matroids over a field, and to prove new theorems in the spirit of Tutte's theorem that a matroid is both binary and ternary if and only if it is regular.
Pendavingh and van Zwam \cite{Pendavingh-vanZwam10a,Pendavingh-vanZwam10b} subsequently introduced the {\em universal partial field} of a matroid $M$, which governs the representations of $M$ over all partial fields. Unfortunately, most matroids (asymptotically 100\%, in fact, by a theorem of Nelson \cite{Nelson18}) are not representable over {\em any} partial field, and in this case the universal partial field gives no information. One can view non-representable matroids as the ``dark matter'' of matroid theory: they are ubiquitous but somehow mysterious. 

Using the theory of matroids over partial hyperstructures presented in \cite{Baker-Bowler19} (which has been continued in \cite{Anderson19}, \cite{Bowler-Pendavingh19} and \cite{Pendavingh18}), we introduced in \cite{Baker-Lorscheid18} a generalization of the universal partial field which we call the {\em foundation} of a matroid. 
The foundation is a kind of algebraic object which we call a {\em pasture}; pastures include both hyperfields and partial fields and form a natural class of ``field-like'' objects within the second author's theory of {\em ordered blueprints} in \cite{Lorscheid15}.
The category of pastures has various desirable categorical properties (e.g., the existence of products and co-products)
which makes it a natural context in which to study algebraic invariants of matroids. Pastures are closely related to fuzzy rings, but they are axiomatically much simpler.

One advantage of the foundation over the universal partial field is that the foundation exists for {\em every} matroid $M$, not just matroids that are representable over some field. 
Moreover, unlike the inner Tutte group, the foundation of a matroid is characterized by a universal property which immediately clarifies its importance and establishes its naturality.

More precisely, the foundation of a matroid $M$ represents the functor taking a pasture $F$ to the set of {\em rescaling equivalence classes} of $F$-representations of $M$; in particular, $M$ is representable over a pasture $F$ if and only if there is a morphism from the foundation of $M$ to $F$.

Our first main result (Theorem~\ref{thm: presentation of foundations in terms of bases}) gives a precise and useful description of the foundation of a matroid in terms of generators and relations. Although this theorem applies to {\em all} matroids, it is easiest to apply in the case of matroids {\em without large uniform minors}, by which we mean matroids which do not have minors isomorphic to either $U^2_5$ or $U^3_5$.\footnote{Note that if $M$ has no minor of type $U^2_5$ or $U^3_5$, then $M$ also has no uniform minor $U^r_n$ with $n\geq 5$ and $2\leq r\leq n-2$, hence the term ``large''.}
For such matroids, we obtain a complete classification (Theorem~\ref{thm: structure theorem for matroids without large uniform minors}) of all possible foundations, from which one can read off just about any desired representability property. This applies, notably, to the dark matter of matroid theory: we show, for example, that there are precisely three different representation classes of matroids without large uniform minors which are not representable over any field. 
The applications of Theorem~\ref{thm: structure theorem for matroids without large uniform minors} which we present in Section~\ref{section: applications} are merely a representative sample of the kinds of things one can deduce from this structural result.

We now give a somewhat more precise introduction to the main concepts, definitions, and results in the present paper.

\subsection*{A quick introduction to pastures}

A field $K$ can be thought of as an abelian group $G = (K^\times,\cdot,1)$, a multiplicatively absorbing element $0$, and a binary operation $+$ on $K = G \cup \{ 0 \}$ which satisfies certain additional natural axioms (e.g.~commutativity, associativity, distributivity, and the existence of additive inverses).
Pastures are a generalization of the notion of field in which we still have a multiplicative abelian group $G$, an absorbing element $0$, and an ``additive structure'', but we relax the requirement that the additive structure come from a binary operation.
The following two examples are illustrative of the type of relaxations we have in mind.

\begin{ex*}[Krasner hyperfield]
 As a pasture, the Krasner hyperfield $\K$ consists of the multiplicative monoid $\{0,1 \}$ with $0\cdot x = 0$ and $1\cdot 1=1$ and the additive relations $0+x=x$, $1+1=1$, and $1+1=0$.
 Note, in particular, that {\em both} $1+1=1$ and $1+1=0$ are true, and in particular the additive structure is not derived from a binary operation. 
 The fact that $1+1$ is equal to two different things may seem counterintuitive at first, but if we think of 1 as a symbol meaning ``non-zero'', it is simply a reflection of the fact that the sum of two non-zero elements (in a field, say) can be either non-zero or zero.
\end{ex*}

\begin{ex*}[Regular partial field]
As a pasture, the regular partial field $\F_1^\pm$ consists of the multiplicative monoid $\{0,1,-1 \}$ with $0\cdot x = 0$, $1\cdot 1=1$, $1 \cdot (-1) = -1$, and $(-1) \cdot (-1) = 1$, together with the additive relations $0+x=x$ and $1 + (-1) = 0$.
Note, in particular, that there is no additive relation of the form $1 + 1 = x$ or $(-1) + (-1) = x$, so that once again the additive structure is not derived from a binary operation (but for a different reason: here, $1+1$ is undefined rather than being multi-valued).
We think of $\F_1^\pm$ as encoding the restriction of addition and multiplication in the ring $\Z$ to the multiplicative subset $\{ 0, \pm 1\}$.
\end{ex*}

In general, we will require that a pasture $P$ has an involution $x \mapsto -x$ (which is trivial in the case of $\K$), and we can use this involution to rewrite additive relations of the form $x+y=z$ as $x+y-z = 0$. It turns out to be more convenient to define pastures using this formalism, and from now on we view the expression $x+y=z$ as shorthand for $x+y + (-z)=0$.
For additional notational convenience, we identify relations of the form $x+y+z=0$ with triples $(x,y,z)$; the set of all such triples will be denoted $N_P$ and called the {\em null set} of the pasture.

More formally, a {\em pasture} is a multiplicative monoid-with-zero $P$ such that $P^\times = P \backslash \{ 0 \}$ is an abelian group, an involution $x \mapsto -x$ on $P$ fixing $0$, and a subset $N_P$ of $P^3$ such that:
\begin{enumerate}
\item (Symmetry) $N_P$ is invariant under the natural action of $S_3$ on $P^3$.
\item (Weak Distributivity) $N_P$ is invariant under the diagonal action of $P^\times$ on $P^3$.
\item (Unique Weak Inverses) $(0,x,y) \in N_P$ if and only if $y=-x$.
\end{enumerate}

If we set $x \hyperplus y = \{ z \in P \; : \; x+y = z \}$, then the pasture $P$ corresponds to a field if and only if $\hyperplus$ is an associative binary operation. If $x \hyperplus y$ contains {\em at least one} element for all $x,y \in P$ and $\hyperplus$ is associative (in the sense of set-wise addition), we call $P$ a {\em hyperfield}.
If $x \hyperplus y$ contains {\em at most one} element for all $x,y \in P$ and satisfies a suitable associative law, we call $P$ a {\em partial field}.
Pastures generalize (and simplify) both hyperfields and partial fields by imposing no conditions on the size of the sets $x \hyperplus y$ and no associativity conditions.

\begin{ex*}[Hyperfields]
Let $K$ be a field and let $G \leq K^\times$ be a multiplicative subgroup. Then the quotient monoid $K/G = (K^\times / G) \cup \{ 0 \}$ is naturally a hyperfield: the additive relations are all expressions of the form $[x]+[y]=[z]$ for which there exist $a,b,c \in G$ such that $ax+by=cz$.
For example, $\R / \R^\times$ is isomorphic to the Krasner hyperfield $\K$, $\R / \R_{>0}$ is isomorphic to the sign hyperfield $\S$ (cf.~\cite[Example 2.13]{Baker-Bowler19}), and if $p\geq 7$ is a prime number with $p \equiv 3 \pmod{4}$ then $\F_p / (\F_p^\times)^2$ is isomorphic to the weak sign hyperfield $\W$ (cf.~\cite[Example 2.13]{Baker-Bowler19}).
However, not every hyperfield arises in this way (cf.~\cite{Baker-Jin19,Massouros85}).
\end{ex*}

\begin{ex*}[Partial fields]
Let $R$ be a commutative ring and let $G \leq R^\times$ be a subgroup of the unit group of $R$ containing $-1$.
Then $P = G \cup \{ 0 \}$ is naturally a partial field: the additive relations are all expressions of the form $x+y=z$ with $x,y,z \in G \cup \{0 \}$ such that $x+y=z$ in $R$.
Unlike the situation with hyperfields, every partial field arises from this construction (cf.~\cite[Theorem 2.16]{Pendavingh-vanZwam10b}).
\end{ex*}
 
\begin{ex*}[Partial fields, continued]
If $R$ is a commutative ring, let $P(R)$ be the partial field corresponding to $R^\times \subset R$. In this paper, we will make extensive use of the following partial fields:
\begin{enumerate}
\item $\F_1^\pm = P(\Z)$. We call this the {\em regular partial field}.
\item $\D = P(\Z[\frac{1}{2}])$. We call this the {\em dyadic partial field}.
\item $\H = P(\Z[\zeta_6])$, where $\zeta_6 \in \C$ is a primitive sixth root of unity. We call this the {\em hexagonal partial field}.\footnote{In \cite{Pendavingh-vanZwam10b} the partial field $\H$ is denoted $\S$, but in our context that would conflict with the established terminology for the sign hyperfield, so we re-christen it as $\H$. The partial field $\U$ is denoted $\U_1$ in \cite{Pendavingh-vanZwam10b}.}
\item $\U = P(\Z[x,\frac{1}{x},\frac{1}{1-x}])$, where $x$ is an indeterminate. We call this the {\em near-regular partial field}. 
\end{enumerate}
\end{ex*}

\begin{ex*}[Fields]
It is perhaps worth pointing out explicitly that fields are special cases of both hyperfields and partial fields; in fact, they are precisely the pastures which are both hyperfields and partial fields. Since we will be making extensive use of the finite fields $\F_2$ and $\F_3$ in this paper, here is how to explicitly realize these fields as pastures:
\begin{enumerate}
\item $\F_2$ has as its underlying monoid $\{ 0,1 \}$ with the usual multiplication. The involution $x \mapsto -x$ is trivial, and the $3$-term additive relations are $0+0+0=0$ and $0+1+1=0$ (and all permutations thereof).
\item $\F_3$ has as its underlying monoid $\{ 0,1, -1 \}$ with the usual multiplication. The involution $x \mapsto -x$ sends $0$ to $0$ and $1$ to $-1$. The $3$-term additive relations are $0+0+0=0$, $1+(-1)+0=0$ (and all permutations thereof), and $1+1+1=0$.
\end{enumerate}
\end{ex*}

A {\em morphism} of pastures is a multiplicative map
$f : P \to P'$ of monoids such that $f(0)=0$, $f(1)=1$ and $f(x)+f(y)+f(z)=0$ in $P'$ whenever $x+y+z=0$ in $P$.
Pastures form a category whose initial object is $\F_1^\pm$ and whose final object is $\K$.

\subsection*{Representations of matroids over pastures and the foundation of a matroid}

Let $M$ be a matroid of rank $r$ on the finite set $E$, and let $P$ be a pasture.

A {\em $P$-representation} of $M$ is a function $\Delta : E^r \to P$ such that:
\begin{enumerate}
\item $\Delta(e_1,\ldots,e_r) \neq 0$ if and only if $\{ e_1,\ldots,e_r \}$ is a basis of $M$.
\item $\Delta(\sigma(e_1),\ldots,\sigma(e_r)) = \sign(\sigma) \cdot \Delta(e_1,\ldots,e_r)$ for all permutations $\sigma \in S_r$.
\item $\Delta$ satisfies the {\em 3-term Pl{\"u}cker relations}: for all $\bJ \in E^{r-2}$ and all $(e_1,e_2,e_3,e_4) \in E^4$, the null set $N_P$ of $P$ contains the additive relation
\[
 \Delta(\bJ e_1e_2)\cdot \Delta(\bJ e_3e_4) - \Delta(\bJ e_1e_3) \cdot \Delta(\bJ e_2e_4) + \Delta(\bJ e_1e_4) \cdot \Delta(\bJ e_2e_3) \ = \ 0,
\]
where $\bJ e_ie_j := (j_1,\dotsc,j_{r-2},e_i,e_j)$.
\end{enumerate}

\begin{df*}\
\begin{enumerate}
\item $M$ is {\em representable} over $P$ if there is at least one $P$-representation of $M$. 
\item Two $P$-representations $\Delta$ and $\Delta'$ are {\em isomorphic} if there exists $c\in P^\times$ such that $\Delta'(e_1,\ldots,e_r) = c \Delta(e_1,\ldots,e_r)$ for all $(e_1,\ldots,e_r) \in E^r$.\footnote{An isomorphism class of $P$-representations of $M$ is the same thing as a {\em weak $P$-matroid} whose support is $M$, in the terminology of \cite{Baker-Bowler19}.}
\item $\Delta$ and $\Delta'$ \emph{rescaling equivalent} if there exist $c \in P^\times$ and a map $d : E \to P^\times$ such that $\Delta'(e_1,\ldots,e_r) = c \cdot d(e_1) \cdots d(e_r) \cdot \Delta(e_1,\ldots,e_r)$ for all $(e_1,\ldots,e_r) \in E^r$.
\item We denote by $\cX^I_M(P)$ (resp. $\cX^R_M(P)$) the set of isomorphism classes (resp. rescaling classes) of $P$-representations of $M$.\footnote{In \cite{Baker-Lorscheid18}, these sets are denoted $\cX^w_M(P)$ and $\cX^f_M(P)$, respectively.}
\end{enumerate}
\end{df*}

\begin{ex*}
By the results in \cite{Baker-Bowler19} and \cite{Baker-Lorscheid18}, we have:
\begin{enumerate}
\item If $K$ is a field, the isomorphism classes of $K$-representations of $M$ are naturally in bijection with $r$-dimensional subspaces of $K^E$ (the $K$-vector space of functions from $E$ to $K$) whose underlying matroid is $M$.
\item Every matroid has a unique representation over the Krasner hyperfield $\K$.
\item If $P$ is a partial field, $M$ is representable over $P$ if and only if it is representable by a $P$-matrix in the sense of \cite{Pendavingh-vanZwam10b}. 
In particular, a matroid is regular (i.e., representable over $\Z$ by a totally unimodular matrix) if and only if it is representable over the partial field $\F_1^\pm$. A regular matroid will in general have many different (non-isomorphic) representations over $\F_1^\pm$, but there is a unique rescaling class of such representations.
\item A matroid is orientable if and only if it is representable over the sign hyperfield $\S$. An {\em orientation} of $M$ is the same thing as an $\S$-representation, and in this case rescaling equivalence is usually called {\em reorientation equivalence}.
\end{enumerate}
\end{ex*}

For fixed $M$ the map taking a pasture $P$ to the set $\cX^I_M(P)$ (resp. $\cX^R_M(P)$) is a {\em functor}. In particular, 
if $f : P_1 \to P_2$ is a morphism of pastures, there are natural maps $\cX^I_M(P_1) \to \cX^I_M(P_2)$ and $\cX^R_M(P_1) \to \cX^R_M(P_2)$. 

We now come to the key result from \cite{Baker-Lorscheid18} motivating the present paper:

\begin{thm*}
Given a matroid $M$, the functor taking a pasture $P$ to the set $\cX^I_M(P)$ is representable by a pasture $P_M$  which we call the {\em universal pasture} of $M$.
In other words, we have a natural isomorphism
\begin{equation} \label{eq:representable1}
\Hom(P_M, -) \simeq \cX^I_M.
\end{equation}

The functor taking a pasture $P$ to the set $\cX^R_M(P)$ is representable by a subpasture $F_M$ of $P_M$ which we call the {\em foundation} of $M$, i.e. there is a natural isomorphism
\begin{equation} \label{eq:representabl2}
\Hom(F_M,-) \simeq \cX^R_M.
\end{equation}
\end{thm*}

For various reasons, including the fact that the foundation can be presented by generators and relations ``induced from small minors'', we will mainly focus in this paper on studying the foundation of $M$ rather than the universal pasture. Note that both $P_M$ and $F_M$ have the property that $M$ is representable over a pasture $P$ if and only if there is a morphism from $P_M$ (resp. $F_M$) to $P$.

\begin{rem*}\
\begin{enumerate}
\item The universal partial field and foundation behave nicely with respect to various matroid operations. For example, the universal partial fields (resp. foundations) of $M$ and its dual matroid $M^*$ are canonically isomorphic. And there is a natural morphism from the universal partial field (resp. foundation) of a minor $N=M\minor IJ$ of $M$ to the universal partial field (resp. foundation) of $M$.
\item The multiplicative group $P_M^\times$ (resp. $F_M^\times$) of the universal partial field (resp. foundation) of $M$ is isomorphic to the {\em Tutte group} (resp. {\em inner Tutte group}) of Dress and Wenzel \cite[Definition 1.6]{Dress-Wenzel89}.
\end{enumerate}
\end{rem*}

If we take $P = P_M$ in (\ref{eq:representable1}), the identity map is a distinguished element of $\Hom(P_M,P_M)$.
It therefore corresponds to a distinguished element $\hat{\Delta}_M \in \cX^I_M(P_M)$, which (by abuse of terminology) we call the {\em universal representation} of $M$.
(Technically speaking, the universal representation is actually an isomorphism class of representations.)

\begin{rem*}
When $F_M$ is a partial field, the foundation coincides with the universal partial field of \cite{Pendavingh-vanZwam10a}. 
However, when $M$ is not representable over any field, the universal partial field does not exist. 
On the other hand, the foundation of $M$ is always well-defined; this is  one sense in which the theory of pastures helps us explore the ``dark matter'' of the matroid universe.
\end{rem*}

\subsection*{Products and coproducts}

The category of pastures admits finite products and co-products (a.k.a.~tensor products). This is a key advantage of pastures over the categories of fields, partial fields, and hyperfields, none of which admit both products and co-products. The relevance of such considerations to matroid theory is illustrated by the following observations:
\begin{enumerate}
\item $M$ is representable over both $P_1$ and $P_2$ if and only if $M$ is representable over the product pasture $P_1 \times P_2$. (This is immediate from the universal property of the foundation and of categorical products.)
\item If $M_1$ and $M_2$ are matroids, the foundation of the direct sum $M_1 \oplus M_2$ is canonically isomorphic to the tensor product $F_{M_1} \otimes F_{M_2}$, and similarly for the $2$-sum of $M_1$ and $M_2$. (These facts, along with some applications, will be discussed in detail a follow-up paper.)
\item Tensor products of pastures are needed in order to state and apply the main theorem of this paper, the classification theorem for foundations of matroids without large uniform minors
(Theorem~\ref{thm: structure theorem for matroids without large uniform minors} below). 
\end{enumerate}

In order to illustrate the utility of categorical considerations for studying matroid representations, we briefly discuss a couple of key examples.

\begin{ex*}
The product of the fields $\F_2$ and $\F_3$, considered as pastures, is isomorphic to the regular partial field $\F_1^\pm$. 
As an immediate consequence, we obtain Tutte's celebrated result that a matroid $M$ is representable over every field if and only if $M$ is regular.
(Proof: If $M$ is regular then since $\F_1^\pm$ is an initial object in the category of pastures, $M$ is representable over every pasture, and in particular over every field. Conversely, if $M$ is representable over every field, then it is in particular representable over both $\F_2$ and $\F_3$, hence over their product $\F_1^\pm$, and thus $M$ is regular.)

One can, in the same way, establish Whittle's theorem that a matroid is representable over both $\F_3$ and $\F_4$ if and only if it is hexagonal, i.e., representable over the partial field $\H$.

These kind of arguments are well-known in the theory of partial fields; however, the theory of pastures is more flexible. For example, the product of the field $\F_2$ and the hyperfield $\S$ is also isomorphic to the partial field $\F_1^\pm$.
In this way, we obtain a unified proof of the result of Tutte just mentioned and the theorem of Bland and Las Vergnas that a matroid is regular if and only if it is both binary and orientable \cite{Bland-LasVergnas78}.
\end{ex*}

\begin{ex*}
If we try to extend this type of argument to more general pastures, we run into some intriguing complications. As an illuminating example, consider the theorem of Lee and Scobee 
\cite{Lee-Scobee99} that a matroid is both ternary and orientable if and only if it is dyadic, i.e., representable over the partial field $\D$.
In this case, the product of $\F_3$ and $\S$ is {\em not}
isomorphic to $\D$; there is merely a morphism from $\D$ to $\F_3 \times \S$.
The theorem of Lee and Scobee therefore lies deeper than the theorems mentioned in the previous example; proving it requires establishing, in particular, that {\em $\F_3 \times \S$ is not the foundation of any matroid}. 

To do this, one needs a structural understanding of foundations, which we obtain by utilizing highly non-trivial results of Tutte, Dress--Wenzel, and Gelfand--Rybnikov--Stone.
The result of our analysis, in the context of matroids which are both ternary and orientable, is that {\em every morphism from the foundation of some matroid to $\F_3 \times \S$ lifts uniquely to $\D$}. More precisely, we prove that if $M$ is a matroid without large uniform minors (e.g. if $M$ is ternary), then the morphism $\D \to \S$ induces a canonical bijection $\Hom(F_M,\D) \to \Hom(F_M, \S)$. This gives a new and non-trivial strengthening of the Lee--Scobee theorem.
The proof goes roughly as follows: 
by Theorem~\ref{thmB} we have $F_M \cong F_1 \otimes \cdots \otimes F_s$, where each $F_i$ belongs to an explicit finite set $\cP$ of pastures. By categorical considerations, the statement that a morphism $f : F_M \to \S$ lifts uniquely to $\D$ is equivalent to the statement that $\Hom(P,\S)=\Hom(P,\D)$ for all $P \in \cP$, and this can be checked by concrete elementary computations.
\end{ex*}

\subsection*{Universal cross ratios}

In order to explain why the ``large'' uniform minors $U^2_5$ and $U^3_5$ play a special role in the theory of foundations, we need to first explain the concept of a universal cross ratio, which is intimately related to $U^2_4$-minors.

Let $M$ be a matroid of rank $r$, let $P$ be a pasture, and let $\Delta$ be a $P$-representation of $M$.
Let $\bJ \in E^{r-2}$ have distinct coordinates and let $J$ be the corresponding unordered subset of $E$ of size $r-2$.
If $\Delta(\bJ e_1e_4)$ and $\Delta(\bJ e_2e_3)$ are both non-zero (i.e., if $J \cup \{ e_1,e_4 \}$ and $J \cup \{ e_2,e_3 \}$ are both bases of $M$), then we can rewrite the 3-term Pl{\"u}cker relation
\[
\Delta(\bJ e_1e_2) \Delta(\bJ e_3e_4) - \Delta(\bJ e_1e_3) \Delta(\bJ e_2e_4) + \Delta(\bJ e_1e_4) \Delta(\bJ e_2e_3) = 0
\]
as 
\[
\frac{\Delta(\bJ e_1e_3) \Delta(\bJ e_2e_4)}{\Delta(\bJ e_1e_4) \Delta(\bJ e_2e_3)} + \frac{\Delta(\bJ e_1e_2) \Delta(\bJ e_3e_4)}{\Delta(\bJ e_1e_4) \Delta(\bJ e_3e_2)} = 1.
\]

Moreover, as one easily checks, the quantities $\frac{\Delta(\bJ e_1e_3) \Delta(\bJ e_2e_4)}{\Delta(\bJ e_1e_4) \Delta(\bJ e_2e_3)}$ and $\frac{\Delta(\bJ e_1e_2) \Delta(\bJ e_3e_4)}{\Delta(\bJ e_1e_4) \Delta(\bJ e_3e_2)}$ are invariant under rescaling equivalence and do not depend on the choice of ordering of elements of $J$.
In particular, 
\[
\cross {e_1}{e_2}{e_3}{e_4}{\Delta,\bJ} := \frac{\Delta(\bJ e_1e_3) \Delta(\bJ e_2e_4)}{\Delta(\bJ e_1e_4) \Delta(\bJ e_2e_3)}
\]
depends only on $J$ and on the rescaling class $[\Delta]$ of $\Delta$ in $\cX^R_M(P)$.

The cross ratio associated to the universal representation $\hat{\Delta}_M : E^r \to P_M$ plays an especially important role in our theory. For notational convenience, we set
\[
\cross {e_1}{e_2}{e_3}{e_4}{M,J} :=
\cross {e_1}{e_2}{e_3}{e_4}{\hat{\Delta}_M,\bJ }.
\]
We will write $\cross {e_1}{e_2}{e_3}{e_4}{J}$ instead of $\cross {e_1}{e_2}{e_3}{e_4}{M,J}$ when $M$ is understood. 

Using the fact that $\cross {e_1}{e_2}{e_3}{e_4}{\hat{\Delta}_M,J }$ depends only on the rescaling class of $\widehat\Delta_M$, one sees easily that 
$\cross {e_1}{e_2}{e_3}{e_4}{J}$, which {\em a priori} is an element of the universal pasture $P_M$, in fact belongs to the foundation $F_M$. 

We call elements of $F_M$ of the form $\cross {e_1}{e_2}{e_3}{e_4}{J}$ {\em universal cross ratios} of $M$. When $J = \emptyset$ we omit the subscript entirely.
By \cite[Lemma~7.7]{Baker-Lorscheid18}, we have:

\begin{lem*} 
The foundation of $M$ is generated by its universal cross ratios.
\end{lem*}

\begin{rem*}\
\begin{enumerate}
\item When $J = \emptyset$ and $M=U^2_4$ is the uniform matroid of rank 2 on the 4-element set $\{ 1,2,3,4 \}$, the quantity $\cross {1}{2}{3}{4}{}$ can be viewed as a ``universal'' version of the usual cross-ratio of four points on a projective line.  
The fact that the cross-ratio is the {\em only} projective invariant of four points on a line corresponds to the fact that the foundation of $U^2_4$ is isomorphic to the partial field $\U=P(\Z[x,\frac{1}{x},\frac{1}{1-x}])$ described above. The six different values of $\cross {\sigma(1)}{\sigma(2)}{\sigma(3)}{\sigma(4)}{}$ for $\sigma \in S_4$ correspond to the elements 
$x,1-x,\frac{1}{x},1-\frac{1}{x},\frac{1}{1-x}$, and $1 - \frac{1}{1-x}$ of $\U$.
\item More generally, we can associate a universal cross ratio to each $U^2_4$-minor $N=M\minor IJ$ of $M$ (together with an ordering of the ground set of $N$) via the natural map from $F_N$ to $F_M$, and every universal cross ratio arises from this construction. 
\end{enumerate}
\end{rem*}

\subsection*{The structure theorem for foundations of matroids without large uniform minors}

In order to calculate and classify foundations of matroids, in addition to knowing that the universal cross ratios generate $F_M$, we need to understand the relations between these generators.

\begin{ex*}
The universal cross ratios of the uniform matroid $U^2_5$ on $\{ 1,2,3,4,5 \}$ satisfy certain {\em tip relations} of the form
 \[
  \cross 1234{} \cdot \cross 1245{} \cdot \cross 1253{} \ = \ 1.
 \]
By duality, the universal cross ratios of $U^3_5$ satisfy similar identities which we call the {\em cotip relations}.
\end{ex*}

The theoretical tool which allows one to understand {\em all} relations between universal cross ratios is Tutte's Homotopy Theorem \cite{Tutte58a,Tutte58b,Tutte65}
(or, more specifically, \cite[Theorem 4]{Gelfand-Rybnikov-Stone95}, whose proof is based on Tutte's Homotopy Theorem).
We give an informal description here; a more precise version is given in Theorem~\ref{thm: presentation of foundations in terms of bases} below.
To state the result, we say that a relation between universal cross-ratios of $M$ is {\em inherited} from a minor 
$N=M\minor IJ$ if it is the image (with respect to the natural inclusion $F_N \subseteq F_M$) of a relation between universal cross ratios in $F_N$.

\begin{thmA}\label{thmA}
Every relation between universal cross ratios of a matroid $M$ is inherited from a minor on a $6$-element set.
The foundation of $M$ is generated as an $\F_1^\pm$-algebra by such generators and relations, together with the relation $-1 = 1$ if $M$ has a minor isomorphic to either the Fano matroid $F_7$ or its dual.
\end{thmA}

The most complicated relations between universal cross ratios come from the tip and cotip relations in $U^2_5$ and $U^3_5$, respectively (six-element minors and non-uniform five-element minors only contribute additional relations identifying certain cross ratios with one another). In the absence of such minors, we can completely classify all possible foundations. Roughly speaking, the conclusion is that the foundation of a matroid is the tensor product of copies of $\F_2$ and quotients of $\U$ (the foundation of $U^2_4$) by groups of automorphisms. By calculating all possible quotients of $\U$ by automorphisms, we obtain the following result (Theorem \ref{thm: structure theorem for matroids without large uniform minors}):

\begin{thmA}\label{thmB}
 Let $M$ be a matroid without large uniform minors and $F_M$ its foundation. Then 
 \[
  F_M \ \simeq \ F_1\otimes \dotsb \otimes F_r
 \]
 for some $r\geq 0$ and pastures $F_1,\dotsc,F_r\in\{\U,\D,\H,\F_3,\F_2\}$.
\end{thmA}

\begin{rem*}
In a sequel paper, we will show that 
every pasture of the form $F_1\otimes \dotsb \otimes F_r$ with $F_1,\dotsc,F_r\in\{\U,\D,\H,\F_3,\F_2\}$ is the foundation of some matroid. 
\end{rem*}

\subsection*{Consequences of the structure theorem}

A matroid $M$ is representable over a pasture $P$ if and only if there is a morphism from the foundation $F_M$ of $M$ to $P$. If $M$ is without large uniform minors (which is automatic if $M$ is binary or ternary), then by Theorem \ref{thm: structure theorem for matroids without large uniform minors} its foundation is isomorphic to a tensor product of copies $F_i$ of $\U$, $\D$, $\H$, $\F_3$ and $\F_2$. There is a morphism from $F_M$ to $P$ if and only if there is a morphism from each $F_i$ to $P$, so one readily obtains various theorems about representability of such matroids.

We mention just a selection of sample applications from the more complete list of results in section \ref{section: applications}. For instance, our method yields short proofs of the excluded minor characterizations of regular, binary and ternary matroids (Theorems \ref{thm: forbidden minors for binary matroids} and \ref{thm: forbidden minors for ternary matroids}). We find a similarly short proof for Brylawski-Lucas's result that every matroid has at most one rescaling class over $\F_3$ (Theorem \ref{thm: unique rescaling class over pastures with at most one fundamental element} and Remark \ref{rem: pastures with at most one fundamental element}).

As already mentioned, we derive a strengthening of a theorem by Lee and Scobee (\cite{Lee-Scobee99}) on lifts of oriented matroids. The lifting result assumes a particularly strong form in the case of {\em positively oriented matroids}, improving on a result by Ardila, Rinc\'on and Williams (\cite{Ardila-Rincon-Williams17}). The following summarizes Theorems \ref{thm: oriented matroids without large minors are uniquely dyadic} and \ref{thm: lifts of positive orientations}:

\begin{thmA}\label{thmC}
 Let $M$ be an oriented matroid whose underlying matroid is without large uniform minors. Then $M$ is uniquely dyadic up to rescaling. If $M$ is positively oriented, then $M$ is near-regular.
\end{thmA}

In Theorem \ref{thm: realizability criteria for S and P and W}, we derive similar statements for the weak hyperfield of signs $\W$ and the phase hyperfield $\P$; cf.\ section \ref{subsubsection: examples} for definitions. Namely, a matroid $M$ without large uniform minors is ternary if it is representable over $\W$, and is representable over $\D\otimes\H$ if it is representable over $\P$.

We define the \emph{representation class of a matroid $M$} as the class $\cP_M$ of all pastures $P$ over which $M$ is representable. Two matroids $M$ and $M'$ are \emph{representation equivalent} if $\cP_M=\cP_{M'}$. The following is Theorem \ref{thm: representation classes of matroids without large uniform minors}.


\begin{thmA}\label{thmD}
 Let $M$ be a matroid without large uniform minors. Then there are precisely 12 possibilities for the representation class of $M$. Nine of these classes are representable over some field, and the other three are not.
\end{thmA}

The structure theorem also provides short proofs of various characterizations (some new, some previously known by other methods) of certain classes of matroids. 
The following summarizes Theorems \ref{thm: characterization of regular matroids}--\ref{thm: characterization of representable matroids}:

\begin{thmA}\label{thmE}
 Let $M$ be a matroid without large uniform minors and $F_M$ its foundation. Then all conditions in a given row in Table \ref{table: equivalent characterization of certain classes of matroids} are equivalent, where the conditions should be read as follows:
 \begin{enumerate}
  \item The first column describes the class by name (cf. Definition~\ref{df:subclasses} for any unfamiliar terms).
  \item The second column characterizes the class in terms of the factors $F_i$ that may appear in a decomposition $F_M\simeq \bigotimes F_i$, as in Theorem \ref{thmB}.
  \item The third column lists various \emph{classifying pastures $P$}, separated by slashes, which means that $M$ is contained in the class in question if and only if it is representable over $P$.
 \end{enumerate}
\end{thmA}

\begin{table}[tb]
 \centering
 \caption{Characterizations of classes of matroids without large uniform minors}
 \label{table: equivalent characterization of certain classes of matroids}
 \begin{tabular}{|c|c|l|}
  \hline
  class            & possible factors of $F_M$ & representable over \\
  \hline\hline
  regular          & --               & $\U$ $\Big/$ $\F_2\times P$ with $-1\neq 1$ in $P$ \\           
  \hline
  binary           & $\F_2$           & $\F_2$ \\
  \hline
  ternary          & $\U,\D,\H,\F_3$  & any field extension $k$ of $\F_3$ $\Big/$ $\W$  \\
  \hline
  quaternary       & $\U,\H,\F_2$     & any field extension $k$ of $\F_4$ \\
  \hline
  near-regular     & $\U$             & $\begin{array}{c} \U \ \Big/ \ \F_3\times\F_8 \ \Big/ \ \F_4\times\F_5 \ \Big/ \ \F_4\times\S \ \Big/ \\ \F_8\times\W \ \Big/ \  \D\times\H \end{array}$ \\
  \hline
  dyadic           & $\U,\D$          & $\begin{array}{c} \D \ \Big/ \ \F_3\times\Q \ \Big/ \ \F_3\times\S \ \Big/ \\ \F_3\times\F_q \text{ with } 2\nmid q \text{ and } 3\nmid q-1 \end{array}$ \\
  \hline
  hexagonal        & $\U,\H$          & $\H$ $\Big/$ $\F_3\times\F_4$ $\Big/$ $\F_4\times\W$ \\
  \hline
  $\D\otimes\H$-representable & $\U,\D,\H$ & $\begin{array}{c} \F_3\times\C \ \Big/ \ \F_3\times\P \ \Big/ \\ \F_3\times\F_q \text{ with } 2\nmid q \text{ and } 3\mid q-1 \end{array}$ \\
  \hline
  representable    & $\begin{array}{c}\U,\D,\H,\F_3\\ \text{or }\U,\H,\F_2\end{array}$ & either $\F_3$ or $\F_4$ \\
  \hline
 \end{tabular}
\end{table}

Another consequence of the structure theorem for foundations of matroids without large uniform minors is the following result, which will be the theme of a forthcoming paper.

\begin{thmA}\label{thmF}
 Let $M$ be a ternary matroid. Then up to rescaling equivalence, 
 \begin{enumerate}
  \item every quarternary representation of $M$ lifts uniquely to $\H$;
  \item every quinternary representation of $M$ lifts uniquely to $\D$;
  \item every septernary representation of $M$ lifts uniquely to $\D\otimes\H$;
  \item every octernary representation of $M$ lifts uniquely to $\U$.
 \end{enumerate}
\end{thmA}


\subsection*{Content overview}

In section \ref{section: background}, we introduce embedded minors and review basic facts concerning the Tutte group of a matroid. In section \ref{section: pastures}, we discuss matroid representations over pastures and explain the concept of the universal pasture of a matroid. In section \ref{section: cross ratios}, we extend the concept of cross ratios to matroid representations over pastures and define universal cross ratios. In section \ref{section: foundations}, we introduce the foundation of a matroid and exhibit a complete set of relations between cross ratios, which culminates in Theorem \ref{thmA}. In section \ref{section: structure theorem}, we focus on matroids without large uniform minors and prove Theorem \ref{thmB}. In section \ref{section: applications}, we explain several consequences of Theorem \ref{thmB}, such as Theorems \ref{thmC}, \ref{thmD} and \ref{thmE}.


\subsection*{Acknowledgements}
The authors thank Nathan Bowler and Rudi Pendavingh for helpful discussions; in particular, we thank Rudi Pendavingh for suggesting that a result along the lines of Theorem~\ref{thm: presentation of foundations in terms of bases} should follow from \cite{Gelfand-Rybnikov-Stone95}. The authors also thank their respective muses Camille and Cec\'ilia for inspiring the name of the matroid $C_5$.


\section{Background}
\label{section: background}


\subsection{Notation}
\label{subsection: notation}

In this paper, we assume that the reader is familiar with basic concepts from matroid theory. 

Typically, $M$ denotes a matroid of rank $r$ on the ground set $E=\{1,\dotsc,n\}$. We denote its set of bases by $\cB=\cB_M$ and its set of hyperplanes by $\cH=\cH_M$. We denote the closure of a subset $J$ of $E$ by $\gen{J}$. We denote the dual matroid of $M$ by $M^\ast$.

Given two subsets $I$ and $J$ of $E$, we denote by $I-J=\{i\in I\mid i\notin J\}$ the complement of $J$ in $I$. For an ordered tuple $\bJ=(j_1,\dotsc,j_s)$ in $E^s$, we denote by $|\bJ|$ the subset $\{j_1,\dotsc,j_s\}$ of $E$. Given $k$ elements $e_{1},\dotsc,e_k\in E$, we denote by $\bJ e_{1}\dotsb e_{k}$ the $s+k$-tuple $(j_1,\dotsc,j_{s},e_1,\dotsc,e_k)\in E^{s+k}$. If $J$ is a subset of $E$, then we denote by $Je_{1}\dotsb e_{k}$ the subset $J\cup\{e_1,\dotsc,e_k\}$ of $E$. In particular, we have $|\bJ e_{1}\dotsb e_{k}|=|\bJ|e_{1}\dotsb e_{k}$ for $\bJ\in E^s$.
\todo[inline]{Reminder: we should point out discrepancies to the notation and terminology of \cite{Baker-Lorscheid18}.}


\subsection{The Tutte group}
\label{subsection: the Tutte group}

The Tutte group is an invariant of a matroid that was introduced and studied by Dress and Wenzel in \cite{Dress-Wenzel89}. We will review the parts of this theory that are relevant for the present text in the following.

\begin{df}\label{def: basis Tutte group}
 Let $M$ be a matroid  of rank $r$ on $E$ with Grassmann-Pl\"ucker function $\Delta:E^r\to\K$. The multiplicatively written abelian group $\T_M^\cB$ is generated by symbols $-1$ and $X_\bI$ for every $\bI\in\supp(\Delta)$ modulo the relations
 \[\tag{T1}\label{T1}
  (-1)^2 \ = \ 1;
 \]
 \[\tag{T2}\label{T2}
  X_{(e_{\sigma(1)},\dotsc,e_{\sigma(r)})} = \sign(\sigma) X_{(e_1,\dotsc,e_r)}
 \]
 for every permutation $\sigma\in S_r$, where we consider $\sign(\sigma)$ as an element of $\{\pm 1\}\subset \T_M^\cB$;
 \[\tag{T3}\label{T3}
  X_{\bJ e_1e_3}X_{\bJ e_2e_4}=X_{\bJ e_1e_4}X_{\bJ e_2e_3}
 \]
 for $\bJ=(j_1,\dotsc,j_{r-2})\in E^{r-2}$ and $e_1,\dotsc,e_4\in E$ such that $\bJ e_ie_j\in\supp(\Delta)$ for all $i=1,2$ and $j=3,4$ but $\Delta(\bJ e_1e_2)\Delta(\bJ e_3e_4)=0$.

 The group $\T_M^\cB$ comes with a morphism $\deg:\T_M^\cB\to\Z$ that sends $X_\bI$ to $1$ for every $\bI\in\supp(\Delta)$. The \emph{Tutte group of $M$} is the kernel $\T_M=\ker(\deg)$ of this map.
\end{df}

By definition, the Tutte group $\T_M$ is generated by ratios $X_\bI/X_\bJ$ of generators of $X_\bI$, $X_\bJ$ of $\T_M^\cB$. Since the basis exchange graph of a matroid is connected, it follows that $\T_M$ is generated by elements of the form $X_{\bJ e}/X_{\bJ e'}$, where $\bJ\in E^{r-1}$ and $e,e'\in E$ are such that both $\bJ e$ and $\bJ e'$ are in the support of $\Delta$.

The Tutte group can equivalently be defined in terms of hyperplanes, as explained in the following.

\begin{df}
 Let $M$ be a matroid and $\cH$ its set of hyperplanes. We define $\T_M^\cH$ as the abelian group generated by symbols $-1$ and $X_{H,e}$ for all $H\in\cH$ and $e\in E-H$ modulo the relations
 \[\tag{TH1}\label{TH1}
  (-1)^2 \ = \ 1;
 \]
 \[\tag{TH2}\label{TH2}
  \frac{X_{H_1,e_2}X_{H_2,e_3}X_{H_3,e_1}}{X_{H_1,e_3}X_{H_2,e_1}X_{H_3,e_2}} \ = \ -1,
 \]
 where $H_1,H_2,H_3\in\cH$ are pairwise distinct such that $F=H_1\cap H_2\cap H_3$ is a flat of rank $r-2$ and $e_i\in H_i-F$ for $i=1,2,3$.
 
 This group comes with a map $\deg_\cH:\T_M^\cH\to\Z^\cH$ that sends an element $X_{H,e}$ to the characteristic function $\chi_{H}:\cH\to\Z$ of $\{H\}\subset\cH$, i.e.\ $\chi_H(H')=\delta_{H,H'}$ for $H'\in\cH$. 
\end{df}

The relation between $\T_M$ and $\T_M^\cH$ is explained in \cite[Thms.\ 1.1 and 1.2]{Dress-Wenzel89}, which is as follows.

\begin{thm}[Dress-Wenzel '89]\label{thm: comparison of the basis Tutte group with the hyperplane Tutte group}
 Let $M$ be a matroid and $\cB$ its set of bases. Then the association $-1\mapsto-1$ and $X_{\bJ e}/X_{\bJ e'}\mapsto X_{H,e}/X_{H,e'}$, where $\bJ\in E^{r-1}$, $e,e'\in E$ with $\norm{\bJ e},\norm{\bJ e'}\in\cB$ and $H=\gen{|\bJ|}$, defines an injective group homomorphism $\T_M\to\T_M^\cH$ whose image is $\ker(\deg_\cH)$.
\end{thm}


\subsection{Embedded minors}
\label{subsection: embedded minors}

In this section, we review some basic facts about minors of a matroid and introduce the concept of an embedded minor.

Let $M$ and $N$ be matroids with respective ground sets $E_M$ and $E_N$. An \emph{isomorphism $\varphi:N\to M$ of matroids} is a bijection $\varphi:E_N\to E_M$ such that $B\subset E_N$ is a basis of $N$ if and only if $\varphi(B)$ is a basis of $M$.

\begin{df}
 Let $M$ be a matroid on $E$. A \emph{minor of $M$} is a matroid isomorphic to $M\minor IJ$, where $I$ and $J$ are disjoint subsets of $E$, $M\backslash I$ denotes the deletion of $I$ in $M$ and $M\minor IJ$ denotes the contraction of $J$ in $M\backslash I$. 
\end{df}
 
Note that there are in general different pairs of subsets $(I,J)$ and $(I',J')$ as above that give rise to isomorphic minors $M\minor IJ\simeq M\minor {I'}{J'}$. In particular, \cite[Prop.\ 3.3.6]{Oxley92} shows that there is a co-independent subset $J$ and an independent subset $I$ of $E$ for every minor $N$ of $M$ such that $I\cap J=\emptyset$ and $N\simeq M\minor IJ$. Still, such $I$ and $J$ are in general not uniquely determined by $N$, cf.\ Example \ref{ex: different embeddings of minors}.

If we fix $I$ and $J$ as above, then we can identify the ground set $E_N$ of $N$ with $E-(I\cup J)$, which yields an inclusion $\iota:E_N\to E$. Since $I$ is co-independent and $J$ is independent, the set of bases of $N$ is
\[
 \cB_N \ = \ \big\{ \, B-J \, \big| \, B\in\cB_M \text{ such that } J\subset B\subset E-I \, \big\},
\]
where $\cB_M$ is the set of bases of $M$. Consequently, the difference between the rank $r$ of $M$ and the rank $r_N$ of $N$ is $r-r_N=\# J$. Moreover, the inclusion $E_N\to E$ induces an inclusion
\[
 \begin{array}{cccc}
  \iota: & \cB_N & \longrightarrow & \cB_M \\
           &   B   & \longmapsto     &  B\cup J.
 \end{array}
\]


\begin{df}
 An \emph{embedded minor of $M$} is a minor $N=M\minor IJ$ together with the pair $(I,J)$, where $I$ is a co-independent subset and $J$ is an independent subset $J$ of $E$ such that $I\cap J=\emptyset$. By abuse of notation, we say that \emph{$\iota:N\hookrightarrow M$ is an embedded minor}, where $N=M\minor IJ$ for fixed subsets $I$ and $J$ as above and where $\iota:\cB_N\to \cB_M$ is the induced inclusion of the respective set of bases. 
 
 Let $N'$ be a matroid. Then we say that an embedded minor $\iota:N\hookrightarrow M$ is \emph{of type $N'$}, or is an \emph{embedded $N'$-minor}, if $N$ is isomorphic to $N'$. 
 
 Let $N$ and $M$ be matroids. A \emph{minor embedding of $N$ into $M$} is an isomorphism $N\simeq M\minor IJ$ of $N$ together with an embedded minor $M\minor IJ\hookrightarrow M$ of $M$. 
 
 Given two minor embeddings $\iota:N=M\minor JI\hookrightarrow M$ and $\iota':N'=N\minor{I'}{J'}\to N$, we define the \emph{composition $\iota\circ\iota'$ of $\iota'$ with $\iota$} as the minor embedding $N'=M\minor{(I\cup I')}{(J\cup J')}\hookrightarrow M$. 
\end{df}

\begin{ex}[Embedded minors of type \texorpdfstring{$U^2_4$}{U(2,4)}]
 Let $M$ be a matroid and $\iota:N\to M$ an embedded minor of type $U^2_4$. Let $I$ and $J$ be as above. Then $\# J=r-2$ since the rank of $N$ is $2$, and $E_N=E-(I\cup J)$ has $4$ elements $e_1,\dotsc, e_4$. The set of bases $\cB_N$ of $N$ consists of all $2$-subsets of $E_N$, and thus 
 \[
  \iota(\cB_N) \ = \ \Big\{\, Je_ie_j \, \Big| \, \{i,j\}\subset\{1,\dotsc,4\}\text{ and }i\neq j \, \Big\}.
 \]
\end{ex}

\begin{rem}
 Note that a composition $N'=N\minor{I'}{J'}\hookrightarrow N=M\minor JI\hookrightarrow M$ of minor embeddings induces a composition $\cB_{N'}\to\cB_N\to\cB_M$ of inclusions of sets of bases. On the other hand, a minor embedding $\iota:N=M\minor JI\to M$ decomposes into $\iota=\iota_1\circ\iota_2$ with $\iota_1:N'=M\minor{I_1}{J_1}\to M$ and $\iota_2:N=N'\minor{I_2}{J_2}\to N'$ for every pair of partitions $I=I_1\cup I_2$ and $J=J_1\cup J_2$.

 Note further that it is slightly inaccurate to suppress the subsets $I$ and $J$ from the notation of an embedded minor $\iota:N\to M$ since they are in general not uniquely determined by the isomorphism type of $N$ and the injection $\iota:\cB_N\to\cB_M$, cf.\ Example \ref{ex: different embedded minors with equal subgraphs}. However, there is always a maximal choice for $I$ and $J$ for a given injection $\iota:\cB_N\to\cB_M$.

 More precisely, for two disjoint subsets $I$ and $J$ of $E$ and $\cB=\cB_M$, let $\cB\minor IJ=\{B\in\cB\mid J\subset B\subset E-I\}$. If $\cB\minor IJ$ is not empty, then $I$ is co-independent and $J$ is independent and $\cB\minor IJ$ is the image $\iota(\cB_{M\minor IJ})\subset \cB$ for the embedded minor $M\minor IJ$ of $M$. Tautologically, 
 \[
  I_\max \ = \ E-\bigcup_{B\in \cB\minor IJ} B \qquad \text{and} \qquad J_\max \ = \ \bigcap_{B\in \cB\minor IJ} B
 \]
 are the maximal co-independent and independent subsets of $E$ such that $\cB\minor IJ=\cB\minor {I_\max}{J_\max}=\iota(\cB_{M\minor{I_\max}{J_\max}})$.
\end{rem}




\begin{ex}\label{ex: different embeddings of minors}
 In the following, we illustrate how different choices of disjoint subsets $I$ and $J$ of $E$ lead to different injections $\iota:\cB_{M\minor IJ}\to\cB_M$.
 
 Let $M$ be the matroid on $E=\{1,2,3\}$ whose set of bases is $\cB_M=\big\{\{1,2\},\{1,3\}\big\}$. Let $N=M\backslash\{23\}$ be the restriction of $M$ to $\{1\}$, whose set of bases is $\cB_N=\big\{\{1\}\big\}$. Since there is no canonical map $\cB_N\to\cB_M$, it is clear that not every pair of disjoint subsets $I$ and $J$ leads to an embedding $\cB_{M\minor IJ}\to\cB_M$.
 
 The minor $N$ is isomorphic to both $N_2=M\minor{\{2\}}{\{3\}}$ and $N_3=M\minor{\{3\}}{\{2\}}$, which are embedded minors with respect to the inclusions $\iota_2:\cB_{N_2}\to\cB_M$ with $\iota_2(\{1\})=\{1,2\}$ and $\iota_3:\cB_{N_3}\to\cB_M$ with $\iota_3(\{1\})=\{1,3\}$, respectively.
\end{ex}

\begin{ex}\label{ex: different embedded minors with equal subgraphs}
 The contrary effect to that illustrated in Example \ref{ex: different embeddings of minors} can also happen: different embedded minors can give rise to the same inclusions of sets of bases.
 
 For instance, consider the matroid $M$ on $E=\{1,2\}$ with $\cB_M=\big\{\{1,2\}\big\}$ and the embedded minor $N=M\backslash\{2\}$. Then $\cB_N=\big\{\{1\}\big\}$ and the induced embedding $\iota:\cB_N\to\cB_M$ is a bijection. This is obviously also the case for the trivial minor $N'=M=M\minor{\emptyset}{\emptyset}$. This shows that $N$ is not determined by $\iota:\cB_N\to\cB_M$.
\end{ex}

 


\section{Pastures}
\label{section: pastures}


\subsection{Definition and first properties}
\label{subsection: definition and first properties}

By a \emph{monoid with zero} we mean a multiplicatively written commutative monoid $P$ with an element $0$ that satisfies $0\cdot a=0$ for all $a\in P$. We denote the unit of $P$ by $1$ and write $P^\times$ for the group of invertible elements in $P$. We denote by $\Sym_3(P)$ all elements of the form $a+b+c$ in the monoid semiring 
$\N[P]$, where $a,b,c\in P$.

\begin{df}
 A \emph{pasture} is a monoid $P$ with zero such that $P^\times=P-\{0\}$, together with a subset $N_P$ of $\Sym_3(P)$ such that for all $a,b,c,d\in P$
 \begin{enumerate}[label={(P\arabic*)}]
  \item\label{P1} $a+0+0\in N_P$ if and only if $a=0$,
  \item\label{P2} if $a+b+c\in N_P$, then $ad+bd+cd$ is in $N_P$,
  \item\label{P3} there is a unique element $\epsilon\in P^\times$ such that $1+\epsilon+0\in N_P$.
 \end{enumerate}
\end{df}
We call $N_P$ the \emph{nullset of $P$}, and say that \emph{$a+b+c$ is null}, and write symbolically $a+b+c=0$, if $a+b+c\in N_P$. For $a\in P$, we call $\epsilon a$ the \emph{weak inverse of $a$}.
 
 The element $\epsilon$ plays the role of an additive inverse of $1$, and the relations $a+b+c=0$ express that certain sums of elements are zero, even though the multiplicative monoid $P$ does not carry an addition. For this reason, we will write frequently $-a$ for $\epsilon a$ and $a-b$ for $a+\epsilon b$. In particular, we have $\epsilon=-1$. Moreover, we shall write $a+b=c$ or $c=a+b$ for $a+b+\epsilon c=0$.

\begin{rem}
 As a word of warning, note that $-1$ is not an additive inverse of $1$ if considered as elements in the semiring $\N[P]$, i.e.\ $1-1=1+\epsilon\neq 0$ \emph{as elements of $\N[P]$}. Psychologically, it is better to think of ``$-$'' as an involution on $P$. 
\end{rem}

\begin{df}
 A \emph{morphism of pastures} is a multiplicative map $f:P_1\to P_2$ with $f(0)=0$ and $f(1)=1$ such that $f(a)+f(b)+f(c)=0$ in $N_{P_2}$ whenever $a+b+c=0$ in $N_{P_1}$. This defines the category $\Pastures$.
\end{df}
 
\begin{df}
 A \emph{subpasture} of a pasture $P$ is a submonoid $P'$ of $P$ together with a subset $N_P'\subset\Sym_3(P')$ such that $a^{-1}\in P'$ for every nonzero $a\in P'$ and $a+b+c\in N_{P'}$ for all $a+b+c\in N_P$ with $a,b,c\in P'$. 
 
 Given a subset $S$ of $P^\times$, the \emph{subpasture generated by $S$} is the submonoid $P' = \{0\} \cup \langle S \rangle$, where $\langle S \rangle$ denotes the subgroup of $P^\times$ generated by $S$, together with the nullset $N_{P'} = N_P \cap\Sym_3(P')$.
\end{df}

\begin{lemma}
 Let $P$ be a pasture. Then $a+b=0$ if and only if $b=\epsilon a$. In particular, we have $\epsilon^2=1$. Let $f:P_1\to P_2$ be a morphism of pastures. Then $f(\epsilon)=\epsilon$.
\end{lemma}

\begin{proof}
 Note that $\epsilon$ is uniquely determined by the relation $1+\epsilon+0=0$. By \ref{P2}, this implies that $\epsilon^{-1}+1+0=0$ and thus by \ref{P3}, we conclude that $\epsilon^{-1}=\epsilon$, or equivalently, $\epsilon^2=1$.
 
 Given a morphism $f:P_1\to P_2$ be a morphism of pastures, the null relation $1+\epsilon+0=0$ in $P_1$ yields the relation $f(1)+f(\epsilon)+0=0$ in $P_2$. Thus $f(\epsilon)$ is the weak inverse of $f(1)=1$, which is $\epsilon$.
\end{proof}

\subsubsection{Free algebras and quotients}
\label{subsubsection: algebras and quotients}

Let $P$ be a pasture with null set $N_P$. We define the \emph{free $P$-algebra in $x_1,\dotsc,x_s$} as the pasture $P\gen{x_1,\dotsc,x_s}$ whose unit group is $P\gen{x_1,\dotsc,x_s}^\times=P^\times\times\gen{x_1,\dotsc,x_s}$, where $\gen{x_1,\dotsc,x_s}$ is the free abelian group generated by the symbols $x_1,\dotsc,x_s$, and whose null set is
\[
 N_{P\gen{x_1,\dotsc,x_s}} \ = \ \big\{ da+db+dc \, \big| \, d\in\gen{x_1,\dotsc,x_s}, \, a+b+c\in N_P \big\},
\]
where $da$ stands for $(a,d)\in P\gen{x_1,\dotsc,x_s}^\times$ if $a\neq0$ and for $0$ if $a=0$. This pasture comes with a canonical morphism $P\to P\gen{x_1,\dotsc,x_s}$ of pastures that sends $a$ to $1a$.

Let $S\subset \Sym_3(P)$ be a set of relations of the form $a+b+c$ with $ab\neq 0$. We define the \emph{quotient $\past PS$ of $P$ by $S$} as the following pasture. Let $\tilde N_{\past PS}$ be the smallest subset of $\Sym_3(P)$ that is closed under property \ref{P2} and that contains $N_P$ and $S$. Since all elements $a+b+c$ in $S$ have at least two nonzero terms by assumption, $\tilde N_{\past PS}$ also satisfies \ref{P1}. But it might fail to satisfy \ref{P3}, necessitating the following quotient construction for $P^\times$.

We define the unit group $(\past PS)^\times$ of $\past PS$ as the quotient of the group $P^\times$ by the subgroup generated by all elements $a$ for which $a-1+0\in \tilde N_{\past PS}$. The underlying monoid of $\past PS$ is, by definition, $\{0\}\cup(\past PS)^\times$, and it comes with a surjection $\pi:P\to \past PS$ of monoids. We denote the image of $a\in P$ by $\bar a=\pi(a)$, and define the null set of $\past PS$ as the subset
\[
 N_{\past PS} \ = \ \big\{ \bar a+\bar b+\bar c \, \big| \, a+b+c\in \tilde N_{\past PS} \big\}
\]
of $\Sym_3(\past PS)$. The quotient $\past PS$ of $P$ by $S$ comes with a canonical map $P\to\past PS$ that sends $a$ to $\bar a$ and is a morphism of pastures.


If $S\subset\Sym_3(P\gen{x_1,\dots,x_s})$ is a subset of relations of the form $a+b+c$ with $ab\neq 0$, then the composition of the canonical morphisms for the free algebra and for the quotient yields a canonical morphism 
\[
 \pi: \ P \ \longrightarrow \ P\gen{x_1,\dots,x_s} \ \longrightarrow \ \past{P\gen{x_1,\dotsc,x_s}}S.
\]
We denote by $\pi_0:\{x_1,\dotsc,x_s\}\to \past{P\gen{x_1,\dotsc,x_s}}S$ the map that sends $x_i$ to $\bar x_i$.

The following result describes the universal property of $\past{P\gen{x_1,\dotsc,x_s}}S$, which is analogous to the universal property of the quotient $k[T_1^{\pm 1},\dotsc,T_r^{\pm}]/(S)$ of the algebra of Laurent polynomials over a field $k$ by the ideal $(S)$ generated by a set $S$ of Laurent polynomials (each with only two or three terms). Note that the special case $S=\emptyset$ yields the universal property of the free algebra $P\gen{x_1,\dotsc,x_s}$ and the special case $s=0$ yields the universal property of the quotient $\past PS$.

\begin{prop}\label{prop: universal property of algebras and quotients}
 Let $P$ be a pasture, $s\geq 0$ and $S\subset\Sym_3(P\gen{x_1,\dots,x_s})$ a subset of relations of the form $a+b+c$ with $ab\neq 0$. Let $f:P\to Q$ be a morphism of pastures and $f_0:\{x_1,\dotsc,x_s\}\to Q^\times$ a map with the property that $a\prod x_i^{\alpha_i}+b\prod x_i^{\beta_i}+c\prod x_i^{\gamma_i}\in S$ with $a,b,c\in P$ and $\alpha_i,\beta_i,\gamma_i\in \Z$ for $i=1,\dotsc,r$ implies that
 \[\textstyle
  f(a)\prod f_0(x_i)^{\alpha_i}+ f(b)\prod f_0(x_i)^{\beta_i}+f(c)\prod f_0(x_i)^{\gamma_i}\in N_{Q}.
 \]
 Then there is a unique morphism $\hat f:\past{P\gen{x_1,\dotsc,x_s}}S\to Q$ such that the diagrams
 \[
  \begin{tikzcd}
   P \ar[r,"f"] \ar[d,"\pi"'] & Q \\
   \past{P\gen{x_1,\dotsc,x_s}}S \ar[ur,"\hat f"']
  \end{tikzcd}
  \qquad \text{and} \qquad
  \begin{tikzcd}
   \{x_1,\dotsc,x_s\} \ar[r,"f_0"] \ar[d,"\pi_0"'] & Q \\
   \past{P\gen{x_1,\dotsc,x_s}}S \ar[ur,"\hat f"']
  \end{tikzcd}
 \]
 commute.
\end{prop}

\begin{proof}
 We claim that the association
 \[
  \begin{array}{cccc}
   \hat f: & \past{P\gen{x_1,\dotsc,x_s}}S & \longrightarrow & Q \\
           & a\prod x_i^{\alpha_i}        & \longmapsto     & f(a)\prod f_0(x_i)^{\alpha_i}
  \end{array}
 \]
 is a morphism of pastures. Once we have proven this, it is clear that $f=\hat f\circ\pi$ and $f_0=\hat f\circ\pi_0$. Since the unit group of $\widehat P=\past{P\gen{x_1,\dotsc,x_s}}S$ is generated by $\{ax_i\mid a\in P^\times,i=1,\dotsc,s\}$, it follows that $\hat f$ is uniquely determined by the conditions $f=\hat f\circ\pi$ and $f_0=\hat f\circ\pi_0$. 
 
 We are left with the verification that $\hat f$ is a morphism. As a first step, we show that the restriction $\hat f^\times:\widehat P^\times\to Q^\times$ defines a group homomorphism. Note that $N_{\widehat P}=\{yz+yz'+yz''\mid y\in\widehat P^\times, z+z'+z''\in S\}$. Thus we have an equality $a\prod x_i^{\alpha_i}=b\prod x_i^{\beta_i}$ in $\widehat P^\times$ if and only if $da\prod x_i^{\alpha_i+\delta_i}-db\prod x_i^{\beta_i+\delta_i}\in S$ for some $d\prod x_i^{\delta_i}\in\widehat P^\times$. By our assumptions, we have $f(da)\prod f_0(x_i)^{\alpha_i+\delta_i}-f(db)\prod f_0(x_i)^{\beta_i+\delta_i}\in N_Q$, and thus multiplying with $f(d^{-1})\prod f_0(x_i)^{-\delta_i}$ yields $\hat f(a\prod x_i^{\alpha_i})=\hat f(b\prod x_i^{\beta_i})$. This verifies that $\hat f^\times:\widehat P^\times\to Q^\times$ is well-defined as a map. It is clear from the definition that it is a group homomorphism. 
 
 For showing that $\hat f:\widehat P\to Q$ is a morphism of pastures, we need to verify that for every element $z+z'+z''$ in $N_{\widehat P}$, the element $\hat f(z)+\hat f(z')+\hat f(z'')$ is in $N_Q$. This can be done by a similar argument as before. We omit the details.
\end{proof}

\subsubsection{Examples}
\label{subsubsection: examples}

The \emph{regular partial field} is the pasture $\Funpm=\past{\{0,1,-1\}}{\{1-1\}}$ whose multiplication is determined by $(-1)^2=1$. 

Let $K$ be a field and $K^\bullet$ its multiplicative monoid. Then we can associate with $K$ the pasture $\past{K^\bullet}{\{a+b+c \; | \; a+b+c=0\text{ in }K\}}$. We can recover the addition of $K$ by the rule $-c=a+b$ if $a+b+c=0$. In particular, we can identify the finite field with $2$ elements with the pasture $\F_2=\past\Funpm{\{1+1\}}$, which implies that $-1=1$, and the finite field with $3$ elements with the pasture $\F_3=\past\Funpm{\{1+1+1\}}$.

Let $F$ be a hyperfield and $F^\bullet$ its multiplicative monoid. Then we can associate with $F$ the pasture $\past{F^\bullet}{\{a+b+c \; | \; 0\in a\hyperplus b\hyperplus c\text{ in }F\}}$. In particular, we can realize the \emph{Krasner hyperfield} as $\K=\past\Funpm{\{1+1,1+1+1\}}$, and the \emph{sign hyperfield} as $\S=\past\Funpm{\{1+1-1\}}$.

The \emph{near-regular partial field} is
\[
 \U = \pastgen\Funpm{x,y}{x+y-1}.
\]

The \emph{dyadic partial field} is
\[
 \D \ = \ \pastgen\Funpm{z}{z+z-1}.
\]

The \emph{hexagonal partial field} is
\[
 \H \ = \ \pastgen\Funpm{z}{z^3+1,z-z^2-1}.
\]

It is a straightforward exercise to verify that these descriptions of $\U,\D,\H$ agree with the definitions given in the introduction.

As final examples, the \emph{weak sign hyperfield} is the pasture 
\[
 \W \ = \ \past\Funpm{\gen{1+1+1,1+1-1}}
\]
and the \emph{phase hyperfield} is the pasture $\P$ whose unit group $\P^\times$ is the subgroup of norm $1$-elements in $\C^\times$ and whose null set is
\[
 N_\P \ = \ \Big\{ a+b+c \in\Sym_3(P) \, \Big| \, \gen{a,b,c}_{>0}\text{ is an $\R$-linear subspace of $\C$} \Big\}
\]
where $\gen{a,b,c}_{>0}$ is the smallest cone in $\C$ that contains $a$, $b$ and $c$. In fact, $\P$ is isomorphic to the quotient of the pasture associated with $\C$ by the action of $\R_{>0}$ by multiplication.

\subsubsection{Initial and final objects}
\label{subsubsection: initial and final objects}

The category $\Pastures$ admits both initial and final objects. The initial object of $\Pastures$ is the regular partial field $\Funpm$. Given a pasture $P$, we denote by $i_P$ the unique {\em initial morphism} $i_P : \F_1^\pm \to P$.

The final object of $\Pastures$ is the Krasner hyperfield $\K$. Given a pasture $P$, we denote by $t_P$ the unique {\em terminal morphism} $t_P :  P \to \K$ sending $0$ to $0$ and every nonzero element of $P$ to $1$.

\subsubsection{Products and coproducts}
\label{subsubsection:prodcoprod}

The category $\Pastures$ admits both a product and coproduct.  

Let $P_1,P_2$ be pastures.
The (categorical) {product} $P_1 \times P_2$ can be constructed explicitly as follows.
As sets, we have $P_1 \times P_2 = (P_1^\times \oplus P_2^\times) \cup \{ 0 \}$, 
endowed with the coordinatewise multiplication on $P_1^\times \oplus P_2^\times$, extended by the rule $(a_1,a_2)\cdot 0=0\cdot(a_1,a_2)=0$,
and the nullset is the subset 
\[
 N_{P_1 \times P_2} \ = \ \Big\{(a_1,a_2)+(b_1,b_2)+(c_1,c_2) \, \Big| \, a_i+b_i+c_i \in N_{P_i}\text{ for }i=1,2\Big\}
\]
of $\Sym^3(P_1\times P_2)$.

The categorical coproduct is given by the {\em tensor product} $P_1 \otimes P_2$ defined as follows. As sets, we have $P_1 \otimes P_2 = (P_1 \times P_2) / \sim$, where $P_1 \times P_2$ denotes the Cartesian product (not the underlying set of the product in the category of pastures) and $(x_1,x_2) \sim (y_1,y_2)$ if and only if either:
\begin{itemize}
\item At least one of $x_1,x_2$ is zero and at least one of $y_1,y_2$ is zero; or
\item $x_1 = y_1$ and $x_2 = y_2$; or
\item $x_1 = - y_1$ and $x_2 = -y_2$.
\end{itemize}

Denoting the equivalence class of $(x_1,x_2)$ by $x_1\otimes x_2$,
the additive relations are given by: 
\begin{itemize}
\item $a\otimes y+b\otimes y+c\otimes y \in N_{P_1 \otimes P_2}$ for $y \in P_2$ and $a,b,c \in P_1$ with $a+b+c \in N_{P_1}$.
\item $x\otimes a+x\otimes b+x\otimes c \in N_{P_1 \otimes P_2}$ for $x \in P_1$ and $a,b,c \in P_2$ with $a+b+c \in N_{P_2}$.
\end{itemize}

\begin{lemma}\label{lemma: universal property of the tensor product}
The tensor product of pastures satisfies the universal property of a coproduct with respect to the morphisms $f_1: P_1 \to P_1 \otimes P_2$ and $f_2: P_2 \to P_1 \otimes P_2$ given by 
$x \mapsto x\otimes 1$ and $y \mapsto 1\otimes y$, respectively.
\end{lemma}

\begin{proof}
Given a pasture $P$ and morphisms $g_i : P_i \to P$ for $i=1,2$, we must show that there is a unique morphism $g : P_1 \otimes P_2 \to P$ such that $g_i = g \circ f_i$ for $i=1,2$.

Define $g$ by the formula $g(x_1 \otimes x_2) = g_1(x_1) \cdot g_2(x_2)$. To see that this is well-defined, suppose $(x_1,x_2) \sim (y_1,y_2)$. If $x_1x_2 = 0$ and $y_1y_2=0$, then $g(x_1 \otimes x_2) = g(y_1 \otimes y_2) = 0$. Otherwise $x_i = (-1)^k y_i$ for $i=1,2$ with $k \in \{ 0,1 \}$, and we have 
\[
g(x_1 \otimes x_2) = (-1)^k g_1(x_1) (-1)^k g_2(x_2) = g_1(y_1) g_2(y_2) = g(y_1 \otimes y_2).
\] 
Hence $g$ is well-defined. 

It is straightforward to verify that $g \circ f_i = g_i$ for $i=1,2$ and that $g$ is a morphism.

To see that $g$ is unique, suppose $g'$ is another such morphism. Then $g'(x_1 \otimes 1) = g_1(x_1)$ and $g'(1 \otimes x_2) = g_2(x_2)$, and since $g'$ is a morphism we have
\[
g'(x_1 \otimes x_2) = g'((x_1 \otimes 1)(1 \otimes x_2)) = g'(x_1 \otimes 1) g'(1 \otimes x_2) = g_1(x_1) g_2(x_2)\]
for all $x_1 \in P_1$ and $x_2 \in P_2$. Thus $g' = g$.
\end{proof}

By comparison, the category of fields (which is a full subcategory of $\Pastures$) does not have an initial object, a final object, products, or coproducts.  

\medskip

\begin{ex}\label{ex: product and coproduct of F2 and F3}
 We have ${\mathbb F}_2 \times {\mathbb F}_3 \cong {\mathbb F}_1^\pm$ and ${\mathbb F}_2 \otimes {\mathbb F}_3 \cong {\mathbb K}$.
The first isomorphism follows easily from our formula for the product of two pastures, and the second is an immediate consequence of the following lemma, which in turn follows easily
from the universal property of the coproduct.
\end{ex}

\begin{lemma} \label{lem:tensorrelations}
If $P_2 \cong \past\Funpm{S}$, where $S \subseteq \Sym_3(\Funpm)$, then $P_1 \otimes P_2 \cong  \past{P_1}{S}$.
\end{lemma}

\begin{ex}\label{ex: product and coproduct of F3 and S}
 We have $\F_3\times\S\simeq\past\D{\{z^2\}}$ and $\F_3\otimes\S\simeq\past\Funpm{\{1+1+1,1+1-1\}}$.
 For the first isomorphism, note that the underlying set of $\F_3 \times \S$ is $\left( \{ \pm 1 \} \times \{ \pm 1 \} \right) \cup \{ 0 \}$ while the underlying set of $\past\D{\{z^2\}}$ is $\left( \{ \pm 1 \} \times \{ \pm z \} \right) \cup \{ 0 \}$. One checks easily that the map sending $(1,1)$ to $1$ and $(-1,1)$ to $z$ is an isomorphism of pastures.
 The second isomorphism is a consequence of Lemma~\ref{lem:tensorrelations}.
\end{ex}

\begin{ex} \label{ex:afewmore}
Here (without proof) are a few more examples of products and coproducts:
\begin{itemize}
\item $\Funpm=\F_2\times\S=\F_2\times\W$.
\item $\K=\F_2\otimes\S=\F_2\otimes\W$.
\item $\H=\F_3\times\F_4$.
\end{itemize}
\end{ex}

\begin{rem}
 More generally, one can show that the category $\Pastures$ is complete and co-complete, i.e., it admits all small limits and colimits. In particular, one can form arbitrary fiber products and fiber coproducts in $\Pastures$.
 We omit the details since we will not need these more general statements in the present paper.
\end{rem}

\subsubsection{Comparison with partial fields, hyperfields, fuzzy rings, tracts and ordered blueprints}

The definitions of partial fields, hyperfields, fuzzy rings, tracts and ordered blueprints, and a comparison thereof, can be found in \cite{Baker-Lorscheid18}. We are not aiming at repeating all definitions, but we will explain how the category of pastures fits within the landscape of these types of algebraic objects.

We have already explained how partial fields and hyperfields give rise to pastures. 
The tract associated with a pasture $P$ is defined as $F=(P^\times, N_F)$, where $N_F$ is the ideal generated by $N_P$ in $\N[P^\times]$. The ordered blueprint associated to a pasture $P$ is defined as $B=\past{P}{\{0\leq u+v+w\mid u+v+w\in N_P\}}$.

These associations yield fully faithful embeddings of the category $\PartFields$ of partial fields and the category $\HypFields$ of hyperfields into $\Pastures$, and of $\Pastures$ into the category $\Tracts$ of tracts and into the category $\OBlpr^\pm$ of ordered blueprints with unique weak inverses. This completes the diagram of \cite[Thm. 2.21]{Baker-Lorscheid18} to 
\[
 \beginpgfgraphicnamed{tikz/fig4}
  \begin{tikzcd}[column sep=2cm,row sep=0.2cm]
                          & \PartFields \ar{ddr}\ar[r] & \Pastures \ar[r] & \Tracts \ar[shift left=1.0ex]{dd} \\
   \Fields \ar[ur]\ar[dr] \\
                          & \HypFields \ar{r}\ar[uur,crossing over]   & \FuzzRings \ar[uur]\ar{r} & \OBlpr^\pm \ar[shift left=1.0ex]{uu}[swap]{\vdash} \ar[from=uul,crossing over]   
  \end{tikzcd}
 \endpgfgraphicnamed
\]
where $\FuzzRings$ is the category of fuzzy rings. This diagram commutes and all functors are fully faithful, with exception of the adjunction between $\Tracts$ and $\OBlpr^\pm$. We omit the details of these claims.

Note that fuzzy rings, seen as objects in either $\Tracts$ or $\OBlpr^\pm$, are not pastures in general since the ideal $I$ of the fuzzy ring might not be generated by $3$-term elements of $\N[P^\times]$. Conversely, not every pasture, seen as a tract or as an ordered blueprint, gives rise to a fuzzy ring since the axiom (FR2) (in the notation of \cite[Section 2.4]{Baker-Lorscheid18}) might not be satisfied. An example of a pasture for which (FR2) fails to hold is the pasture $\pastgen\Funpm{z}{z^2+1,1+1+z}$; cf.\ \cite[Ex.\ 2.11]{Baker-Lorscheid18} for more details on this example.


\subsection{Matroid representations}
\label{subsection: matroid representations}

We recall the notion of weak matroids over pastures from \cite{Baker-Bowler19}. 
Let $P$ be a pasture.  
A {\em weak Grassmann--Pl{\"u}cker function} of rank $r$ on $E$ with values in $P$ is a function $\Delta : E^r \to P$ such that:
\begin{enumerate}
\item The set of $r$-element subsets $\{ e_1,\ldots,e_r \} \subseteq E$ such that $\Delta(e_1,\ldots,e_r) \neq 0$ is the set of bases of a matroid $\underline{M}$.
\item $\Delta(\sigma(e_1),\ldots,\sigma(e_r)) = \sign(\sigma) \cdot \Delta(e_1,\ldots,e_r)$ for all permutations $\sigma \in S_r$.
\item $\Delta$ satisfies the {\em 3-term Pl{\"u}cker relations}: for all $\bJ  \in E^{r-2}$ and all $(e_1,e_2,e_3,e_4) \in E^4$, 
\[
 \Delta(\bJ e_1e_2)\cdot \Delta(\bJ e_3e_4) - \Delta(\bJ e_1e_3) \cdot \Delta(\bJ e_2e_4) + \Delta(\bJ e_1e_4) \cdot \Delta(\bJ e_2e_3) \ = \ 0.
\]
\end{enumerate}

Two weak Grassmann--Pl{\"u}cker functions $\Delta, \Delta'$ are {\em isomorphic} if there is a $c\in P^\times$ such that $\Delta'(e_1,\ldots,e_r) = c \Delta(e_1,\ldots,e_r)$ for all $(e_1,\ldots,e_r) \in E^r$.

A {\em weak $P$-matroid} $M$ of rank $r$ on $E$ is an isomorphism class of weak Grassmann--Pl{\"u}cker functions $\Delta : E^r \to P$.

We call $\underline{M}$ the {\em underlying matroid} of $M$, and we refer to $\Delta$ as a {\em $P$-representation} of $\underline{M}$.

We say that a matroid $\underline{M}$ is {\em representable} over a pasture $P$ if there is at least one $P$-representation of $\underline{M}$.

\begin{rem}
In \cite{Baker-Bowler19} one also finds a definition of strong $P$-matroids, but this will not play a role in the present paper. We therefore omit the adjective ``weak'' when talking about $P$-representations.
\end{rem}

\medskip

With this terminology, we introduce the following subclasses of matroids:

\begin{df} \label{df:subclasses}
A matroid $M$ is  
 \begin{itemize}
  \item \emph{regular} if it is representable over $\Funpm$;
  \item \emph{binary} if it is representable over $\F_2$;
  \item \emph{ternary} if it is representable over $\F_3$;
  \item \emph{quaternary} if it is representable over $\F_4$;
  \item \emph{near-regular} if it is representable over $\U$;
  \item \emph{dyadic} if it is representable over $\D$;
  \item \emph{hexagonal} if it is representable over $\H$;
  \item \emph{$\D\otimes\H$-representable}\footnote{In \cite[p.~55]{Pendavingh-vanZwam10b}, the partial field $\D\otimes\H$ is denoted ${\mathbb Y}$.} if it is representable over $\D\otimes\H$;
  \item \emph{representable} if it representable over some field;
  \item \emph{orientable} if it is representable over $\S$.
    \item \emph{weakly orientable} if it is representable over $\W$.
 \end{itemize}
 \end{df}
 
Note that hexagonal matroids are also called $\sqrt[6]{1}$-matroids or sixth-root-of-unity-matroids in the literature, cf.\ \cite{Pendavingh-vanZwam10b} and \cite{Semple-Whittle96b}.
 
\subsection{Matroid representations via hyperplane functions}
\label{subsection: matroid representations via hyperplanes}

There are various ``cryptomorphic'' descriptions of weak $P$-matroids, for example in terms of ``weak $P$-circuits'', cf. \cite{Baker-Bowler19}. 
For the purposes of the present paper, it will be more convenient to reformulate things in terms of {\em hyperplanes} rather than circuits. 

\begin{df} \label{df:P-hyperplanes}
Let $P$ be a pasture and let $\underline{M}$ be a matroid on the finite set $E$. Let $\underline{\cH}$ be the set of hyperplanes of $\underline{M}$.
\begin{enumerate}
\item Given $H \in \underline{\cH}$, we say that $f_H : E \to P$ is a {\em $P$-hyperplane function} for $H$ if $f_H(e)=0$ if and only if $e \in H$.
\item Two $P$-hyperplane functions $f_H,f'_H$ for $H$ are {\em projectively equivalent} if there exists $c \in P^\times$ such that $f_H'(e)=c f_H(e)$ for all $e \in E$.
\item A triple of hyperplanes $(H_1,H_2,H_3) \in \underline{\cH}^3$ is \emph{modular} if $F=H_1\cap H_2\cap H_3$ is a flat of corank $2$ such that $F=H_i\cap H_j$ for all distinct $i,j\in\{1,2,3\}$.
\item A {\em modular system} of $P$-hyperplane functions for $\underline{M}$ is a collection of $P$-hyperplane functions $f_H : E \to P$, one for each $H \in \underline{\cH}$, such that
whenever $H_1,H_2,H_3$ is a modular triple of hyperplanes in $\underline{\cH}$, the corresponding functions $H_i$ are linearly dependent, i.e., there exist constants $c_1,c_2,c_3$ in $P$, not all zero, such that 
\[
c_1 f_{H_1}(e) + c_2 f_{H_2}(e) + c_3 f_{H_3}(e) = 0
\]
for all $e \in E$.
\item Two modular systems of $P$-hyperplane functions $\{ f_H \}$ and $\{ f'_H \}$ are {\em equivalent} if $f_H$ and $f'_H$ are projectively equivalent for all $H \in \underline{\cH}$.
\end{enumerate}
\end{df}

The following result can be viewed as a generalization of ``Tutte's representation theorem'' \cite[Theorem 5.1]{Tutte65} (compare with ~\cite[Theorem 3.5]{Dress-Wenzel92}).
One can also view it as adding to the collection of cryptomorphisms for weak matroids established in \cite{Baker-Bowler19}.

\begin{thm}\label{thm: relation between P-representations and P-hyperplanes}
 Let $P$ be a pasture and let $\underline{M}$ be a matroid of rank $r$ on $E$. Let $\underline{\cH}$ be the set of hyperplanes of $\underline{M}$. There is a canonical bijection 
 \[
  \begin{array}{cccc}
   \Xi: & \Big\{\text{$P$-representations of $\underline{M}$}\Big\} & \longrightarrow & \Big\{\text{modular systems of $P$-hyperplanes for $\underline{M}$}\Big\}.
  \end{array}
 \]
 If $\Delta:E^r\to P$ is a $P$-representation of $\underline M$ and $\cH=\Xi(\Delta)$, then 
 \[
  \frac{f_H(e)}{f_H(e')} \ = \ \frac{\Delta(\bI e)}{\Delta(\bI e')}
 \]
 for every $f_H\in\cH$, elements $e,e'\in E-H$ and $\bI\in E^{r-1}$ such that $|\bI|$ is an independent set which spans $H$.
 \end{thm}

\begin{proof}
 Let $M$ be a weak $P$-matroid with underlying matroid $\underline{M}$. Let $H$ be a hyperplane of $\underline{M}$.  The complement of $H$ in $E$ is a cocircuit $\underline{D}$ of $\underline{M}$; choose a $P$-cocircuit $D$ of $M$ whose support is $\underline{D}$. Now define $f_H : E \to P$ by $f_H(e)=D(e)$. Then $f_H(e)=0$ iff $D(e)=0$ iff $e \not\in \underline{D}$ iff $e \in H$, so $f_H$ is a $P$-hyperplane function for $H$.

 Suppose $H_1,H_2,H_3$ is a modular triple of hyperplanes of $\underline{M}$ with intersection $F$, a flat of corank 2. Let $e$ be an element of $H_3 - F$. Then $e \in H_3 - (H_1 \cup H_2)$ by the covering axiom for flats \cite[Exercise 1.4.11, Axiom (F3)]{Oxley92}. Let $D_1$ and $D_2$ be the $P$-cocircuits of $M$ corresponding to $H_1$ and $H_2$, respectively, and let $\alpha_1 = D_2(e), \alpha_2 = -D_1(e) \in P$. Then $\alpha_1 D_1(e) = -\alpha_2 D_2(e)$, so by \cite[Axiom ${\rm (C3)}^\prime$]{Baker-Bowler19}, there is a $P$-cocircuit $D_3$ of $M$ such that $D_3(e)=0$ and $\alpha_1 D_1(f) + \alpha_2 D_2(f) - D_3(f)=0$ for all $f \in E$. By \cite[Lemma 3.7]{Baker-Bowler19}, the support of $D_3$ is $E-H_3$. By \cite[Axiom (C2)]{Baker-Bowler19}, $D_3$ is a scalar multiple of $f_{H_3}$, say $D_3 = -\alpha_3 f_{H_3}$. Then $\alpha_1 f_{H_1} + \alpha_2 f_{H_2} + \alpha_3 f_{H_3} = 0$, so $\{ f_H \}$ is a modular system of $P$-hyperplane functions for $\underline{M}$.

 Conversely, a similar argument shows that given a modular system of $P$-hyperplane functions $\{ f_H \}$ for $\underline{M}$, there is a corresponding family of $P$-cocircuits $\cD$ defining a weak $P$-matroid $M$. These operations are inverse to one another by construction, and this establishes the desired bijection.

 We turn to the second claim, which is obvious for $e=e'$, so we may assume that $e\neq e'$. Let $n=\#E$ and choose $\bI'\in E^{n-r-1}$ such that $E=|\bI|\cup|\bI'|\cup\{e,e'\}$. Note that since $|\bI e'|$ is a basis of $\underline M$, the complement $|\bI' e|$ is a basis for $\underline{M}^\ast$. If $\bI=(i_1,\dotsc,i_{r-1})$ and $\bI'=(i'_1,\dotsc,i'_{n-r-1})$, we define a total order on $E$ by
 \[
  i'_1 < \dotsb < i'_{n-r-1} < e < i_1 < \dotsb < i_{r-1}< e'.
 \]
 By \cite[Lemma 4.1]{Baker-Bowler19}, there is a dual Grassmann-Pl\"ucker function $\Delta^\ast:E^{n-r}\to P$ to $\Delta$ that satisfies
 \[
  \Delta^\ast(\bI'e) \ = \ \sign(\id_E) \cdot\Delta(\bI e') \ = \ \Delta(\bI e')
 \]
 and
 \[
  \Delta^\ast(\bI'e') \ = \ \sign(\tau_{e,e'}) \cdot\Delta(\bI e) \ = \ -\Delta(\bI e),
 \]
 where $\id_E:E\to E$ is the identity and $\tau_{e,e'}:E\to E$ is the transposition that exchanges $e$ with $e'$. This implies that
 \[
  \frac{f_H(e)}{f_H(e')} \ = \ -\frac{\Delta^\ast(\bI' e')}{\Delta^\ast(\bI' e)} \ = \ \frac{\Delta(\bI e)}{\Delta(\bI e')}
 \]
 as desired, where we use \cite[Def.\ 4.6 and Lemma 4.7]{Baker-Bowler19} for the first equality. 
\end{proof}


\subsection{The universal pasture}
\label{subsection: the universal pasture}

The universal pasture of a matroid was introduced in \cite{Baker-Lorscheid18} as a tool to control the representations of a matroid $M$ over other pastures. We review this in the following.

The symmetric group $S_r$ on $r$ elements acts by permutation of coefficients on $E^r$. In the following, we understand the sign $\sign(\sigma)$ of a permutation $\sigma\in S_r$ as an element of $(\Funpm)^\times=\{\pm 1\}$.

\begin{df}
 Let $M$ be a matroid with Grassmann-Pl\"ucker function $\Delta:E^r\to \K$. The \emph{extended universal pasture of $M$} is the pasture $P_M^+=\pastgen\Funpm{T_\bI|\Delta(\bI)\neq0}{S}$, where $S$ is the set of the relations $T_{\sigma(\bI)}=\sign(\sigma)T_\bI$ for all $\bI\in E^r$ and $\sigma\in S_r$, together with the \emph{$3$-term Pl\"ucker relations}
 \[
  T_{\bJ e_1e_2}T_{\bJ e_3e_4}-T_{\bJ e_1e_3}T_{\bJ e_2e_4}+T_{\bJ e_1e_4}T_{\bJ e_2e_3} \ = \ 0
 \]
 for all $\bJ\in E^{r-2}$ and $e_1,\dotsc,e_4\in E$.
 
 The pasture $P_M^+$ is naturally graded by the rule that $T_\bI$ has degree $1$ for every $\bI\in\supp(\Delta)$. The \emph{universal pasture of $M$} is the subpasture $P_M$ of degree $0$-elements of $P_M^+$.
\end{df}

The relevance of the universal pasture is that it represents the set of isomorphism classes of $P$-representations of $M$. This is derived in \cite{Baker-Lorscheid18} by means of the algebraic geometry of the moduli space of matroids. We include an independent, and more elementary, proof in the following.

\begin{thm}[{\cite[Prop.\ 6.22]{Baker-Lorscheid18}}]\label{thm: universal pasture and representability}
 Let $M$ be a matroid of rank $r$ on $E$ and $P$ a pasture. Then there is a functorial bijection between the set of isomorphism classes of $P$-representations of $M$ and $\Hom(P_M,P)$. In particular, $M$ is representable over $P$ if and only if there is a morphism $P_M\to P$.
\end{thm}

\begin{proof}
 Let $\Delta:E^r\to P$ be a $P$-representation of $M$ and $P_M^+$ the extended universal pasture of $M$. Define the map $\chi^+_{\Delta,0}:T_\bI\mapsto \Delta(\bI)$ from the set $\{T_\bI\mid \bI\in \supp(\Delta)\}$ of generators of $P_M^+$ to $P$. Let $S$ be the set of $3$-term Pl\"ucker relations 
 \[
  T_{\bJ e_1e_2}T_{\bJ e_3e_4} - T_{\bJ e_1e_3}T_{\bJ e_2e_4} + T_{\bJ e_1e_4}T_{\bJ e_2e_3},
 \]
 where $\bJ\in E^{r-2}$ and $e_1,\dotsc,e_4\in E$ such that $|\bJ e_1\dotsc e_4|$ has $r+2$ elements. Applying $\chi^+_{\Delta,0}$ to this relation, with the convention that $\chi^+_{\Delta,0}(T_\bI)=0$ if $\Delta(\bI)=0$, yields
 \begin{multline*}
  \chi^+_{\Delta,0}(T_{\bJ e_1e_2})\chi^+_{\Delta,0}(T_{\bJ e_3e_4}) - \chi^+_{\Delta,0}(T_{\bJ e_1e_3})\chi^+_{\Delta,0}(T_{\bJ e_2e_4}) + \chi^+_{\Delta,0}(T_{\bJ e_1e_4})\chi^+_{\Delta,0}(T_{\bJ e_2e_3}) \\
  = \ \Delta(\bJ e_1e_2)\Delta(\bJ e_3e_4) - \Delta(\bJ e_1e_3)\Delta(\bJ e_2e_4) + \Delta(\bJ e_1e_4)\Delta(\bJ e_2e_3),
 \end{multline*}
 which is an element of $N_P$ since $\Delta$ is a Grassmann-Pl\"ucker function. Thus, by Proposition \ref{prop: universal property of algebras and quotients}, the map $\chi^+_{M,0}$ together with the unique morphism $\Funpm\to P$ define a morphism
 \[
  \chi^+_\Delta: \ P_M^+ = \past{\Funpm\gen{T_\bI\mid\bI\in\supp(\Delta)}}{S} \ \longrightarrow \ P
 \]
 with $\chi^+_\Delta(T_\bI)=\Delta(\bI)$ for $\bI\in\supp(\Delta)$. We define $\chi_\Delta:P_M\to P$ as the composition of the inclusion $P_M\to P_M^+$ with $\chi^+_\Delta$. Since every element of $P_M$ has degree $0$, we have $\chi_\Delta=\chi_{a\Delta}$ for every $a\in P^\times$, which shows that $\chi_\Delta$ depends only on the isomorphism class of $\Delta$. 
 
 This yields a canonical map
 \[
  \begin{array}{ccc}
   \Big\{\text{isomorphism classes of $P$-representations of $M$}\Big\} & \longrightarrow & \Hom(P_M,P),\\
   {}[\Delta]                                                           & \longmapsto     & \chi_\Delta
  \end{array}
 \]
 which turns out to be a bijection whose inverse can be described as follows. Let $\chi:P_M\to P$ be a morphism. Choose an $\bI_0\in E^r$ such that $|\bI_0|$ is a basis of $M$ and define the map
 \[
  \begin{array}{cccl}
   \Delta_\chi: & E^r & \longrightarrow & P, \\
                & \bI & \longmapsto     & \bigg\{\begin{array}{ll}\text{\footnotesize $\chi(T_\bI/T_{\bI_0})$} & \text{\footnotesize if $|\bI|$ is a basis of $M$;} \\ \text{\footnotesize 0} & \text{\footnotesize otherwise.} \end{array}
  \end{array}
 \]
This is a Grassmann-Pl\"ucker function, since
 \begin{multline*}
  \Delta_\chi(\bJ e_1e_2)\Delta_\chi(\bJ e_3e_4) - \Delta_\chi(\bJ e_1e_3)\Delta_\chi(\bJ e_2e_4) + \Delta_\chi(\bJ e_1e_4)\Delta_\chi(\bJ e_2e_3) \\
  \textstyle 
  = \ \chi\bigg(\frac{T_{\bJ e_1e_2}}{T_{\bI_0}}\bigg)\chi\bigg(\frac{T_{\bJ e_3e_4}}{T_{\bI_0}}\bigg) - \chi\bigg(\frac{T_{\bJ e_1e_3}}{T_{\bI_0}}\bigg)\chi\bigg(\frac{T_{\bJ e_2e_4}}{T_{\bI_0}}\bigg) + \chi\bigg(\frac{T_{\bJ e_1e_4}}{T_{\bI_0}}\bigg)\chi\bigg(\frac{T_{\bJ e_2e_3}}{T_{\bI_0}}\bigg)
 \end{multline*}
 is in the nullset of $P_M$. 
 Note that the isomorphism class of $\Delta_\chi$ is independent of the choice of $\bI_0$, since any two such choices yield Grassmann-Pl\"ucker functions that are constant multiples of each other.
 
 It is straightforward to verify that the associations $\chi\mapsto [\Delta_\chi]$ and $[\Delta]\mapsto\chi_\Delta$ are mutually inverse, and that both maps are functorial in $P$; we omit the details.
\end{proof}

\begin{rem}
 We call the morphism $\chi_\Delta:P_M\to P$ associated with the (isomorphism class of a) $P$-representation $\Delta$ the \emph{characteristic morphism}. 

 The proof of Theorem \ref{thm: universal pasture and representability} also shows that the set of $P$-representations of $M$ are in functorial bijection with $\Hom(P_M^+,P)$. Under this identification, the identity morphism $P_M^+\to P_M^+$ defines a $P_M^+$-representation $\widehat\Delta:E^r\to P_M^+$ of $M$, which we call the \emph{universal Grassmann-Pl\"ucker function} of $M$. It satisfies $\widehat\Delta(\bI)=T_\bI$ if $|\bI|$ is a basis of $M$ and $\widehat\Delta(\bI)=0$ otherwise, and $t_{P_M^+}\circ\widehat\Delta:E^r \to\K$ is a Grassmann-Pl\"ucker function for $M$ where $t_{P_M^+}:P_M^+\to \K$ is the terminal morphism, cf.\ section \ref{subsubsection: initial and final objects}.
\end{rem}

\subsection{The Tutte group and the universal pasture}

The connection between the Tutte group and the universal pasture is explained in Theorem 6.26 of \cite{Baker-Lorscheid18}, which is as follows:

\begin{thm}\label{thm: Tutte group as units of the universal pasture}
 Let $M$ be a matroid with Grassmann-Pl\"ucker function $\Delta:E^r\to\K$. The association $-1\mapsto-1$ and $T_\bI\mapsto X_\bI$ for $\bI\in\supp(\Delta)$ defines an isomorphism of groups $(P_M^+)^\times\to\T_M^\cB$ that restricts to an isomorphism $P_M^\times\to\T_M$.
\end{thm}

\begin{rem}\label{rem: matroid representations in fuzzy rings correspond to morphisms from the Tutte group}
 Dress and Wenzel show in \cite[Thm.\ 3.7]{Dress-Wenzel92} that a matroid $M$ is representable over a fuzzy ring $R$ if and only if there is a group homomorphism $\T_M\to R^\times$ that preserves the Pl\"ucker relations. This can be seen as an analogue of Theorem \ref{thm: universal pasture and representability} in the formalism of Dress and Wenzel, but it also lets us explain the advantage of our formulation.
 
 Namely, the foundation of a matroid is an object in the same category $\Pastures$ as the coefficient domains for matroid representations. We can thus use standard arguments from category theory to deduce results about the representability of a matroid. For example, if the foundation of a matroid $M$ is the tensor product $F_1\otimes F_2$ of two pastures $F_1$ and $F_2$, then $M$ is representable over a third pasture $P$ if and only if there exist morphisms $F_1\to P$ and $F_2\to P$. We will make a frequent use of this observation in section \ref{section: applications}.
\end{rem}


\section{Cross ratios}
\label{section: cross ratios}

In this section, we review the theory of cross ratios for matroids from different angles, and explain the connection between these viewpoints, which are derived from cryptomorphic descriptions of a matroid in terms of bases and hyperplanes. There are two principally different types of cross ratios: cross ratios for $P$-matroids, which are elements of $P$, and universal cross ratios of a matroid $M$, which are elements of the universal pasture $P_M$ of $M$. It turns out that there is a close relation between these two types of cross ratios and their different incarnations in terms of bases and hyperplanes. In particular, we identify in a concluding subsection the set of universal cross ratios with the set of fundamental elements in $P_M$.


\subsection{Cross ratios of \texorpdfstring{$P$}{P}-matroids}
\label{subsection: cross ratios of P-matroids}

Let $E=\{1,\dotsc,n\}$ and $0\leq r\leq n$. Let $P$ be a pasture and $M$ a $P$-matroid with Grassmann-Pl\"ucker function $\Delta:E^r\to P$. 

Define $\Omega_M$ to be the set of tuples $(J;e_1,\dotsc,e_4)$ for which there exists a ${\mathbf J}\in E^{r-2}$ with underlying set $|{\mathbf J}|=J$ such that 
\[
 \Delta({\mathbf J}e_1e_4) \ \Delta({\mathbf J}e_2e_3) \ \Delta({\mathbf J}e_1e_3) \ \Delta({\mathbf J}e_2e_4) \ \neq \ 0,
\]
where ${\mathbf J}e_ke_l=(j_1,\dotsc,j_{r-2},e_k,e_l)$. 

\begin{df}
 Let $M$ be a $P$-matroid with Grassmann-Pl\"ucker function $\Delta:E^r\to P$ and $(J;e_1,\dotsc,e_4)\in\Omega_M$. The \emph{cross ratio of $(J;e_1,\dotsc,e_4)$ in $M$} is the element
 \[
  \cross {e_1}{e_2}{e_3}{e_4}{M,J} \ = \ \cross {e_1}{e_2}{e_3}{e_4}{\Delta,\bJ} \ = \ \frac{\Delta({\mathbf J}e_1e_3)\Delta({\mathbf J}e_2e_4)}{\Delta({\mathbf J}e_1e_4)\Delta({\mathbf J}e_2e_3)}
 \]
 of $P$ for any ${\mathbf J}\in E^{r-2}$ with $|{\mathbf J}|=J$ .
 \end{df}

Note that the value of the cross ratio $\cross {e_1}{e_2}{e_3}{e_4}{M,J}$ does not depend on the ordering of ${\mathbf J}$, nor on the choice of Grassmann-Pl\"ucker function $\Delta$ for $M$, which justifies our notation.

We find the following relations between cross ratios with permuted arguments. Let $(J;e_1,\dots,e_4)\in\Omega_M$ and ${\mathbf J} \in E^{r-2}$ be such that $J=|{\mathbf J}|$. We say that $(J;e_1,\dots,e_4)$ is \emph{non-degenerate} if  
\[
 \Delta({\mathbf J}e_1e_2) \Delta({\mathbf J}e_3e_4) \ \neq \ 0,
\]
or equivalently, if $\cross {e_{\sigma(1)}}{e_{\sigma(2)}}{e_{\sigma(3)}}{e_{\sigma(4)}}{M,J}$ is defined and nonzero for every permutation $\sigma$ of $\{1,\dotsc,4\}$. We define $\Omega_M^\octa$ to be the subset of $\Omega_M$ consisting of all non-degenerate $(J;e_1,\dots,e_4)$. We call a cross ratio $\cross {e_1}{e_2}{e_3}{e_4}{M,J}$ \emph{non-degenerate} if $(J;e_1,\dots,e_4)$ is non-degenerate. We call $(J;e_1,\dots,e_4)\in\Omega_M$ \emph{degenerate} if it is not in $\Omega_M^\octa$.

One finds some relations that follow immediately from the definition, such as the fact that permuting rows and columns has no effect on the value of the cross ratio, i.e.\
\[
 \cross {e_1}{e_2}{e_3}{e_4}{M,J} \ = \ \cross {e_2}{e_1}{e_4}{e_3}{M,J} \ = \ \cross {e_3}{e_4}{e_1}{e_2}{M,J} \ = \ \cross {e_4}{e_3}{e_2}{e_1}{M,J};
\]
that permuting the last two entries inverts the cross ratio, i.e.\ 
\[
 \cross {e_1}{e_2}{e_4}{e_3}{M,J} \ = \ \crossinv {e_1}{e_2}{e_3}{e_4}{M,J};
\]
and that a cyclic rotation of the last three entries yields the relation 
\[
 \cross {e_1}{e_2}{e_3}{e_4}{M,J} \cdot \cross {e_1}{e_3}{e_4}{e_2}{M,J}  \cdot \cross {e_1}{e_4}{e_2}{e_3}{M,J} \ = \ -1
\]
if $(J;e_1,\dotsc,e_4)\in\Omega_M^\octa$ is non-degenerate. We will discuss these relations and others in detail in Theorem \ref{thm: presentation of foundations in terms of bases}.

The cross ratios keep track of the Pl\"ucker relations
\begin{equation} \label{eq:PluckerCR}
 \Delta({\mathbf J}e_1e_2)\Delta({\mathbf J}e_3e_4) - \Delta({\mathbf J}e_1e_3)\Delta({\mathbf J}e_2e_4) + \Delta({\mathbf J}e_1e_4)\Delta({\mathbf J}e_2e_3) \ = \ 0
\end{equation}
satisfied by the Grassmann-Pl\"ucker function $\Delta:E^r\to P$. Namely, if $(J;e_1,\dotsc, e_4)\in\Omega_M^\octa$ and ${\mathbf J}\in E^{r-2}$ are such that $J=|{\mathbf J}|$, then dividing both sides of the Pl\"ucker relation \eqref{eq:PluckerCR} by $-\Delta({\mathbf J}e_1e_4)\Delta({\mathbf J}e_2e_3)$ yields the \emph{Pl\"ucker relation for cross ratios}
\[
 \cross {e_1}{e_2}{e_3}{e_4}{M,J} + \cross {e_1}{e_3}{e_2}{e_4}{M,J} \ = \ 1,
\]
where the notation $a+b=c$ in a pasture $P$ is short-hand for $a+b-c \in N_P$.

If $(J;e_1,\dotsc, e_4)\in\Omega_M$ is degenerate, then $\Delta({\mathbf J}e_1e_2)\Delta({\mathbf J}e_3e_4)=0$ and dividing the Pl\"ucker relation by $-\Delta({\mathbf J}e_1e_4)\Delta({\mathbf J}e_2e_3)$ yields $\cross {e_1}{e_2}{e_3}{e_4}{M,J}-1=0$, and thus 
\[
 \cross {e_1}{e_2}{e_3}{e_4}{M,J} \ = \ 1
\]
by the uniqueness of additive inverses in $P$.

\begin{lemma}\label{lemma: cross ratio of the dual matroid}
 Let $P$ be a pasture and $M$ a $P$-matroid of rank $r$ on $E$ with dual $M^\ast$. Let $(J;e_1,\dots,e_4)\in\Omega_{\underline{M}}$ and $I=E-Je_1\dotsc e_4$. Then
 \[
  \cross{e_1}{e_2}{e_3}{e_4}{M^\ast,I} \ = \ \cross{e_1}{e_2}{e_3}{e_4}{M,J} 
 \]
 as elements of $P$.
\end{lemma}

\begin{proof}
 Let $n=\# E$. Choose $\bJ=(j_1,\dotsc,j_{r-2})$ with $|\bJ|=J$ and $\bI=(i_1,\dotsc,i_{n-r-2})$ with $|\bI|=I$. Choose a total order on $E$. Let $\Delta:E^r\to P$ be a Grassmann-Pl\"ucker function for $M$. Then by \cite[Lemma 4.2]{Baker-Bowler19}, there is a Grassmann-Pl\"ucker function $\Delta^\ast:E^{n-r}\to P$ for $M^\ast$ such that for all identifications $\{i,j,k,l\}=\{1,2,3,4\}$, we have
 \[
  \Delta^\ast(\bI e_ie_k) \ = \ \sign(\pi_{i,j,k,l})\cdot \Delta(\bJ e_je_l),
 \]
 where $\pi=\pi_{i,j,k,l}$ is the permutation of $E$ such that
 \[
  \pi(i_1) < \dotsc < \pi(i_{n-r-2}) < \pi(e_i) < \pi(e_k) < \pi(j_1) < \dotsc < \pi(j_{r-2}) < \pi(e_j) < \pi(e_l)
 \]
 in the chosen total order of $E$. Since $\pi_{i,j,l,k}=\pi_{i,j,k,l}\circ\tau_{k,l}$ for the transposition $\tau_{k,l}$ that exchanges $e_k$ and $e_l$, we have $\sign(\pi_{i,j,k,l})/\sign(\pi_{i,j,l,k})=-1$. Thus we obtain
 \begin{multline*}
  \cross{e_1}{e_2}{e_3}{e_4}{M^\ast,I} \ = \ \frac{\Delta^\ast(\bI e_1e_3)\Delta^\ast(\bI e_2e_4)}{\Delta^\ast(\bI e_1e_4)\Delta^\ast(\bI e_2e_3)} \\
  = \ \frac{\sign(\pi_{1,2,3,4})}{\sign(\pi_{1,2,4,3})} \cdot \frac{\sign(\pi_{2,1,4,3})}{\sign(\pi_{2,1,3,4})} \cdot \frac{\Delta(\bJ e_2e_4)\Delta(\bJ e_1e_3)}{\Delta(\bJ e_2e_3)\Delta(\bJ e_1e_4)} \ = \ \cross{e_1}{e_2}{e_3}{e_4}{M,J}
 \end{multline*}
 as claimed.
\end{proof}


\subsection{Cross ratios for hyperplanes}

There is a different, but closely related, notion of cross ratios associated to certain quadruples of hyperplanes.

\begin{df}
 Let $M$ be a matroid of rank $r$ on $E$ and $\cH$ be its set of hyperplanes. A quadruple of hyperplanes $(H_1,\dotsc,H_4)\in\cH^4$ is \emph{modular} if $F=H_1\cap H_2\cap H_3\cap H_4$ is a flat of corank $2$ such that $F=H_i\cap H_j$ for all $i\in\{1,2\}$ and $j\in\{3,4\}$. A modular quadruple $(H_1,\dotsc,H_4)$ is \emph{non-degenerate} if $F=H_i\cap H_j$ for all distinct $i,j\in\{1,\dotsc,4\}$. Otherwise it is called \emph{degenerate}.\footnote{Note that in some papers the term ``modular quadruple'' is used for what we call a non-degenerate quadruple; e.g.\ see \cite{Baker-Bowler19}, \cite[Def.\ 5.1]{Bland-Jensen87} and \cite[Def.\ 3.18]{Pendavingh-vanZwam11}.} We denote the set of all modular quadruples of hyperplanes by $\Theta_M$ and the subset of all non-degenerate modular quadruples by $\Theta_M^\octa$.
\end{df}
 
\begin{df}
 Let $P$ be a pasture and $M$ a $P$-matroid with underlying matroid $\underline M$. Let $(H_1,\dotsc,H_4)\in\Theta_{\underline M}$. The \emph{cross ratio of $(H_1,\dotsc,H_4)$ in $M$} is the element
 \[
  \cross{H_1}{H_2}{H_3}{H_4}{M} \ = \ \frac{f_1(e_3)f_2(e_4)}{f_1(e_4)f_2(e_3)}
 \]
 of $P$, where $f_i:E\to P$ is a $P$-hyperplane function for $H_i$ for $i=1,2$ (cf.\ Definition \ref{df:P-hyperplanes}), and where $e_k\in H_k-F$ for $k=3,4$ with $F=H_1\cap\dotsb\cap H_4$.
\end{df}

Since $f_1$ and $f_2$ are determined by $H_1$ and $H_2$ up to a factor in $P^\times$, the definition of $\cross{H_1}{H_2}{H_3}{H_4}{M}$ is independent of the choices of $f_1$ and $f_2$. It follows from \cite[Theorem 3.21, Lemma 4.5, and Definition 4.6]{Baker-Bowler19} that it is also independent of the choices of $e_3$ and $e_4$.
 
We continue with a comparison of the two notions of cross ratios.

\begin{lemma}\label{lemma: well-definedness and surjectivity of the map Psi}
 Let $M$ be a matroid of rank $r$ on $E$. The association $(J;e_1,\dotsc,e_4)\mapsto (H_1,\dotsc,H_4)$ with $H_i=\gen{Je_i}$ for $i=1,\dotsc,4$ defines a surjective map $\Psi:\Omega_M\to\Theta_M$, which restricts to a surjective map $\Psi^\octa:\Omega_M^\octa\to\Theta_M^\octa$.
\end{lemma}

\begin{proof}
 The flat $F=H_1\cap\dotsb\cap H_4=\gen J$ is of rank $r-2$ since $J$ is an independent set of rank $r-2$. We have $H_i\cap H_j=F$ for all $i=1,2$ and $j=3,4$ since $\Delta(Je_ie_j)\neq 0$ and thus $\gen{H_i\cup H_j}=E$. This shows that $(H_1,\dotsc,H_4)$ is indeed a modular quadruple. By the same reasoning applied to arbitrary distinct $i,j\in\{1,\dotsc,4\}$, we conclude that $\Psi$ restricts to a map $\Psi^\octa:\Omega_M^\octa\to\Theta_\cH^\octa$.
 
 Given $(H_1,\dotsc,H_4)\in\Theta_M$ and $F=H_1\cap\dotsb\cap H_4$, choose an independent subset $J\subset F$ with $r-2$ elements and $e_i\in H_i-F$ for $i=1,\dotsc,4$. Since $H_i\cap H_k=F$ for $i\in\{1,2\}$ and $k\in\{3,4\}$, the closure of $Je_ie_k$ is $E$, i.e.\ $Je_ie_k$ is a basis of $M$. Thus $(J;e_1,\dotsc,e_4)\in\Omega_M$ and $\Psi(J;e_1,\dotsc,e_4)=(H_1,\dotsc,H_4)$, which establishes the surjectivity of $\Psi$. If $(H_1,\dotsc,H_4)\in\Theta_M^\octa$, then $H_i\cap H_k=F$ and thus $Je_ie_k$ is a basis of $M$ for all distinct $i,k\in\{1,\dotsc,4\}$. Thus $(J;e_1,\dotsc,e_4)\in\Omega_M^\octa$ and $\Psi^\octa(J;e_1,\dotsc,e_4)=(H_1,\dotsc,H_4)$, which establishes the surjectivity of $\Psi^\octa$. 
\end{proof}

\begin{prop}\label{prop: comparison of hyperplane cross ratio with basis cross ratios}
 Let $P$ be a pasture and $M$ a $P$-matroid with underlying matroid $\underline M$. Let $(J;e_1,\dotsc,e_4)\in\Omega_{\underline M}$ and $(H_1,\dotsc,H_4)=\Psi(J;e_1,\dotsc,e_4)$. Then we have 
 \[
  \cross{H_1}{H_2}{H_3}{H_4}{M} \ = \ \cross{e_1}{e_2}{e_3}{e_4}{M,J}
 \]
 as elements of $P$.
\end{prop}

\begin{proof}
 Since $|\bJ e_i|$ is an $(r-1)$-set that generates $H_i$ and $e_j\notin H_i$ for $i\in\{1,2\}$ and $j\in\{3,4\}$, we can apply Theorem \ref{thm: relation between P-representations and P-hyperplanes} to conclude that
 \[
  \cross{H_1}{H_2}{H_3}{H_4}{M} \ = \ \frac{f_1(e_3)f_2(e_4)}{f_1(e_4)f_2(e_3)} \ = \ \frac{\Delta(\bJ e_1e_3)\Delta(\bJ e_2e_4)}{\Delta(\bJ e_1e_4)\Delta(\bJ e_2e_3)} \ = \ \cross{e_1}{e_2}{e_3}{e_4}{M,J}
 \]
 as claimed. 
\end{proof}

Our comparison of different notions of cross ratios has the following immediate consequence.

\begin{cor}\label{cor: equality of universal cross ratios with the same associated hyperplane cross ratio}
 Let $M$ be a matroid and $(J;e_1,\dotsc,e_4),(J';f_1,\dotsc,f_4)\in\Omega_M$. If $\gen{Je_i}=\gen{J'f_i}$ for $i=1,\dotsc,4$, then $\cross{e_1}{e_2}{e_3}{e_4}{J}=\cross{f_1}{f_2}{f_3}{f_4}{J'}$.
\end{cor}

\begin{proof}
 By Proposition \ref{prop: comparison of hyperplane cross ratio with basis cross ratios}, we have $\cross{e_1}{e_2}{e_3}{e_4}{J}=\cross{H_1}{H_2}{H_3}{H_4}{}=\cross{f_1}{f_2}{f_3}{f_4}{J'}$ if $H_i=\gen{Je_i}=\gen{J'f_i}$ for $i=1,\dotsc,4$.
\end{proof}


\subsection{Universal cross ratios}
\label{subsection: universal cross ratios}

Let $M$ be a matroid of rank $r$ on $E=\{1,\dotsc,n\}$ with Grassmann-Pl\"ucker function $\Delta:E^r\to \K$. 

Recall from section \ref{subsection: the universal pasture} the definition of the extended universal pasture 
\[
 P_M^+ \ = \ \pastgen\Funpm{T_\bI|\Delta(\bI)\neq 0}{S}
\]
of $M$, where $S$ contains the relations $T_{\sigma(\bI)}=\sign(\sigma)T_\bI$ and the $3$-term Pl\"ucker relations 
\[
 T_{\bJ e_1e_2}T_{\bJ e_3e_4}-T_{\bJ e_1e_3}T_{\bJ e_2e_4}+T_{\bJ e_1e_4}T_{\bJ e_2e_3} \ = \ 0
\]
for all $\bJ\in E^{r-2}$ and $e_1,\dotsc,e_4\in E$, where we use the convention $T_{\bI}=0$ if $\Delta(\bI)=0$. The universal Grassmann-Pl\"ucker function $\widehat\Delta:E^r\to P_M^+$ for $M$ sends $\bI\in E^r$ to $T_\bI$ if $|\bI|$ is a basis of $M$, and to $0$ otherwise. The universal $P_M$-matroid $\widehat M$ for $M$ is defined by the Grassmann-Pl\"ucker function $T_\bI^{-1}\widehat\Delta:E^r\to P_M$, where $\bI\in E^r$ is any $r$-tuple with $\Delta(\bI)\neq 0$.

\begin{df}
 Let $M$ be a matroid with universal $P_M$-matroid $\widehat M$. Let $(J;e_1,\dots,e_4)\in\Omega_M$ and $(H_1,\dots,H_4)\in\Theta_M$. The {\em universal cross ratio of $(J;e_1,\dots,e_4)$} is the element
 \[
  \cross{e_1}{e_2}{e_3}{e_4}{J} \ := \ \cross{e_1}{e_2}{e_3}{e_4}{\widehat M,J}
 \]
  of $P_M$, and the \emph{universal cross ratio of $(H_1,\dotsc,H_4)$} is the element
 \[
  \cross{H_1}{H_2}{H_3}{H_4}{} \ := \ \cross{H_1}{H_2}{H_3}{H_4}{\widehat M}
 \]
 of $P_M$.
\end{df}

The relation between cross ratios of a $P$-matroid and the universal cross ratio of the underlying matroid $\underline{M}$ is explained in the following statement.

\begin{prop}\label{prop: cross ratios as images of universal cross ratios}
 Let $P$ be a pasture and $M$ a $P$-matroid with Grassmann Pl\"ucker function $\Delta:E^r\to P$. Let $\underline M$ be the underlying matroid and $P_{\underline M}$ its universal pasture. Let $\chi_M:P_{\underline M}\to P$ be the universal morphism associated with $M$, which maps $T_\bI/T_{\bI'}$ to $\Delta(\bI)/\Delta(\bI')$. Then 
 \[
  \chi_M\bigg( \cross{e_1}{e_2}{e_3}{e_4}{J} \bigg) \ = \ \cross{e_1}{e_2}{e_3}{e_4}{M,J}
 \]
 as elements of $P$ for every $(J;e_1,\dotsc,e_4)\in\Omega_{\underline M}$.
\end{prop}

\begin{proof}
 This follows directly from the definitions of $\chi_M$, $\widehat\Delta$ and the (universal) cross ratios.
\end{proof}


\subsection{Fundamental elements}
\label{subsection: fundamental elements}

Universal cross ratios can be characterized intrinsically as the fundamental elements of the universal pasture of a matroid. To the best of our knowledge, the importance of fundamental elements in the study of matroid representations goes back to Semple's paper \cite{Semple97}, where this concept was introduced in the context of partial fields. We extend the notion of fundamental elements to pastures and explain its relation to universal cross ratios in the following.

The property of cross ratios that lead to the definition of fundamental elements are the $3$-term Pl\"ucker relations
\[
 \Delta(\bJ e_1e_2)\Delta(\bJ e_3e_4) - \Delta(\bJ e_1e_3)\Delta(\bJ e_2e_4) + \Delta(\bJ e_1e_4)\Delta(\bJ e_2e_3) \ = \ 0
\]
for a Grassmann-Pl\"ucker function $\Delta:E^r\to P$, where $\bJ\in E^{r-2}$ and $e_1,\dotsc,e_4\in E$. If $\Delta(\bJ e_ie_j)\neq 0$ for all distinct $i,j\in\{1,\dotsc,4\}$, then division by $-\Delta(\bJ e_1e_4)\Delta(\bJ e_2e_3)$ yields
\[
 \cross{e_1}{e_2}{e_3}{e_4}{\Delta,\bJ} + \cross{e_1}{e_3}{e_2}{e_4}{\Delta,\bJ} \ = \ \frac{\Delta(\bJ e_1e_3)\Delta(\bJ e_2e_4)}{\Delta(\bJ e_1e_4)\Delta(\bJ e_2e_3)} + \frac{\Delta(\bJ e_1e_2)\Delta(\bJ e_3e_4)}{\Delta(\bJ e_1e_4)\Delta(\bJ e_3e_2)} \ = \ 1
\]
for the non-degenerate cross ratios $\cross{e_1}{e_2}{e_3}{e_4}{\Delta,\bJ}$ and $\cross{e_1}{e_3}{e_2}{e_4}{\Delta,\bJ}$ in $P^\times$.

\begin{df}
 Let $P$ be a pasture. A \emph{fundamental element of $P$} is an element $z\in P^\times$ such that $z+z'=1$ for some $z'\in P^\times$.
\end{df}

\begin{prop}\label{prop: universal cross ratios are precisely the fundamental elements of the foundation}
 Let $M$ be a matroid. For an element $z\in P_M$, the following are equivalent:
 \begin{enumerate}
  \item\label{fund1} $z$ is a fundamental element of $P_M$;
  \item\label{fund2} $z= \cross{e_1}{e_2}{e_3}{e_4}{J}$ for some $(J;e_1,\dotsc,e_4)\in\Omega_M^\octa$; 
  \item\label{fund3} $z=\cross{H_1}{H_2}{H_3}{H_4}{}$ for some $(H_1,\dotsc,H_4)\in\Theta_M^\octa$.
 \end{enumerate}
\end{prop}

\begin{proof}
 Our preceding discussion shows that $\cross{e_1}{e_2}{e_3}{e_4}{J}+\cross{e_1}{e_3}{e_2}{e_4}{J}=1$ for $(J;e_1,\dotsc,e_4)\in\Omega_M^\octa$. Thus \eqref{fund2}\implies\eqref{fund1}. The equivalence of \eqref{fund2} and \eqref{fund3} follows from Proposition \ref{prop: comparison of hyperplane cross ratio with basis cross ratios}.

 We are left with \eqref{fund1}\implies\eqref{fund2}. Assume that $z\in P_M^\times$ is a fundamental element, i.e.\ $z+z'-1=0$ for some $z'\in P_M^\times$. Since the nullset of the extended universal pasture $P_M^+$ is generated by the $3$-terms Pl\"ucker relations, there must be an element $a\in (P_M^+)^\times$ such that $az+az'-a=0$ is of the form
 \[
  T_{\bJ e_1e_2}T_{\bJ e_3e_4} - T_{\bJ e_1e_3}T_{\bJ e_2e_4} + T_{\bJ e_1e_4}T_{\bJ e_2e_3} \ = \ 0
 \]
 for some $\bJ\in E^{r-2}$ and $e_1,\dotsc,e_4\in E$ such that $|\bJ e_ie_j|$ is a basis of $M$ for all distinct $i,j\in\{1,\dotsc,4\}$, i.e.\ $(J;e_1,\dotsc,e_4)\in\Omega_M^\octa$ where $J=|\bJ|$. After a suitable permutation of $e_1,\dotsc,e_4$, we can assume that $-a=T_{\bJ e_1e_4}T_{\bJ e_2e_3}=-a$ and $az=-T_{\bJ e_1e_3}T_{\bJ e_2e_4}$. Thus 
 \[
  z \ = \ \frac{-az}{-a} \ = \ \frac{T_{\bJ e_1e_3}T_{\bJ e_2e_4}}{T_{\bJ e_1e_4}T_{\bJ e_2e_3}} \ = \ \cross{e_1}{e_2}{e_3}{e_4}{J}  
 \]
 is a cross ratio, as claimed.
\end{proof}


\subsection{Compatibility with the Tutte group formulation of Dress and Wenzel}

We provide a comparison of the different types of universal cross ratios, as introduced above, with the cross ratios introduced by Dress and Wenzel in \cite[Def.\ 2.3]{Dress-Wenzel90}.

The image of a universal cross ratio $\cross{e_1}{e_2}{e_3}{e_4}{J}$ under the isomorphism $P_M^\times\to\T_M$ from Theorem \ref{thm: Tutte group as units of the universal pasture} appears implicitly already in \cite[Prop.\ 2.2]{Dress-Wenzel89}, and is as follows.

\begin{lemma}\label{lemma: universal basis cross ratios in the Tutte group}
 Let $M$ be a matroid with Grassmann-Pl\"ucker function $\Delta:E^r\to\K$, Tutte group $\T_M$ and universal pasture $P_M$. Let $\varphi:P_M^\times\to\T_M$ be the isomorphism of groups that sends $T_\bI/T_{\bI'}$ to $X_\bI/X_{\bI'}$ for $\bI,\bI'\in\supp(\Delta)$. Then
 \[
  \varphi\bigg(  \cross{e_1}{e_2}{e_3}{e_4}{J}\bigg) \ = \ \frac{X_{\bJ e_1e_3}X_{\bJ e_2e_4}}{X_{\bJ e_1e_4}X_{\bJ e_2e_3}}
 \]
 for all $(J;e_1,\dotsc,e_4)\in\Omega_M$ and $\bJ\in E^{r-2}$ with $|\bJ|=J$.
\end{lemma}

\begin{proof}
 Note that the ratio $\big(X_{\bJ e_1e_3}X_{\bJ e_2e_4}\big)\big({X_{\bJ e_1e_4}X_{\bJ e_2e_3}}\big)^{-1}$ does not depend on the ordering of $\bJ$. The rest follows immediately from the definitions.
\end{proof}
 
Let $(H_1,\dotsc,H_4)$ be a modular quadruple of hyperplanes of $M$ and $F$ the corank $2$ flat contained in all $H_i$. Let $e_3\in H_3-F$ and $e_4\in H_4-F$. The {\em Dress--Wenzel universal cross ratio} of $(H_1,\dotsc,H_4)$ is the element
\[
 \cross{H_1}{H_2}{H_3}{H_4}{\T} \ := \ \frac{X_{H_1,e_3}X_{H_2,e_4}}{X_{H_2,e_3}X_{H_1,e_4}}
\]
of the group $\T_M^\cH$.

As shown in \cite[Lemma 2.1]{Dress-Wenzel90}, this definition is independent of the choices of $e_3$ and $e_4$. Since $\deg_\cH\big(\cross{H_1}{H_2}{H_3}{H_4}{\cH}\big)=0$, it follows from Theorem \ref{thm: comparison of the basis Tutte group with the hyperplane Tutte group} that $\cross{H_1}{H_2}{H_3}{H_4}{\cH}$ is contained in the image of the injection $\iota:\T_M\to\T_M^\cH$. 

\begin{lemma}\label{lemma: comparision of hyperplane cross ratios of the foundation and of the Tutte group}
 Let $\psi:P_M^\times\to\T_M^\cH$ be the group homomorphism that maps $T_{\bI e}T_{\bI e'}^{-1}$ to $X_{H,e}X_{H,e'}^{-1}$ where $\bI\in E^{r-1}$, $e,e'\in E$, $I = |\bI|$, $H=\gen{I}$, and $Ie,Ie'$ are bases of $M$. Let $(H_1,\dotsc,H_4)\in\Theta_M$ be a modular quadruple of hyperplanes of $M$. Then
 \[
  \psi\bigg(\cross{H_1}{H_2}{H_3}{H_4}{}\bigg) \ = \ \cross{H_1}{H_2}{H_3}{H_4}{\T}.
 \]
\end{lemma}

\begin{proof}
 It is clear from the definitions that $\psi=\iota\circ\varphi$. By Lemma \ref{lemma: well-definedness and surjectivity of the map Psi}, there is an element $(J;e_1,\dotsc,e_4)\in\Omega_M$ with $\Psi(J;e_1,\dotsc,e_4)=(H_1,\dotsc,H_4)$, i.e.\ $H_i=\gen{Je_i}$ for $i=1,\dotsc,4$. Using Proposition \ref{prop: comparison of hyperplane cross ratio with basis cross ratios}, we obtain
 \[
  \psi\bigg(\cross{H_1}{H_2}{H_3}{H_4}{}\bigg) \, = \, \iota\circ\varphi\bigg( \cross{e_1}{e_2}{e_3}{e_4}{J} \bigg) \, = \, \iota\Bigg( \frac{X_{\bJ e_1e_3}X_{\bJ e_2e_4}}{X_{\bJ e_1e_4}X_{\bJ e_2e_3}} \Bigg) \, = \, \frac{X_{H_1,e_3}X_{H_2,e_4}}{X_{H_1,e_4}X_{H_2,e_3}} \, = \, \cross{H_1}{H_2}{H_3}{H_4}{\T}
 \]
 as claimed.
\end{proof}


\section{Foundations}
\label{section: foundations}

The foundation $F_M$ of a matroid $M$ is the subpasture of degree $0$-elements of the universal pasture $P_M$, and it represents the functor taking a pasture $P$ to the set of $P$-rescaling classes of $M$. In particular, just as with $P_M$, the foundation can detect whether or not a matroid is representable over a given pasture $P$ in terms of the existence of a morphism from $F_M$ to $P$.

One advantage of the foundation over the universal pasture is that, because of some deep theorems due to Tutte, Dress--Wenzel, and Gelfand--Rybnikov--Stone, there is an explicit presentation of $F_M$ in terms of generators and relations in which the relations are all inherited from ``small'' embedded minors. More precisely, the foundation of $M$ is generated by the universal cross ratios of $M$, and all relations between these cross ratios are generated by a small list of relations stemming from embedded minors of $M$ having at most $7$ elements.

We begin our discussion of foundations by reviewing some facts which were proved in the authors' previous paper \cite{Baker-Lorscheid18}. Next we explain the role of embedded minors in the study of foundations. We then exhibit, through very explicit computations, the relations between universal cross ratios inherited from small minors which enter into the presentation by generators and relations alluded to above. Finally, we use the aforementioned result of Gelfand, Rybnikov and Stone to prove that these relations generate {\em all} relations in $F_M$ between universal cross ratios.


\subsection{Definition and basic facts}
\label{subsection: definition and basic facts}

Let $M$ be a matroid of rank $r$ on $E$ with extended universal pasture $P_M^+$. For a subset $I$ of $E$, let $\delta_I:E\to\Z$ be the characteristic function of $I$, which is an element of $\Z^E$. The \emph{multidegree} is the group homomorphism 
\[
 \begin{array}{cccc}
  \deg_E: & (P_M^+)^\times & \longrightarrow & \Z^E \\
          &  T_\bI           & \longmapsto     & \delta_I,
           \end{array} 
\]
where $I = |\bI|$.
It is easily verified that this map is well-defined, cf.\ \cite[section 7.3]{Baker-Lorscheid18}. The \emph{degree in $i$} is the function $\deg_i:(P_M^+)^\times\to\Z$ that is the composition of $\deg_E:(P_M^+)^\times\to\Z^E$ with the canonical projection to the $i$-th component, i.e.\ $\deg_i(T_\bI)=1$ if $i\in I$ and $\deg_i(T_\bI)=0$ if $i\notin I$. The \emph{total degree} is the function $\deg:(P_M^+)^\times\to\Z$ that is the sum over $\deg_i$ for all $i\in E$, i.e.\ $\deg(T_\bI)=\sum_{i\in E}\deg_i(T_\bI)=\# I=r$. 

\begin{df}
 Let $M$ be a matroid with extended universal pasture $P_M^+$. The \emph{foundation of $M$} is the subpasture $F_M$ of $P_M^+$ that consists of $0$ and all elements of multidegree $0$.
\end{df}

Note that the universal pasture $P_M$ of $M$ is the subpasture of $P_M^+$ that is generated by all units of total degree $0$. Since $\deg(x)=0$ if $\deg_E(x)=0$, the foundation $F_M$ of $M$ is a subpasture of $P_M$.

The relevance of the foundation of $M$ is the fact that it represents the rescaling class space 
\[
 \cX^R_M(P) \ = \ \Big\{ \text{rescaling classes of $M$ over $P$} \Big\}
\]
considered as a functor in $P$.

\begin{thm}[{\cite[Cor.\ 7.26]{Baker-Lorscheid18}}]\label{thm: foundation and representability}
 Let $M$ be a matroid and $P$ a pasture. Then there is a functorial bijection $\cX^R_M(P)=\Hom(F_M,P)$. In particular, $M$ is representable over $P$ if and only if there is a morphism $F_M\to P$.
\end{thm}

Recall from \cite{Dress-Wenzel89} that the inner Tutte group $\T^{(0)}_M$ of a matroid $M$ is defined as the subgroup of the Tutte group $\T_M$ of $M$ that consists of all elements of multidegree $0$, where the multidegree $\deg:\T_M\to\Z^E$ is defined in the same way as the multidegree $\deg:P_M\to\Z^E$. This yields at once the following consequence of Theorem \ref{thm: Tutte group as units of the universal pasture} (cf.\ \cite[Cor.\ 7.11]{Baker-Lorscheid18}).

\begin{cor}\label{cor: inner Tutte group as untis of the foundation}
 The canonical isomorphism $\P_M^\times\to\T_M$ restricts to an isomorphism $F_M^\times\to\T^{(0)}_M$.
\end{cor}

\begin{rem}
 Wenzel observes in \cite[Thm.\ 6.3]{Wenzel91} that a matroid representation over a fuzzy ring $K$ induces a group homomorphism $\T^{(0)}_M\to K^\times$, and that this homomorphism detects the rescaling class of a representation. This can be seen as a partial analogue of Theorem \ref{thm: foundation and representability} for fuzzy rings (cf.\ Remark \ref{rem: matroid representations in fuzzy rings correspond to morphisms from the Tutte group}).
\end{rem}

\subsection{Universal cross ratios as generators of the foundation}
\label{subsection: cross ratios as generators of the foundation}

Let $M$ be a matroid of rank $r$ on $E$ and $P_M^+$ its extended universal pasture. The simplest type of elements of $P_M^+$ with multidegree $0$ are universal cross ratios 
\[
 \cross{e_1}{e_2}{e_3}{e_4}{J} \ = \ \frac{T_{\bJ e_1e_3}T_{\bJ e_2e_4}}{T_{\bJ e_1e_4}T_{\bJ e_2e_3}}
\]
where $(J;e_1,\dotsc,e_4)\in\Omega_M$ and $\bJ\in E^{r-2}$ such that $\norm\bJ=J$. This formula shows that the universal cross ratios are elements of the foundation $F_M$ of $M$. It is proven in \cite[Cor.\ 7.11]{Baker-Lorscheid18} that the foundation is generated by the universal cross ratios. To summarize, we have:

\begin{thm}\label{thm: foundation is generated by cross ratios}
 Let $M$ be a matroid. Then $\cross {e_1}{e_2}{e_3}{e_4}{J}\in F_M^\times$ for every $(J;e_1,\dotsc,e_4)\in\Omega_M$, and $F_M^\times$ is generated by the collection of all such universal cross ratios.
\end{thm}

Using Proposition~\ref{prop: comparison of hyperplane cross ratio with basis cross ratios}, we obtain:

\begin{cor} \label{cor:hypercrossfound}
 Let $M$ be a matroid. Then $\cross {H_1}{H_2}{H_3}{H_4}{}\in F_M^\times$ for every $(H_1,\ldots,H_4)\in\Theta_M$, and $F_M^\times$ is generated by the collection of all such hyperplane universal cross ratios.
\end{cor}

\subsection{The foundation of the dual matroid}
\label{subsection: foundation of the dual matroid}

Let $M$ be a matroid of rank $r$ on $E$ and $P_M$ its universal pasture. By definition the identity morphism $\id:P_M\to P_M$ is the characteristic morphism of the universal $P_M$-matroid $\widehat M$; cf.\ Theorem \ref{thm: universal pasture and representability}. The underlying matroid of $\widehat M$ is $\underline{\widehat M}=M$. The underlying matroid of the dual $P_M$-matroid $\widehat M^\ast$ of $\widehat M$ is the dual $\underline{\widehat M^\ast}=M^\ast$ of $M$, cf.\ \cite[Thm.\ 3.24]{Baker-Bowler19}. Let $\omega_M:P_{M^\ast}\to P_M$ be the characteristic morphism of $\widehat M^\ast$.

\begin{prop}\label{prop: foundation of the dual matroid}
 Let $M$ be a matroid of rank $r$ on $E$. Then $\omega_M:P_{M^\ast}\to P_M$ is an isomorphism of pastures that restricts to an isomorphism $F_{M^\ast}\to F_M$ between the respective foundations of $M^\ast$ and $M$. Let $n=\#E$. For every $\bI\in E^{n-r-1}$, $\bJ\in E^{r-1}$ and $e,f\in E$ such that $E=|\bI|\cup |\bJ|\cup\{e,f\}$, we have
 \[
  \omega_M\bigg(\frac{T_{\bI e}}{T_{\bI f}}\bigg) \ = \ - \frac{T_{\bJ f}}{T_{\bJ e}},
 \]
 and for every $(J;e_1,\dotsc,e_4)\in\Omega_M$ and $I=E-Je_1\dotsc e_4$, we have $(I;e_1,\dotsc,e_4)\in\Omega_{M^\ast}$ and
 \[
  \omega_M\bigg(\cross{e_1}{e_2}{e_3}{e_4}{\widehat{M^\ast},I}\bigg) \ = \ \cross{e_1}{e_2}{e_3}{e_4}{\widehat M,J},
 \]
 where $\widehat M$ is the universal $P_M$-matroid of $M$ and $\widehat{M^\ast}$ is the universal $P_{M^\ast}$-matroid of $M^\ast$.
\end{prop}

\begin{proof}
 The construction of $\omega_M$, applied to $M^\ast$ in place of $M$, yields a morphism $\omega_{M^\ast}:P_{M^{\ast\ast}}\to P_{M^\ast}$. Since $M^{\ast\ast}=M$, we have $P_{M^{\ast\ast}}=P_M$. The composition $\omega_M\circ\omega_{M^\ast}:P_M=P_{M^{\ast\ast}}\to P_{M^\ast}\to P_M$ is the characteristic morphism of the double dual $\widehat M^{\ast\ast}$ of $\widehat M$, which is equal to $\widehat M$ by \cite[Thm.\ 3.24]{Baker-Bowler19}, and thus $\omega_M\circ\omega_{M^{\ast}}$ is the identity of $P_M$. Similarly, the composition $\omega_{M^\ast}\circ\omega_M$ is the identity of $P_{M^\ast}$. This shows that $\omega_M$ and $\omega_{M^\ast}$ are mutually inverse isomorphisms.
 
 Let $\Delta:E^r\to P_M$ be a Grassmann-Pl\"ucker function for $\widehat M$. Endow $E$ with a total order and define $\sign(i_1,\dotsc,i_n)=\sign(\pi)$ as the sign of the permutation $\pi$ of $E$ such that $\pi(i_1)<\dotsb <\pi(i_n)$ if $i_1,\dotsc,i_n\in E$ are pairwise distinct. Then by \cite[Lemma 4.1]{Baker-Bowler19}, there is a Grassmann-Pl\"ucker function $\Delta^\ast:E^{n-r-1}\to P_M$ for $\widehat M^\ast$ that satisfies
 \[
  \Delta^\ast(i_1,\dotsc,i_{n-r}) \ = \ \sign(i_1,\dotsc,i_{n})\Delta(i_{n-r+1},\dotsc,i_n)
 \]
 for all pairwise distinct $i_1,\dotsc,i_{n}\in E$. Thus if $\bI=(i_1,\dotsc,i_{n-r-1})$, $\bJ=(j_1,\dotsc,j_{r-1})$ and $e,f\in E$ are as in the hypothesis of the theorem, then
 \[
  \omega_M\bigg(\frac{T_{\bI e}}{T_{\bI f}}\bigg) \ = \ \frac{\Delta^\ast(\bI e)}{\Delta^\ast(\bI f)} \ = \ \frac{\sign(i_1,\dotsc,i_{n-r-1},e,j_1,\dotsc,j_{r-1},f)\Delta(\bJ f)}{\sign(i_1,\dotsc,i_{n-r-1},f,j_1,\dotsc,j_{r-1},e)\Delta(\bJ e)}  \ = \ - \frac{T_{\bJ f}}{T_{\bJ e}},
 \]
 as claimed. If $(J;e_1,\dotsc,e_4)\in\Omega_M$ and $I=E-Je_1\dotsc e_4$, then $J e_ie_k$ is a basis for $M$, and thus $I e_je_l$ is a basis for $M^\ast$ for all $i,j\in\{1,2\}$ and $k,l\in\{3,4\}$. Thus $(I;e_1,\dotsc,e_4)\in\Omega_{M^\ast}$. The image of the corresponding cross ratio under $\omega_M$ is
 \[
  \omega_M\bigg(\cross{e_1}{e_2}{e_3}{e_4}{I}\bigg) \ = \ \frac{\Delta^\ast(\bI e_1e_3)\Delta^\ast(\bI e_2e_4)}{\Delta^\ast(\bI e_1e_4)\Delta^\ast(\bI e_2e_3)} \ = \ \cross{e_1}{e_2}{e_3}{e_4}{\widehat M^\ast,I} \ = \ \cross{e_1}{e_2}{e_3}{e_4}{\widehat M,J}
 \]
 where $\bI\in E^{n-r-2}$ such that $|\bI|=I$ and where we use Lemma \ref{lemma: cross ratio of the dual matroid} for the last equality. Since the foundations of $M$ and $M^\ast$ are generated by cross ratios, it follows at once that $\omega_M$ restricts to an isomorphism $F_{M^\ast}\to F_M$.
\end{proof}

\subsection{Foundations of embedded minors}
\label{subsection: foundations of embedded minors}

Let $M$ be a matroid of rank $r$ on $E$, and let $\widehat M$ be the universal $P_M$-matroid associated with $M$, whose characteristic function is the identity map on $P_M$; cf.\ Theorem \ref{thm: universal pasture and representability}. Let $\Delta:E^r\to P_M$ be a Grassmann-Pl\"ucker function for $\widehat M$; e.g.\ we can choose some $\bI_0\in E^r$ such that $|\bI_0|$ is a basis of $M$ and define $\Delta(\bI)=T_\bI/T_{\bI_0}$ if $|\bI|$ is a basis of $M$ and $\Delta(\bI)=0$ if not.

Let $N=M\minor IJ$ be an embedded minor of $M$. Let $s$ be its rank and $E_N=E-(I\cup J)$ its ground set. Choose an ordering $J=\{j_{s+1},\dotsc,j_r\}$ of the elements of $J$. By \cite[Lemma 4.4]{Baker-Bowler19}, the function
\[
 \begin{array}{cccc}
  \Delta\minor IJ: & E^s_N & \longrightarrow & P_M \\
                   & \bI   & \longmapsto     & \Delta(\bI j_{s+1}\dotsc j_r)
 \end{array}                  
\]
is a Grassmann-Pl\"ucker function that represents $N=M\minor IJ$ and its isomorphism class $\widehat N=\widehat M\minor IJ$ is independent of the choice of ordering of $J$. The characteristic function of the $P_M$-matroid $\widehat N$ is a morphism $\psi_{M\minor IJ}:P_N\to P_M$; once again cf.\ Theorem \ref{thm: universal pasture and representability}.

\begin{prop}\label{prop: minor embeddings induce morphisms between foundations}
 Let $M$ be a matroid of rank $r$ on $E$ and $N=M\minor IJ$ an embedded minor of rank $s$ on $E_N=E-(I\cup J)$. Let $J=\{j_{s+1},\dotsc,j_r\}$. Then the morphism $\psi_{M\minor IJ}:P_N\to P_M$ satisfies the following properties.
 \begin{enumerate}
  \item\label{minor1}
   For all $\bI,\bJ\in E_N^s$ such that $|\bI|$ and $|\bJ|$ are bases of $N$, we have
   \[
    \psi_{M\minor IJ}\bigg(\frac{T_{\bI}}{T_{\bJ}}\bigg) \ = \ \frac{T_{\bI j_{s+1}\dotsc j_r}}{T_{\bJ j_{s+1}\dotsc j_r}}.
   \]
  \item\label{minor2}
   The identification $N^\ast=M^\ast\minor JI$ yields a commutative diagram
   \[
    \begin{tikzcd}[column sep=3cm, row sep=0.6cm]
     P_{N^\ast} \arrow[r,"\psi_{M^\ast\minor JI}"] \arrow[d,swap,"\omega_N"] & P_{M^\ast} \arrow[d,"\omega_M"] \\
     P_N \arrow[r,"\psi_{M\minor IJ}"]                                       & P_M
    \end{tikzcd}
   \]
   of pastures, where $\omega_N$ and $\omega_M$ are the isomorphisms from Proposition \ref{prop: foundation of the dual matroid}. 
  \item\label{minor3}
   The morphism $\psi_{M\minor IJ}:P_N\to P_M$ restricts to a morphism $\varphi_{M\minor IJ}:F_N\to F_M$ between the foundations of $N$ and $M$. For $(J';e_1,\dotsc,e_4)\in \Omega_N$, we have $(J'\cup J;e_1,\dotsc,e_4)\in\Omega_M$ and 
   \[
    \varphi_{M\minor IJ}\bigg(\cross{e_1}{e_2}{e_3}{e_4}{J'} \bigg) \ = \  \cross{e_1}{e_2}{e_3}{e_4}{J\cup J'}.
   \]
  \item\label{minor4}
   If every element in $I$ is a loop and if every element in $J$ is a coloop, then $\psi_{M\minor IJ}$ is an isomorphism. If every element in $I$ is a loop or parallel to an element in $E_N$ and if every element in $J$ is a coloop or coparallel to an element in $E_N$, then $\varphi_{M\minor IJ}$ is an isomorphism. 
 \end{enumerate}
\end{prop}

\begin{proof}
 Property \eqref{minor1} follows from the direct computation
 \[
  \psi_{M\minor IJ}\bigg(\frac{T_{\bI}}{T_{\bJ}}\bigg) \ = \ \frac{\Delta\minor IJ(T_\bI)}{\Delta\minor IJ(T_\bJ)} \ = \ \frac{T_{\bI j_{s+1}\dotsc j_r}}{T_{\bJ j_{s+1}\dotsc j_r}}.
 \]
 
 We continue with \eqref{minor2}. Let $r^\ast$ be the corank of $M$ and $s^\ast$ the corank of $N$. Choose an ordering $I=\{i_{s^\ast+1},\dotsc,i_{r^\ast}\}$. Let $\bI\in E_N^{s^\ast-1}$, $\bJ\in E_N^{s-1}$ and $e,f\in E_N$ be such that $E_N=|\bI|\cup|\bJ|\cup\{e,f\}$, which are the assumptions needed to apply Proposition \ref{prop: foundation of the dual matroid} to $\omega_N$. Since $P_{N^\ast}$ is generated by elements of the form $T_{\bI e}/T_{\bI f}$, the commutativity of the diagram in question follows from
 \begin{multline*}
  \psi_{M\minor IJ}\circ\omega_N\bigg(\frac{T_{\bI e}}{T_{\bI f}}\bigg) \ = \ \psi_{M\minor IJ}\bigg(-\frac{T_{\bJ f}}{T_{\bJ e}}\bigg) \ = \ -\frac{T_{\bJ fj_{s+1}\dotsc j_r}}{T_{\bJ ej_{s+1}\dotsc j_r}} \\ 
  = \ \omega_M\bigg( \frac{T_{\bI ei_{s^\ast+1}\dotsc i_{r^\ast}}}{T_{\bI fi_{s^\ast+1}\dotsc i_{r^\ast}}} \bigg) \ = \ \omega_M\circ\psi_{M^\ast\minor JI} \bigg(\frac{T_{\bI e}}{T_{\bI f}}\bigg).
 \end{multline*}
 Note that we can apply Proposition \ref{prop: foundation of the dual matroid} to $\omega_M$ since $E=|\bI|\cup|\bJ|\cup\{e,f\}\cup I\cup J$.
 
 We continue with \eqref{minor3}. If $(J';e_1,\dotsc,e_4)\in\Omega_N$, then for all $i\in\{1,2\}$ and $k\in\{3,4\}$, the set $J'e_ie_k$ is a basis of $N$ and thus $J'\cup J\cup\{e_i,e_k\}$ is a basis of $M$. Thus $(J'\cup J;e_1,\dotsc,e_4)\in\Omega_M$. Let $\bJ'\in E_N^s$ such that $|\bJ'|=J'$. Then
 \[
  \psi_{M\minor IJ}\bigg(\cross{e_1}{e_2}{e_3}{e_4}{J'} \bigg) = \Delta_N\bigg(\frac{T_{\bJ'e_1e_3}T_{\bJ'e_2e_4}}{T_{\bJ'e_1e_4}T_{\bJ'e_2e_3}}\bigg) = \frac{T_{\bJ'e_1e_3j_{s+1}\dotsc j_r}T_{\bJ'e_2e_4j_{s+1}\dotsc j_r}}{T_{\bJ'e_1e_4j_{s+1}\dotsc j_r}T_{\bJ'e_2e_3j_{s+1}\dotsc j_r}} = \cross{e_1}{e_2}{e_3}{e_4}{J\cup J'}.
 \]
 By Theorem \ref{thm: foundation is generated by cross ratios}, the foundation of a matroid is generated by its cross ratios. Thus the previous calculation shows that $\psi_{M\minor IJ}$ restricts to a morphism $\varphi_{M\minor IJ}:F_N\to F_M$ which maps $\cross{e_1}{e_2}{e_3}{e_4}{J'}$ to $\cross{e_1}{e_2}{e_3}{e_4}{J'\cup J}$.

 We continue with \eqref{minor4}. By successively deleting or contracting one element at a time, it suffices to prove the claim for $\#(I\cup J)=1$. Using \eqref{minor2}, we can assume that $I=\{e\}$ and $J=\emptyset$. If $e$ is a loop, then $I'\mapsto I'$ defines a bijection between the set of bases $I'\subset E_N=E-\{e\}$ of $N$ and the set of bases of $M$. Moreover, for every $(J';e_1,\dotsc,e_4)\in\Omega_M$, we have $e\notin J'e_1\dotsc e_4$, which provides an identification $\Omega_N=\Omega_M$. Thus $P_N$ and $P_M$ have the same generators and the same $3$-term Pl\"ucker relations, so $\psi_{M\minor IJ}:P_N\to P_M$ is an isomorphism. This argument also shows that $\varphi_{M\minor IJ}:F_N\to F_M$ is an isomorphism. 
 
 If $e$ is parallel to an element $f\in E_N$, then $\gen{J'e}=\gen{J'f}$ for every subset $J'$ of $E_N$. Thus for $e_1,\dotsc,e_4\in E$ and $f_1,\dotsc,f_4\in E_N$ with either $e_i=f_i$ or $e_{i}=e$ and $f_{i}=f$ for $i=1,\dotsc,4$, we have $(J';e_1,\dotsc,e_4)\in\Omega_M$ if and only if $(J';f_1,\dotsc,f_4)\in\Omega_N$, and $\varphi_{M\minor IJ}\Big(\cross{f_1}{f_2}{f_3}{f_4}{J'}\Big)=\cross{e_1}{e_2}{e_3}{e_4}{J'}$. This shows that $\varphi_{M\minor IJ}:F_N\to F_M$ is an isomorphism, which completes the proof.
\end{proof}

An immediate consequence of Proposition \ref{prop: minor embeddings induce morphisms between foundations} is the following.

\begin{cor}
 The foundation of a matroid is isomorphic to the foundation of its simplification and isomorphic to the foundation of its cosimplification. 
\end{cor}

\begin{proof}
 This follows at once from Proposition \ref{prop: minor embeddings induce morphisms between foundations}, since the simplification of a matroid $M$ is an embedded minor of $M$ of the form $M\backslash I$, where $I$ consists of all loops of $M$ and a choice of all but one element in each class of parallel elements. Similarly, the cosimplification of $M$ is an embedded minor of $M$ of the form $M/J$, where $J$ consists of all coloops of $M$ and a choice of all but one element in each class of coparallel elements. 
\end{proof}

Another consequence of Proposition \ref{prop: minor embeddings induce morphisms between foundations}, which we will utilize constantly in the upcoming sections, is the following observation. Since a universal cross ratio $\cross{e_1}{e_2}{e_3}{e_4}{J}$ involves only bases ${Je_ie_k}$ that contain $J$ and have a trivial intersection with $I=E-Je_1e_2e_3e_4$, we have
\[
 \cross{e_1}{e_2}{e_3}{e_4}{J} \ = \ \frac{T_{\bJ e_1e_3}T_{\bJ e_2e_4}}{T_{\bJ e_1e_4}T_{\bJ e_2e_3}} \ = \ \frac{\varphi(T_{(e_1,e_3)})\ \varphi(T_{(e_2,e_4)})}{\varphi(T_{(e_1,e_4)})\ \varphi(T_{(e_2,e_3)})} \ = \ \varphi\bigg(\cross{e_1}{e_2}{e_3}{e_4}{\emptyset}\bigg)
\]
for the morphism $\varphi=\varphi_{M\minor IJ}:F_{M\minor IJ}\to F_M$ from Proposition \ref{prop: minor embeddings induce morphisms between foundations}. Thus every universal cross ratio in $F_M$ is the image of a universal cross ratio of an embedded minor $N=M\minor IJ$ of rank $2$ on a $4$-element set $\{e_1,e_2,e_3,e_4\}=E-(I\cup J)$.


\subsection{The foundation of \texorpdfstring{$U^2_4$}{U(2,4)}}
\label{subsection: the foundation of U24}

Let $M=U^2_4$ be the uniform minor of rank $2$ on the set $E=\{1,\dotsc,4\}$, which is represented by the Grassmann-Pl\"ucker function $\Delta:E^2\to\K$ with $\Delta(i,j)=1-\delta_{i,j}$. The cross ratios of $M$ are of the form
\[
 \cross {e_1}{e_2}{e_3}{e_4}{} \ := \ \cross {e_1}{e_2}{e_3}{e_4}\emptyset
\]
for some permutation $e:i\mapsto e_i$ of $E$. Since permuting columns and rows in $\cross{e_1}{e_2}{e_3}{e_4}{}$ does not change the cross ratio, as pointed out in section \ref{subsection: cross ratios of P-matroids}, we have
\[\tag{R$\sigma^*$}\label{Rs*}
 \cross 1234{} \ = \ \cross 2143{} \ = \ \cross 3412{} \ = \ \cross 4321{}.
\]
Thus we can assume that $e_1=1$, and with this convention, we find that each of the 24 possible cross ratios is equal to one of the following six:
\[
 \cross 1234{}, \quad \cross 1243{}, \quad \cross 1324{}, \quad \cross 1342{}, \quad \cross 1423{}, \quad \cross 1432{}.
\]
They satisfy the following two types of multiplicative relations
\[\tag{R1${}^*$}\label{R1*}
 \cross 1243{} \ = \ \cross 1234{}^{-1}, \qquad \cross 1243{} \ = \ \cross 1234{}^{-1}, \qquad \cross 1243{} \ = \ \cross 1234{}^{-1};
\]
\[\tag{R2${}^*$} \label{R2*}
 \cross 1234{} \cdot \cross 1342{} \cdot \cross 1423{} \ = \ -1, \qquad \cross 1243{} \cdot \cross 1324{} \cdot \cross 1432{} \ = \ -1;
\]
and the Pl\"ucker relations
\[\tag{R+${}^*$} \label{R+*}
 \cross 1234{} + \cross 1324{} = 1, \quad \cross 1342{} + \cross 1432{} = 1, \quad \cross 1423{} + \cross 1243{} = 1.
\]
These relations can be illustrated in the form of a hexagon, see Figure \ref{fig: hexagon}. The three edges with label $*$ refer to relations of type \eqref{R1*}, the three edges with label $+$ refer to the Pl\"ucker relations \eqref{R+*}, and the two inner triangles refer to the relations of type \eqref{R2*}. 

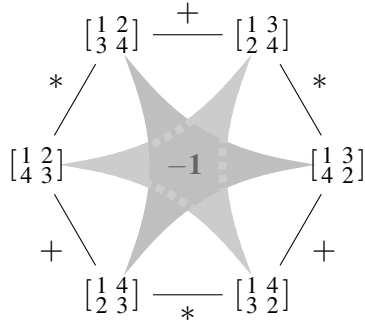
\begin{figure}[htb]
 \centering
 \leavevmode
 \beginpgfgraphicnamed{tikz/fig3}
  \begin{tikzpicture}[x=1cm,y=1cm]
   \draw[line width=3pt,color=gray!40,fill=gray!40,bend angle=20] (60:1.4) to[bend left] (180:1.4) to[bend left] (300:1.4) to[bend left] cycle;
   \draw[line width=3pt,color=gray!50,fill=gray!50,bend angle=20] (0:1.4) to[bend left] (120:1.4) to[bend left] (240:1.4) to[bend left] cycle;
   \draw[line width=3pt,color=gray!40,dotted,bend angle=20] (60:1.4) to[bend left] (180:1.4) to[bend left] (300:1.4) to[bend left] cycle;
   \node (-1) at (0,0) {\small \textcolor{black!60}{{$\mathbf{-1\ }$}}};
   \node (1234) at (-1, 1.73) {$\cross 1234{}$};
   \node (1324) at ( 1, 1.73) {$\cross 1324{}$};
   \node (3124) at ( 2, 0   ) {$\cross 1342{}$};
   \node (3214) at ( 1,-1.73) {$\cross 1432{}$};
   \node (2314) at (-1,-1.73) {$\cross 1423{}$};
   \node (2134) at (-2, 0   ) {$\cross 1243{}$};
   \path (1234) edge node[auto] {$+$} (1324);
   \path (1324) edge node[auto] {$*$} (3124);
   \path (3124) edge node[auto] {$+$} (3214);
   \path (3214) edge node[auto] {$*$} (2314);
   \path (2314) edge node[auto] {$+$} (2134);
   \path (2134) edge node[auto] {$*$} (1234);
  \end{tikzpicture}
 \endpgfgraphicnamed
 \caption{The hexagon of cross ratios of $U^2_4$}
 \label{fig: hexagon}
\end{figure}

Note that we can rewrite the relations of type \eqref{R1*} as $\cross 1234{} \cdot \cross 1243{}=1$, and so forth, which highlights an analogy with the Pl\"ucker relations $\cross 1234{} + \cross 1324{}=1$. This makes the meaning of the edge labels $\ast$ and $+$ easy to remember.

\begin{prop}\label{prop: foundation of U24}
 Let $x=\cross 1234{}$ and $y=\cross 1324{}$. Then the foundation of $M=U^2_4$ is
 \[
  F_M \ = \ \U \ = \ \pastgen{\Funpm}{x,y}{x+y-1}.
 \]
 In particular, we have 
\[
 \cross 1243 \ = \ x^{-1}, \qquad \cross 1342 \ = \ y^{-1}, \qquad \cross 1432 \ = \ -xy^{-1}, \qquad \cross 1423 \ = \ -x^{-1}y.
\] 
\end{prop}

\begin{proof}
 By relation \eqref{Rs*}, $F_M$ is generated by the $6$ cross ratios
 \[
  x \ = \ \cross 1234{}, \quad y \ = \ \cross 1324{}, \quad \cross 1243{}, \quad \cross 1342{}, \quad \cross 1432{}, \quad \cross 1423{}.
 \]
 By relation \eqref{R1*}, we have
 \[
  \cross 1243{} \ = \ \crossinv 1234{} \ = \ x^{-1} \qquad \text{and} \qquad \cross 1342{} \ = \ \crossinv 1324{} \ = \ y^{-1}. 
 \]
 Relation \eqref{R2*}, paired with \eqref{R1*}, yields
 \[
  \cross 1432{} \ = \ \crossinv 1423{} \ = \ - \ \cross 1234{} \ \cdot \ \cross 1342{} \ = \ -xy^{-1}.
 \]
 Applying \eqref{R1*} once again yields
 \[
  \cross 1423{} \ = \ \crossinv 1432{} \ = \ -x^{-1}y.
 \]
 By \eqref{R+*}, we have $x+y-1=0$. This shows that the foundation $F_M$ of $M=U^2_4$ is a quotient of $\U=\pastgen{\Funpm}{x,y}{x+y-1}$. 
 
There are several different ways to show that there are no further relations in $F_M$ aside from those already present in $\U$, for example:
\begin{enumerate}
\item One can work this out ``by hand''.
\item One can utilize the fact that $U^2_4$ is near-regular, which implies that there is a morphism $F_M\to\U$. 
\item One can apply Theorem~\ref{thm: presentation of foundations in terms of bases}, whose proof does not rely on Proposition~\ref{prop: foundation of U24}.
\end{enumerate}

We explain a fourth route, which uses a theorem of Dress and Wenzel determining the inner Tutte group of a uniform matroid.
In the case of $M=U^2_4$, \cite[Thm.\ 8.1]{Dress-Wenzel89}, paired with Corollary \ref{cor: inner Tutte group as untis of the foundation}, shows that $F_M^\times\simeq\T^{(0)}\simeq (\Z/2\Z)\times\Z^2\simeq\U^\times$. We conclude that the quotient map $\U\to F_M$ is an isomorphism between the underlying monoids. We are left with showing that every relation in the nullset of $F_M$ comes from $\U$, which is the intersection of the nullset $N_{P_M^+}$ of $P_M^+$ with $\Sym^3(F_M)$. Since $N_{P_M^+}$ is generated by the single term
 \[
  T_{1,2}T_{3,4} \ - \ T_{1,3}T_{2,4} \ + \ T_{1,4}T_{2,3} \ = \ -T_{1,4}T_{2,3} \cdot (x+y-1),
 \]
 where we use the short-hand notation $T_{i,j}=T_{(i,j)}$, every term in $N_{F_M}$ is a multiple of $x+y-1$. This shows that $\U\to F_M$ is an isomorphism.
\end{proof}

Morphisms from $\U$ into another pasture can be studied in terms of pairs of fundamental elements:

\begin{df}
 A \emph{pair of fundamental elements in $P$} is an ordered pair $(z,z')$ of elements $z,z'\in P^\times$ such that $z+z'=1$.
\end{df}

\begin{lemma}\label{lemma: morphisms from U to P are determined by pairs of fundamental elements}
 Let $P$ be a pasture. Then there is a bijection between $\Hom(\U,P)$ with the set of pairs of fundamental elements.
\end{lemma}

\begin{proof}
 Every morphism $f:\U=\pastgen\Funpm{x,y}{x+y=1}\to P$ maps $x$ and $y$ to invertible elements in $P$. Since $x+y=1$, we have $f(x)+f(y)=1$ in $P$, which shows that $\big(f(x),f(y)\big)$ is a pair of fundamental elements. This defines a map $\Phi:\Hom(\U,P)\to \cF_P$, where $\cF_P$ is the set of pairs of fundamental elements in $P$.
 
 Since $f$ is determined by the images of $x$ and $y$, we see that $\Phi$ is injective. On the other hand, for every pair $(u,v)$ of fundamental elements in $P$, the map $x\mapsto u$ and $y\mapsto v$ extends to a morphism $f:\U\to P$. Thus $\Phi$ is surjective as well.
\end{proof}

Recall that a reorientation class is a rescaling class over the sign hyperfield $\S$. The following corollary is well known:

\begin{cor}\label{cor: representations and rescaling classes of U24} 
 The rescaling classes of $U^2_4$ over a field $k$ are in bijection with $k-\{0,1\}$, and $U^2_4$ has $3$ reorientation classes.
\end{cor}

\begin{proof}
  If $P=k$ is a field, then $y=1-x$ is uniquely determined by $x$, and $x,y$ both belong to $k^\times$ precisely when $x\in k-\{0,1\}$, which establishes the first claim. The second claim follows from the observation that $a+b=1$ in $\S$ if and only if $(a,b)$ is one of the $3$ pairs $(1,1)$, $(1,-1)$ and $(-1,1)$.
\end{proof}


\subsection{The tip and cotip relations}
\label{subsection: tip and cotip relations}

In this section, we exhibit two types of relations that occur for matroids of ranks $2$ and $3$, respectively, on the five element set $E=\{1,\dotsc,5\}$.

As in the case of the uniform matroid $U^2_4$, we write $\cross ijkl{}$ for $\cross ijkl\emptyset$ in the case of a rank $2$-matroid $M$. 
We also use the shorthand notation $T_{i,j}=T_{(i,j)}$ and $T_{i,j,k}=T_{(i,j,k)}$.

\begin{lemma}\label{lemma: tip relation}
 Let $M$ be a matroid of rank $2$ on $E=\{1,\dotsc,5\}$. Assume that $\{i,j\}$ is a basis of $M$ for all $i\in\{1,2\}$ and all $j\in\{3,4,5\}$. Then 
 \[\tag{R3*} \label{R3*}
  \cross 1234{} \cdot \cross 1245{} \cdot \cross 1253{} \ = \ 1. 
 \]
\end{lemma}

\begin{proof}
 Equation \eqref{R3*} follows from the direct computation
 \[
  \cross 1234{} \cdot \cross 1245{} \cdot \cross 1253{} \ = \ \frac{T_{1,3}T_{2,4}}{T_{1,4}T_{2,3}} \cdot \frac{T_{1,4}T_{2,5}}{T_{1,5}T_{2,4}} \cdot \frac{T_{1,5}T_{2,3}}{T_{1,3}T_{2,5}} \ = \ 1. \qedhere
 \]
\end{proof}

We call equation \eqref{R3*} the \emph{tip relation with tip $\{1,2\}$ and cyclic orientation $(3,4,5)$}. The reason for this terminology is that in the case of the uniform matroid $M=U^2_5$, the three cross ratios in equation \eqref{R3*} stem from three octahedrons in the basis exchange graph of $M$, which share exactly one common vertex, or \emph{tip}, which is $\{1,2\}$.

Note that if $M$ is not uniform, i.e.\ some $2$-subsets $\{i,j\}$ of $E$ are not bases, then some of the cross ratios in equation \eqref{R3*} are trivial. We will examine this situation in more detail in section \ref{subsection: foundations of matroids on 5 elements}.

In the case of a matroid of rank $3$, we write $\cross ijklm$ for $\cross ijkl{\{m\}}$.

\begin{lemma}\label{lemma: cotip relation}
 Let $M$ be a matroid of rank $3$ on $E=\{1,\dotsc,5\}$. Assume that $\{i,j,k\}$ is a basis of $M$ for all $i\in\{1,2\}$ and all $j,k\in\{3,4,5\}$ with $j\neq k$. Then 
 \[\tag{R4*}\label{R4*} 
  \cross 1234{5} \cdot \cross 1245{3} \cdot \cross 1253{4} \ = \ 1.
 \]
\end{lemma}

\begin{proof}
 Equation \eqref{R4*} follows from the direct computation
 \begin{multline*}
  \cross 1234{5} \cdot \cross 1245{3} \cdot \cross 1253{4} \ = \ \frac{T_{5,1,3}\cdot T_{5,2,4}}{T_{5,1,4}\cdot T_{5,2,3}} \ \cdot \ \frac{T_{3,1,4}\cdot T_{3,2,5}}{T_{3,1,5}\cdot T_{3,2,4}} \ \cdot \ \frac{T_{4,1,5}\cdot T_{4,2,3}}{T_{4,1,3}\cdot T_{4,2,5}} \\
  = \ \frac{T_{4,1,5}}{-T_{4,1,5}} \cdot \frac{T_{3,2,5}}{-T_{3,2,5}} \cdot \frac{T_{5,1,3}}{-T_{5,1,3}} \cdot \frac{T_{4,2,3}}{-T_{4,2,3}} \cdot \frac{T_{3,1,4}}{-T_{3,1,4}} \cdot \frac{T_{5,2,4}}{-T_{5,2,4}} \ = \ (-1)^6 \ = \ 1. \mbox{\qedhere}
 \end{multline*}
\end{proof}

We call equation \eqref{R4*} the \emph{cotip relation with cotip $\{1,2\}$ and cyclic orientation $(3,4,5)$}. Similar to the rank $2$-case, we use this terminology since in the case of the uniform matroid $M=U^3_5$, the three cross ratios in equation \eqref{R4*} stem from three octahedrons in the basis exchange graph of $M$, which share exactly one common vertex, which is $\{3,4,5\}$. Therefore we call the complement $\{1,2\}$ of this common vertex the \emph{cotip}.

Note that the tip and cotip relations are both invariant under permuting $\{1,2\}$ and under cyclic permutations of $(3,4,5)$. Any other permutation of $E$ will produce another tip or cotip relation, provided that all involved values of $\Delta$ are nonzero.


\subsection{Relations for parallel elements}
\label{subsection: relations for parallel elements}

In this section, we exhibit a type of relation between universal cross ratios that stems from parallel elements. As in the previous section, we write $\cross 12345$ for $\cross 1234{\{5\}}$.

\begin{lemma}\label{lemma: relation for parallel elements}
 Let $M$ be a matroid of rank $3$ on $E=\{1,\dotsc,6\}$ and assume that $5$ and $6$ are parallel elements, i.e.\ $\{5,6\}$ is a circuit of $M$. If $(\{k\};1,\dotsc,4)\in\Omega_M$ for $k=5,6$, then
 \[\tag{R5*}\label{R5*}
  \cross 12345 \ = \ \cross 12346.
 \]
\end{lemma}

\begin{proof}
 By our assumptions, every subset of the form $\{i,j,k\}$ with $i\in\{1,2\}$, $j\in\{3,4\}$ and $k\in\{5,6\}$ is a basis of $M$, but no basis contains both $5$ and $6$. Thus $(\{1\};3,4,6,5)$ and $(\{2\};3,4,5,6)$ are degenerate tuples in $\Omega_M$, and thus $\cross 34651=\cross 34562=1$. With this, equation \eqref{R5*} follows from the computation
 \begin{multline*}
  \cross 1234{5} \ = \ \cross 1234{5} \cdot \cross 34651 \cdot \cross 34562 \ = \ \frac{ T_{5,1,3}\cdot T_{5,2,4}}{ T_{5,1,4}\cdot T_{5,2,3}} \ \cdot \ \frac{ T_{1,3,6}\cdot T_{1,4,5}}{ T_{1,3,5}\cdot T_{1,4,6}} \ \cdot \ \frac{ T_{2,3,5}\cdot T_{2,4,6}}{ T_{2,3,6}\cdot T_{2,4,5}} \\
  = \ \frac{ T_{1,4,5}}{ T_{1,4,5}} \ \cdot \ \frac{ T_{2,3,5}}{ T_{2,3,5}} \ \cdot \ \frac{ T_{5,1,3}}{ T_{5,1,3}} \ \cdot \ \frac{ T_{6,1,3}\cdot T_{6,2,4}}{ T_{6,1,4}\cdot T_{6,2,3}} \ \cdot \ \frac{ T_{5,2,4}}{ T_{5,2,4}} \ = \ \cross 12346. \mbox{\qedhere}
 \end{multline*}
\end{proof}


\subsection{The foundation of the Fano matroid and its dual}
\label{subsection: foundation of the Fano matroid and its dual}

In this section, we show that the Fano matroid $F_7$ and its dual $F_7^\ast$ impose the relation $-1=1$ on their foundation, which is $\F_2$. This already follows from \cite[Thms.\ 7.30 and 7.33]{Baker-Lorscheid18}, using the fact that $F_7$ and $F_7^\ast$ are not regular. Here we offer a proof in terms of a direct calculation that does not rely on knowledge of the representability of $F_7$.

The Fano matroid $F_7$ is the rank $3$ matroid on $E=\{1,\dotsc,7\}$ represented by the Grassmann-Pl\"ucker function $\Delta:E^3\to\K$ with $\Delta(i,i+1,i+3)=0$ for $i\in E$, where we read $i+1$ and $i+3$ modulo $7$, and $\Delta(i,j,k)=1$ otherwise. Thus the family of circuits is $\Big\{\{i,i+1,i+3\}\,\Big|\,i\in E\Big\}$, together with all $4$-element subsets that do not contain one of these, which can be illustrated as follows:
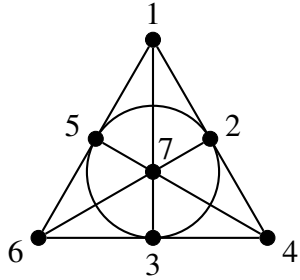
\[
 \beginpgfgraphicnamed{tikz/fig9}
  \begin{tikzpicture}[x=25pt,y=25pt]
   \node[draw,fill=black,circle,inner sep=2pt] at  (90:2) (1) {};  
   \node[draw,fill=black,circle,inner sep=2pt] at  (30:1) (2) {};  
   \node[draw,fill=black,circle,inner sep=2pt] at (270:1) (3) {};  
   \node[draw,fill=black,circle,inner sep=2pt] at (330:2) (4) {};  
   \node[draw,fill=black,circle,inner sep=2pt] at (150:1) (5) {};  
   \node[draw,fill=black,circle,inner sep=2pt] at (210:2) (6) {};  
   \node[draw,fill=black,circle,inner sep=2pt] at   (0:0) (7) {}; 
   \node at  (90:2.4) {$1$};  
   \node at  (30:1.4) {$2$};  
   \node at (270:1.4) {$3$};  
   \node at (330:2.4) {$4$};  
   \node at (150:1.4) {$5$};  
   \node at (210:2.4) {$6$};  
   \node at  (60:0.4) {$7$};  
   \draw[thick] (1) to (3);   
   \draw[thick] (1) to (4);   
   \draw[thick] (1) to (6);   
   \draw[thick] (4) to (6);   
   \draw[thick] (4) to (5);   
   \draw[thick] (6) to (2); 
   \draw[thick] (0,0) circle (1);
  \end{tikzpicture}
 \endpgfgraphicnamed
\] 

\begin{lemma}\label{lemma: foundation of the Fano matroid and its dual}
 The foundation of both the Fano matroid $F_7$ and its dual $F_7^\ast$ is $\F_2$.
\end{lemma}

\begin{proof}
 Since the foundation of $F_7^\ast$ is isomorphic to the foundation of $F_7$, it is enough to prove the lemma for the Fano matroid. Throughout the proof, we read expressions like $i+k$ and $i-k$ modulo $7$ for all $i,k\in E$. 
 
 Since the rank of $F_7$ is $3$, the set $J$ of a tuple $(J;e_1,\dotsc,e_4)\in\Omega_M$ is a singleton, i.e.\ $J=\{j\}$ for some $j\in E$. The element $j$ is contained in the three circuits $C_1=\{j,j+1,j+3\}$, $C_2=\{j-1,j,j+2\}$ and $C_3=\{j-3,j-2,j\}$ whose union is equal to $E$. By the pigeonhole principle, we must have $e_k,e_l\in C_i$ for some $i$ and $k\neq l$. Since $j,e_k,e_l$ are pairwise distinct, $C_i=\{j,e_k,e_l\}$ is not a basis. This shows that every $(J;e_1,\dotsc,e_4)\in\Omega_M$ is degenerate, and thus $\cross{e_1}{e_2}{e_3}{e_4}{J}=1$. We conclude that $F_M$ is a quotient of $\Funpm$.
 
 We use the shorthand notations $\cross ijklm=\cross ijkl{\{m\}}$ and $T^i_{j,k,l}=T_{(i+j,i+k,i+l)}$ in the following. Note that $T^{i-m}_{j+m,k+m,l+m}=T^i_{j,k,l}$ and $T^i_{\sigma(j),\sigma(k),\sigma(l)}=\sign(\sigma)T^i_{j,k,l}$ for every permutation $\sigma$ of $\{j,k,l\}$. We calculate that 
 \begingroup
  \allowdisplaybreaks 
  \begin{align*}
   1 \quad &= \quad \prod_{i=1}^7 \quad \cross{i+1}{i+3}{i+2}{i+4}{i} \ \cdot \ \cross{i+2}{i+6}{i+5}{i+4}{i} \\
           &= \quad \prod_{i=1}^7 \quad \frac{T^i_{0,1,2} \cdot T^i_{0,3,4}}{T^i_{0,1,4} \cdot T^i_{0,3,2}} \ \cdot \ \frac{T^i_{0,2,5} \cdot T^i_{0,6,4}}{T^i_{0,2,4} \cdot T^i_{0,6,5}} \\
           &= \quad \prod_{i=1}^7 \quad \frac{T^i_{0,1,2} \cdot T^i_{0,3,4} \cdot T^i_{0,2,5} \cdot T^i_{0,6,4}}{T^{i-3}_{3,4,0} \cdot T^{i-4}_{4,0,6} \cdot T^{i-5}_{5,0,2} \cdot T^{i-2}_{2,1,0}} \\
           &= \quad \prod_{i=1}^7 \quad \frac{T^i_{0,3,4}}{T^{i-3}_{0,3,4}} \ \cdot \ \frac{T^i_{0,6,4}}{T^{i-4}_{0,6,4}} \ \cdot \ \frac{T^i_{0,2,5}}{T^{i-5}_{0,2,5}} \ \cdot \ \frac{T^i_{0,1,2}}{-T^{i-2}_{0,1,2}} \\
           &= \quad (-1)^7 \quad = \quad -1.
  \end{align*}
 \endgroup
 This shows that the foundation $F_M$ of $F_7$ is a quotient of $\F_2=\past{\Funpm}{\{-1=1\}}$. Since $F_7$ does not contain any $U^2_4$-minors, all cross ratios are degenerate and thus the nullset of $F_M$ does not contain any $3$-term relations. We conclude that $F_M=\F_2$.
\end{proof}


\subsection{A presentation of the foundation by hyperplanes}
\label{subsection: presentation of the inner Tutte group}

Gelfand, Rybnikov and Stone exhibit in \cite[Thm.\ 4]{Gelfand-Rybnikov-Stone95} a complete set of multiplicative relations in the inner Tutte group of $M$ between the cross ratios $\cross{C_1}{C_2}{C_3}{C_4}{}$ of modular quadruples $(C_1,\dotsc,C_4)$ of circuits, which results in essence from Tutte's homotopy theorem. Since hyperplanes are just complements of circuits of the dual matroid, this set of relations yields at once a complete set of relations for cross ratios $\cross{H_1}{H_2}{H_3}{H_4}{}$ of modular quadruples $(H_1,\dotsc,H_4)$ of hyperplanes. 

\begin{thm}\label{thm: relations between hyperplane cross ratios}
 Let $M$ be a matroid with foundation $F_M$. Then
 \[\textstyle
  F_M \ = \ \Funpm \, \big\langle \, \cross {H_1}{H_2}{H_3}{H_4}{} \, \big| \, (H_1,\dotsc,H_4)\in\Theta_M \, \big\rangle \, \sslash \, S,
 \]
 where $S$ is defined by the multiplicative relations
 \[\tag{H--}\label{H-}
  (-1)^2 \ = \ 1, \qquad \text{and} \qquad -1=1
 \]
 if the Fano matroid $F_7$ or its dual $F_7^\ast$ is a minor of $M$;
 \[\tag{H$\sigma$}\label{Hs}
  \cross{H_1}{H_2}{H_3}{H_4}{} \ = \ \cross{H_2}{H_1}{H_4}{H_3}{} \ = \ \cross{H_3}{H_4}{H_1}{H_2}{} \ = \ \cross{H_4}{H_3}{H_2}{H_1}{}
 \]
 for all $(H_1,\dotsc,H_4)\in\Theta_\cH^\octa$;
  \[\tag{H0}\label{H0}
  \cross {H_1}{H_2}{H_3}{H_4}{} \ = \ 1
 \]
 for all degenerate $(H_1,\dotsc,H_4)\in\Theta_\cH$;
 \[\tag{H1}\label{H1}
  \cross {H_1}{H_2}{H_4}{H_3}{} \ = \ \crossinv {H_1}{H_2}{H_3}{H_4}{}
 \]
 for all $(H_1,\dotsc,H_4)\in\Theta_\cH^\octa$;
 \[\tag{H2}\label{H2}
  \cross {H_1}{H_2}{H_3}{H_4}{} \cdot \cross {H_1}{H_3}{H_4}{H_2}{}  \cdot \cross {H_1}{H_4}{H_2}{H_3}{} \ = \ -1
 \]
 for all $(H_1,\dotsc,H_4)\in\Theta_\cH^\octa$;
 \[\tag{H3}\label{H3}
  \cross {H_1}{H_2}{H_3}{H_4}{} \cdot \cross {H_1}{H_2}{H_4}{H_5}{} \cdot \cross {H_1}{H_2}{H_5}{H_3}{} \ = \ 1
 \]
 for all $(H_1,H_2,H_3,H_4),(H_1,H_2,H_4,H_5),(H_1,H_2,H_5,H_3)\in\Theta_\cH^\octa$; and
 \[\tag{H4}\label{H4}
  \cross {H_{15}}{H_{25}}{H_{35}}{H_{45}}{} \cdot \cross {H_{13}}{H_{23}}{H_{43}}{H_{53}}{} \cdot \cross {H_{14}}{H_{24}}{H_{54}}{H_{34}}{} \ = \ 1,
 \]
 where $H_{ij}=\gen{F_i\cup F_j}$ for five pairwise distinct corank $2$-flats $F_1,\dotsc,F_5$ that contain a common flat of corank $3$ such that $(H_{15},H_{25},H_{35},H_{45}),(H_{14},H_{24},H_{54},H_{34})\in\Theta_\cH^\octa$ and $(H_{13},H_{23},H_{43},H_{53})\in\Theta_\cH$,
  \\[7pt]
 as well as the additive Pl\"ucker relations
 \[\tag{H+}\label{H+}
  \cross {H_1}{H_2}{H_3}{H_4}{} + \cross {H_1}{H_3}{H_2}{H_4}{}  \ = \ 1
 \]
 for all $(J;e_1,\dotsc,e_4)\in\Theta_M^\octa$.
\end{thm}

\begin{proof}
 The theorem follows by translating the relations between cross ratios $\cross{C_1}{C_2}{C_3}{C_4}{\T}$ in $\T^{(0)}_{M^\ast}$ for modular quadruples of cycles of the dual matroid $M^\ast$ from \cite[Thm.\ 4]{Gelfand-Rybnikov-Stone95} to the hyperplane formulation by replacing a cocycle $C$ by the hyperplane $H=E-C$. To pass from the inner Tutte group to the foundation, we employ Lemma \ref{lemma: comparision of hyperplane cross ratios of the foundation and of the Tutte group}, which identifies $\cross {H_1}{H_2}{H_3}{H_4}{\T}$ with $\cross {H_1}{H_2}{H_3}{H_4}{}$ under the canonical isomorphism $\P_M^\times\to\T^{(0)}_M$.
 
 Using this translation, relation \eqref{H-} is equivalent to (TG0) and (CR5) in \cite{Gelfand-Rybnikov-Stone95}. Relation \eqref{Hs} is equivalent to (CR3). Relation \eqref{H0} is equivalent to (CR1). Relation (CR4) is equivalent to \eqref{H1} (in the case that one cross ratio is degenerate) and \eqref{H3} (in the case that all cross ratios are non-degenerate). Relation \eqref{H2} is equivalent to (CR4). Relation \eqref{H4} is equivalent to (CR6), where we observe that the degenerate case $L=L'$ in \cite{Gelfand-Rybnikov-Stone95} reduces (CR6) to (CR1). Finally note that the $3$-term Pl\"ucker relations of $F_M$ are captured in \eqref{H+}.
\end{proof}

\begin{rem}
 We include a discussion of relation \eqref{H4}, which has the most complicated formulation among the relations of Theorem \ref{thm: relations between hyperplane cross ratios}. Since all flats contain a common flat of corank $3$, this constellation comes from a minor of rank $3$, which has $5$ corank $2$-flats corresponding to $F_1,\dotsc,F_5$. In the non-degenerate situation where all hyperplanes $H_{ij}$ are pairwise distinct, this minor is of type $U^3_5$, and the containment relation of the $F_i$ and $H_{ij}$ can be illustrated as on the right-hand side of Figure \ref{fig: point-lines for U25 and flats for U35}. 
 
 The original formulation of Gelfand, Rybnikov and Stone concerns \emph{points}, which are circuits, and \emph{lines}, which are unions of circuits having projective dimension 1. To pass from our formulation to that of Gelfand-Rybnikov-Stone, we take the complement of a hyperplane $H_{ij}$, which is a circuit $C_{ij}$ of the dual matroid. Thus, in the non-degenerate case, axiom (CR6) expresses the point-line configuration of $U^2_5$, which we illustrate on the left-hand side of Figure \ref{fig: point-lines for U25 and flats for U35}. The lines $L_i$ are the complements of the flats $F_i$, and therefore the union of the circuits $C_{ij}$ (with varying $j$).
 
 Note that there are two degenerate situations that (CR6) allows for: \emph{(a)} three lines, say $L_1$, $L_2$ and $L_3$, intersect in one point $C_{12}=C_{13}=C_{23}$; this case corresponds to the point-line arrangement of a parallel extension of $U^2_4$, which we denote by $C_5^\ast$ in section \ref{subsubsection: matroids with exactly two embedded U24-minors}; and \emph{(b)} two lines agree; this case corresponds to the point-line arrangement of $U^2_4$.
\end{rem}

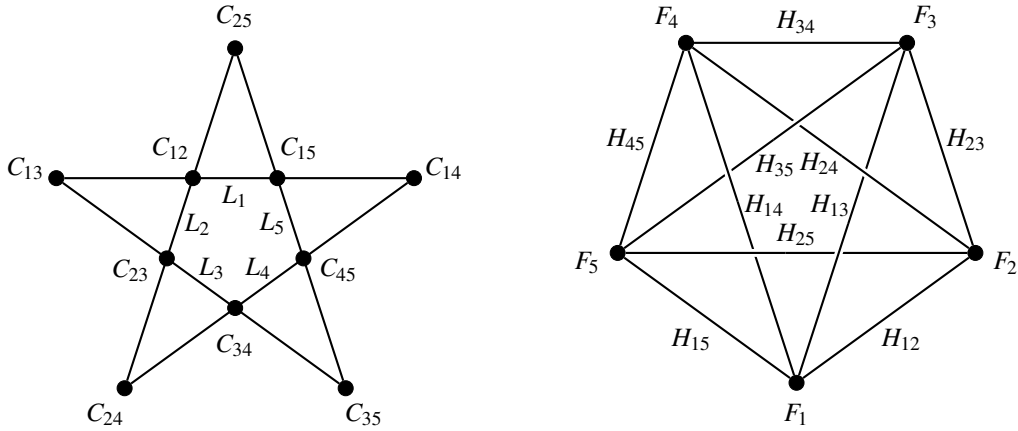
\begin{figure}[htb]
 \centering
 \leavevmode
 \beginpgfgraphicnamed{tikz/fig7}
  \begin{tikzpicture}[x=2.5cm,y=2.5cm]
   \node at  ( 90:1.17) {\footnotesize $C_{25}$};  
   \node at  (162:1.17) {\footnotesize $C_{13}$};  
   \node at  (234:1.17) {\footnotesize $C_{24}$};  
   \node at  (306:1.17) {\footnotesize $C_{35}$};  
   \node at  ( 18:1.17) {\footnotesize $C_{14}$};  
   \node at  (126:0.58) {\footnotesize $C_{12}$};  
   \node at  (198:0.58) {\footnotesize $C_{23}$};  
   \node at  (270:0.58) {\footnotesize $C_{34}$};  
   \node at  (342:0.58) {\footnotesize $C_{45}$};  
   \node at  ( 54:0.58) {\footnotesize $C_{15}$};  
   \node at  ( 90:0.21) {\footnotesize $L_1$};  
   \node at  (162:0.21) {\footnotesize $L_2$};  
   \node at  (234:0.21) {\footnotesize $L_3$};  
   \node at  (306:0.21) {\footnotesize $L_4$};  
   \node at  ( 18:0.21) {\footnotesize $L_5$};  
   \node[draw,fill=black,circle,inner sep=2pt] (25) at  ( 90:1) {};  
   \node[draw,fill=black,circle,inner sep=2pt] (13) at  (162:1) {};  
   \node[draw,fill=black,circle,inner sep=2pt] (24) at  (234:1) {};  
   \node[draw,fill=black,circle,inner sep=2pt] (35) at  (306:1) {};  
   \node[draw,fill=black,circle,inner sep=2pt] (14) at  ( 18:1) {};  
   \node[draw,fill=black,circle,inner sep=2pt] (12) at  (126:0.382) {};
   \node[draw,fill=black,circle,inner sep=2pt] (23) at  (198:0.382) {};
   \node[draw,fill=black,circle,inner sep=2pt] (34) at  (270:0.382) {};
   \node[draw,fill=black,circle,inner sep=2pt] (45) at  (342:0.382) {};
   \node[draw,fill=black,circle,inner sep=2pt] (15) at  ( 54:0.382) {};
   \draw[thick] (25) to (24) to (14) to (13) to (35) to (25);
  \end{tikzpicture}
 \endpgfgraphicnamed
 \hspace{1cm}
 \beginpgfgraphicnamed{tikz/fig8}
  \begin{tikzpicture}[x=2.5cm,y=2.5cm]
   \node at  ( 90:0.92) {\footnotesize $H_{34}$};  
   \node at  (162:0.95) {\footnotesize $H_{45}$};  
   \node at  (234:0.95) {\footnotesize $H_{15}$};  
   \node at  (306:0.95) {\footnotesize $H_{12}$};  
   \node at  ( 18:0.95) {\footnotesize $H_{23}$};  
   \node at  (126:1.17) {\footnotesize $F_{4}$};  
   \node at  (198:1.17) {\footnotesize $F_{5}$};  
   \node at  (270:1.17) {\footnotesize $F_{1}$};  
   \node at  (342:1.17) {\footnotesize $F_{2}$};  
   \node at  ( 54:1.17) {\footnotesize $F_{3}$};  
   \node at  (126:0.19) {\footnotesize $H_{35}$};  
   \node at  (198:0.18) {\footnotesize $H_{14}$};  
   \node at  (270:0.21) {\footnotesize $H_{25}$};  
   \node at  (342:0.19) {\footnotesize $H_{13}$};  
   \node at  ( 54:0.21) {\footnotesize $H_{24}$};  
   \node[draw,fill=black,circle,inner sep=2pt] (1) at  (270:1) {};
   \node[draw,fill=black,circle,inner sep=2pt] (2) at  (342:1) {};
   \node[draw,fill=black,circle,inner sep=2pt] (3) at  ( 54:1) {};
   \node[draw,fill=black,circle,inner sep=2pt] (4) at  (126:1) {};
   \node[draw,fill=black,circle,inner sep=2pt] (5) at  (198:1) {};
   \node (-1) at  (270:0.309) {};
   \node (-2) at  (342:0.309) {};
   \node (-3) at  ( 54:0.309) {};
   \node (-4) at  (126:0.309) {};
   \node (-5) at  (198:0.309) {};
   \draw[thick]  (1) to (2) to (3) to (4) to (5) to (1) to (3) to (5) to (2) to (4) to (1);
   \draw[line width=3pt,white]  (1) to (-2);
   \draw[thick]  (1) to (-2);
   \draw[line width=3pt,white]  (2) to (-3);
   \draw[thick]  (2) to (-3);
   \draw[line width=3pt,white]  (3) to (-4);
   \draw[thick]  (3) to (-4);
   \draw[line width=3pt,white]  (4) to (-5);
   \draw[thick]  (4) to (-5);
   \draw[line width=3pt,white]  (5) to (-1);
   \draw[thick]  (5) to (-1);
  \end{tikzpicture}
 \endpgfgraphicnamed
 \caption{Point-line configuration for $U^2_5$ and flat configuration for $U^3_5$}
 \label{fig: point-lines for U25 and flats for U35}
\end{figure}


\subsection{A presentation of the foundation by bases}
\label{subsection: presentation of the foundation}

Using the relation between cross ratios $\cross {H_1}{H_2}{H_3}{H_4}{}$ for modular quadruples $(H_1,\dotsc,H_4)$ of hyperplanes and universal cross ratios $\cross {e_1}{e_2}{e_3}{e_4}J$ for $(J;e_1,\dotsc,e_4)\in\Omega_M$, as exhibited in Proposition \ref{prop: comparison of hyperplane cross ratio with basis cross ratios}, we derive from Theorem \ref{thm: relations between hyperplane cross ratios} the following description of a complete set of relations between universal cross ratios. The possibility of such a deduction was observed and communicated to us by Rudi Pendavingh, who proves a similar result in the joint work \cite{Brettell-Pendavingh20} in progress with Brettell.

\begin{thm}\label{thm: presentation of foundations in terms of bases}
 Let $M$ be a matroid with foundation $F_M$. Then
 \[\textstyle
  F_M \ = \ \Funpm \, \big\langle \, \cross {e_1}{e_2}{e_3}{e_4}J \, \big| \, (J;e_1,\dotsc,e_4)\in\Omega_M \, \big\rangle \, \sslash \, S,
 \]
 where $S$ is defined by the multiplicative relations
 \[\tag{R--}\label{R-}
  -1=1 
 \]
 if the Fano matroid $F_7$ or its dual $F_7^\ast$ is a minor of $M$;
 \[\tag{R$\sigma$}\label{Rs}
  \cross {e_1}{e_2}{e_3}{e_4}J \ = \ \cross {e_2}{e_1}{e_4}{e_3}J \ = \ \cross {e_3}{e_4}{e_1}{e_2}J \ = \ \cross {e_4}{e_3}{e_2}{e_1}J
 \]
 for all $(J;e_1,\dotsc,e_4)\in\Omega_M^\octa$;
 \[\tag{R0}\label{R0}
  \cross {e_1}{e_2}{e_3}{e_4}J \ = \ 1
 \]
 for all degenerate $(J;e_1,\dotsc,e_4)\in\Omega_M$;
 \[\tag{R1}\label{R1}
  \cross {e_1}{e_2}{e_4}{e_3}J \ = \ \crossinv {e_1}{e_2}{e_3}{e_4}J
 \]
 for all $(J;e_1,\dotsc,e_4)\in\Omega_M^\octa$;
 \[\tag{R2}\label{R2}
  \cross {e_1}{e_2}{e_3}{e_4}J \cdot \cross {e_1}{e_3}{e_4}{e_2}J  \cdot \cross {e_1}{e_4}{e_2}{e_3}J \ = \ -1
 \]
 for all $(J;e_1,\dotsc,e_4)\in\Omega_M^\octa$;
 \[\tag{R3}\label{R3}
  \cross {e_1}{e_2}{e_3}{e_4}{J} \cdot \cross {e_1}{e_2}{e_4}{e_5}J \cdot \cross {e_1}{e_2}{e_5}{e_3}J \ = \ 1
 \]
 for all $e_1,\dotsc,e_5\in E$ and $J\subset E$ such that each of $(J;e_1,e_2,e_3,e_4)$, $(J;e_1,e_2,e_4,e_5)$ and $(J;e_1,e_2,e_5,e_3)$ is in $\Omega_M$;
 \[\tag{R4}\label{R4}
  \cross {e_1}{e_2}{e_3}{e_4}{Je_5} \cdot \cross {e_1}{e_2}{e_4}{e_5}{Je_3} \cdot \cross {e_1}{e_2}{e_5}{e_3}{Je_4} \ = \ 1
 \]
 for all $e_1,\dotsc,e_5\in E$ and $J\subset E$ such that $(Je_5;e_1,e_2,e_3,e_4)$, $(Je_3;e_1,e_2,e_4,e_5)$ and $(Je_4;e_1,e_2,e_5,e_3)$ are in $\Omega_M$; 
 \[\tag{R5}\label{R5}
  \cross {e_1}{e_2}{e_3}{e_4}{Je_5} \ = \ \cross {e_1}{e_2}{e_3}{e_4}{Je_6} 
 \]
 for all $e_1,\dotsc,e_6\in E$ and $J\subset E$ such that $\gen{Je_5}=\gen{Je_6}$ and such that $(Je_5;e_1,e_2,e_3,e_4)$ and $(Je_6;e_1,e_2,e_3,e_4)$ are in $\Omega_M^\octa$; 
 \\[7pt]
 as well as the additive Pl\"ucker relations
 \[\tag{R+}\label{R+}
  \cross {e_1}{e_2}{e_3}{e_4}J + \cross {e_1}{e_3}{e_2}{e_4}J  \ = \ 1
 \]
 for all $(J;e_1,\dotsc,e_4)\in\Omega_M^\octa$.
\end{thm}

\begin{proof}
 By Proposition \ref{prop: comparison of hyperplane cross ratio with basis cross ratios}, we have $\cross {H_1}{H_2}{H_3}{H_4}J=\cross {H_1}{H_2}{H_3}{H_4}{}$ for every $(J;e_1,\dotsc,e_4)\in\Omega_M$ and $H_i=\gen{Je_i}$ for $i=1,\dotsc,4$. Thus \eqref{R-}--\eqref{R3} follow from \eqref{H-}--\eqref{H3}. To see that \eqref{R4} implies \eqref{H4}, define for given $j_1,\dotsc,j_{r-3},e_1,\dotsc,e_6\in E$ and $J=\{j_1,\dotsc,j_{r-3}\}$ as in \eqref{R4} the corank $2$-flats $F_i=\gen{Je_i}$ for $i=1,\dotsc,5$, which are pairwise distinct and contain the common flat $\gen{J}$ of corank $3$, as required. For $i\neq j$, we define hyperplanes $H_{ij}=\gen{F_i\cup F_j}=\gen{Je_ie_j}$. Then we have for all identifications $\{i,j,k\}=\{3,4,5\}$ that 
 \[
  \cross {e_1}{e_2}{e_i}{e_j}{Je_k} \ = \ \cross {H_{1k}}{H_{2k}}{H_{ik}}{H_{jk}}{},
 \]
 which shows that \eqref{H4} implies \eqref{R4}. The relation \eqref{R5} follows from
 \[
  \cross{e_1}{e_2}{e_3}{e_4}{Je_5} \ = \ \cross{H_1}{H_2}{H_3}{H_4}{} \ = \ \cross{e_1}{e_2}{e_3}{e_4}{Je_6},
 \]
 where $H_i=\gen{Je_5e_i}=\gen{Je_6e_i}$ is $i$-th coefficient of the common image $(H_1,\dots,H_4)$ of $(Je_5;e_1,\dotsc,e_4)$ and $(Je_6;e_1,\dotsc,e_4)$ under $\Psi:\Omega_M\to\Theta_M$.
 
 We are left to show that $\cross{e_1}{e_2}{e_3}{e_4}{J}=\cross{e_1'}{e_2'}{e_3'}{e_4'}{J'}$ if $\Psi(J;e_1,\dotsc,e_4)=\Psi(J';e'_1,\dotsc,e'_4)$, i.e.\ if $\gen{Je_i}=\gen{J'e_i'}$ for $i=1,\dotsc,4$. We will prove this by replacing one element of $Je_1\dotsc e_4$ by an element of $J'e'_1\dotsc e'_4$ at a time. Note that both $J$ and $J'$ are bases of the restriction $M\vert_F=M\backslash (E-F)$, where $F=\gen J=\gen{J'}$ is the flat of rank $r-2$ generated by $J$ and $J'$. Since the basis exchange graph of $M\vert_F$ is connected, we find a sequence $J=J_0,J_1,\dotsc,J_{s-1},J_s=J'$ of bases for $M\vert_F$ such that $J_k=I_kj_k$ and $J_{k+1}=I_kj'_k$ for $I_k=J_k\cap J_{k+1}$ and some $j_k\in J_k$ and $j_k'\in J_{k+1}$. Considered as subsets of $M$, we have $\gen{J_k}=F$ and thus $(J_k;e_1',\dotsc,e_4')\in\Omega_M^\octa$ for all $k=0,\dotsc,s$. Thus we can apply \eqref{R5}, which yields
 \[
  \cross{e_1}{e_2}{e_3}{e_4}{J_k} \ = \ \cross{e_1}{e_2}{e_3}{e_4}{I_kj_k} \ = \ \cross{e_1}{e_2}{e_3}{e_4}{I_kj'_k}  \ = \ \cross{e_1}{e_2}{e_3}{e_4}{J_{k+1}}.
 \]
 We conclude that $\cross{e_1}{e_2}{e_3}{e_4}{J}=\cross{e_1}{e_2}{e_3}{e_4}{J'}$.
 
 Next we replace the $e_i$ by the $e_i'$, one at a time. After permuting rows and columns appropriately, which does not change the value of the cross ratio by \eqref{Rs}, we are reduced to studying cross ratios of the forms $\cross{f_1}{f_2}{f_3}{f_4}{J'}$ and $\cross{f_1}{f_2}{f_3}{f'_4}{J'}$ such that $\gen{J'f_4}=\gen{J'f'_4}$ is a hyperplane. By \eqref{R3}, we have
 \[
  \cross {f_1}{f_2}{f_3}{f_4}{J'} \cdot \cross {f_1}{f_2}{f_4}{f_4'}{J'} \cdot \cross{f_1}{f_2}{f_4'}{f_3}{J'} \ = \ 1.
 \]
 Since $\gen{J'f_4}=\gen{J'f'_4}$ is a hyperplane, the subset $J'f_4f_4'$ of $M$ has rank $r-1$ and is not a basis of $M$. Thus $\cross {f_1}{f_2}{f_4}{f_4'}{J'}=1$ by \eqref{R0}, which shows that
 \[
  \cross {f_1}{f_2}{f_3}{f_4}{J'} \ = \ \crossinv{f_1}{f_2}{f_4'}{f_3}{J'} \ = \ \cross{f_1}{f_2}{f_3}{f_4'}{J'},
 \]
 where we use \eqref{R1} for the last equality. We conclude that
 \[
  \cross{e_1}{e_2}{e_3}{e_4}{J} \ = \ \cross{e_1}{e_2}{e_3}{e_4}{J'} \ = \ \cross{e'_1}{e'_2}{e'_3}{e'_4}{J'},
 \]
 as desired. This completes the proof of the theorem.
\end{proof}

\begin{cor}\label{cor: foundation as a quotient of a tensorproduct of Us}
 The foundation $F_M$ of a matroid $M$ is naturally isomorphic to a quotient
 \[
  F_M \ \simeq \ \past{\bigg(\hspace{-5pt}\bigotimes_{\substack{N\to M\\ \text{ of type $U^2_4$}}} \hspace{-5pt} F_N \ \bigg)}{S}
 \]
 of a tensor product of foundations $F_N\simeq\U$, where the set $S$ is generated by the relations of type \eqref{R-} in the presence of an $F_7$ or $F_7^\ast$-minor and of types \eqref{R3}--\eqref{R5} that are induced by embedded minors $M\minor IJ\to M$ on at most $6$ elements $\{e_1,\dotsc,e_6\}=E-(I\cup J)$.
\end{cor}

\begin{proof}
 By Theorem \ref{thm: presentation of foundations in terms of bases}, the foundation is generated by the universal cross ratios $\cross{e_1}{e_2}{e_3}{e_4}{J}$ of $M$, which are the images $\cross{e_1}{e_2}{e_3}{e_4}{J}=\varphi_{M\minor IJ}\Big(\cross{e_1}{e_2}{e_3}{e_4}{}\Big)$ of the universal cross ratios $\cross{e_1}{e_2}{e_3}{e_4}{}$ of minors $N=M\minor IJ$ on $4$ elements $e_1,\dotsc,e_4$; cf.\ Proposition \ref{prop: minor embeddings induce morphisms between foundations}. The morphisms $\varphi_{M\minor IJ}:F_N\to F_M$ testify that all relations of $F_N$ also hold in $F_M$, and therefore we conclude that $F_M$ is of the form
 \[
  F_M \ \simeq \ \past{\bigg(\hspace{-5pt}\bigotimes_{\substack{N\to M\\ \text{with $\# E_N=4$}}} \hspace{-5pt} F_N \ \bigg)}{S}
 \]
 for some set of $3$-term relations $S$, where $E_N$ denotes the ground set of $N$. {\em A priori}, this holds if we include all relations \eqref{R-}--\eqref{R+} of Theorem \ref{thm: presentation of foundations in terms of bases} in $S$. To reduce this to the assertion of the corollary, we observe the following.
 
 If $N=M\minor IJ$ is a minor on $4$ elements that is not of type $U^2_4$, then $N$ is regular and $F_N=\Funpm$. Thus we can omit these factors from the tensor product. Note that \eqref{R0} assures that the cross ratios coming from such a minor are trivial in $F_M$. Therefore we can omit \eqref{R0} from $S$.
 
 Each of \eqref{Rs}, \eqref{R1}, \eqref{R2} and \eqref{R+} involve only cross ratios that come from the same $U^2_4$-minor $N=M\minor IJ$. Therefore the analogous relations hold already in $F_N$, and we can omit them from the set $S$.
 
 By Theorem \ref{thm: presentation of foundations in terms of bases}, the relation \eqref{R-} holds if $M$ has a minor of type $F_7$ or $F_7^\ast$. Each of the relations \eqref{R3}--\eqref{R5} involve cross ratios that come from the same minor on $5$ or $6$ elements. This shows all assertions of the corollary.
\end{proof}


\subsection{A presentation of the foundation by embedded minors}
\label{subsection: presentation in terms of embedded minors}

Let $N=M\minor IJ$ and $N'=M\minor{I'}{J'}$ be two embedded minors of a matroid $M$. If $I'\subset I$ and $J'\subset J$, then $N=N'\minor{(I-I')}{(J-J')}$ is an embedded minor of $N'$. We write $\iota:N\to N'$ for the inclusion as embedded minors and $\iota_\ast:F_N\to F_{N'}$ for the induced morphism between the respective foundations.

\begin{thm}\label{thm: presentation of the foundation by embedded minors}
 Let $M$ be a matroid with foundation $F_M$. Let $\cE$ be the collection of all embedded minors $N=M\minor IJ$ of $M$ on at most $7$ elements with the following properties:
 \begin{itemize}
  \item if $N$ has at most $6$ elements, then it contains a minor of type $U^2_4$;
  \item if $N$ has exactly $6$ elements, then it contains two parallel elements;
  \item if $N$ has $7$ elements, then it is isomorphic to $F_7$ or $F_7^\ast$.
 \end{itemize}
 Then
 \[
  F_M \ \simeq \ \past{\bigg(\bigotimes_{N\in\cE} F_N \bigg)}S,
 \]
 where the set $S$ is generated by the relations $a=\iota_\ast(a)$ for every inclusion $\iota:N\to N'$ of embedded minors $N$ and $N'$ in $\cE$.
\end{thm}

\begin{proof}
 It is clear that the morphisms $\varphi_{M\minor IJ}:F_{M\minor IJ}\to F_M$ from Proposition \ref{prop: minor embeddings induce morphisms between foundations} induce a canonical morphism $\past{\Big(\bigotimes_{N\in\cE} F_N \Big)}S\to F_M$, and since $\cE$ contains all embedded $U^2_4$-minors of $M$, this morphism is surjective. Thus we are left with showing that $S$ contains all defining relations of $M$.
 
 Let us define $\cE_i=\{N\in \cE\mid \# E_N=i\}$ for $i=4,\dotsc,7$ where $E_N$ denotes the ground set of the embedded minor $N$. Then $\cE=\cE_4\cup\dotsc\cup\cE_7$. The set $\cE_4$ consists of the embedded $U^2_4$-minors of $M$, and thus 
 \[
  F_M\simeq \past{\bigg(\bigotimes_{N\in\cE_4} F_N \ \bigg)}{S'}
 \]
 by Corollary \ref{cor: foundation as a quotient of a tensorproduct of Us}, where $S'$ contains all relations of types \eqref{R-} (in the presence of an $F_7$ or $F_7^\ast$-minor) and \eqref{R3}--\eqref{R5}.
 
 The relations \eqref{R3} and \eqref{R4} stem from embedded minors $N=M\minor IJ$ on $5$ elements, and these relations involve a nondegenerate cross ratio only if $N$ contains a $U^2_4$-minor, i.e.\ $N\in\cE_5$. Thus \eqref{R3} and \eqref{R4} can be replaced by tensoring with $F_N$ and including the relations $a=\iota_\ast(a)$ for every minor embedding $\iota:N'=N\minor{I'}{J'}\to N$ with $N'\in\cE_4$.
 
 Similarly, \eqref{R5} stems from embedded minors $N=M\minor IJ$ on $6$ elements with two parallel elements, and involves a nondegenerate cross ratio only if $N$ contains a $U^2_4$-minor, i.e.\ $N\in\cE_6$. Thus \eqref{R5} can be replaced by tensoring with $F_N$ and including the relations $a=\iota_\ast(a)$ for every minor embedding $\iota:N'=N\minor{I'}{J'}\to N$ with $N'\in\cE_4$.
  
 The set $\cE_7$ consists of all embedded minors of types $F_7$ and $F_7^\ast$. Since $F_{F_7}=F_{F_7^\ast}=\F_2$ and $\past{P}{\gen{1=-1}}\simeq P\otimes\F_2$ for every pasture $P$, we can replace the relation \eqref{R-} by $-\otimes F_N$ if $N\in \cE_7$. This recovers all relations in $S'$ and completes the proof.
\end{proof}


\section{The structure theorem}
\label{section: structure theorem}

In this section, we prove the central result of this paper, Theorem \ref{thm: structure theorem for matroids without large uniform minors}, which asserts that the foundation of a matroid $M$ without large uniform minors is isomorphic to a tensor product of finitely many copies of the pastures $\U$, $\D$, $\H$, $\F_3$ and $\F_2$.

This is done by first showing that in the absence of large uniform minors, the tip and cotip relations are of a particularly simple form, which eventually leads to the conclusion that the foundation of $M$ is the tensor product of quotients of $\U$ by automorphism groups, and possibly $\F_2$. The quotients of $\U$ by automorphisms are precisely $\U$, $\D$, $\H$ and $\F_3$.


\subsection{Foundations of matroids on \texorpdfstring{$5$}{5} elements}
\label{subsection: foundations of matroids on 5 elements}

By Theorem \ref{thm: presentation of the foundation by embedded minors}, the foundation of a matroid is determined completely by its minors on at most $5$ elements and the embedded minors on $6$ with two parallel elements.

In this section, we will determine the foundations of all matroids on at most $5$ elements. Most of these matroids are regular and have foundation $\Funpm$ by \cite[Thm.\ 7.33]{Baker-Lorscheid18}. There is only a small number of non-regular matroids on at most $5$ elements, which we will inspect in detail.

Let $0\leq r\leq n\leq 5$ and $M$ be a matroid of rank $r$ on $E=\{1,\dotsc,n\}$.

\subsubsection{Regular matroids}
\label{subsubsection: regular matroids on 5 elements}

A matroid $M$ is regular if and only if there is no nontrivial cross ratio, which is the case if and only if the matroid $M$ does not contain any minor of type $U^2_4$. 

This is the case in exactly one of the following situations: \textit{(a)} $r\in\{0,1,n-1,n\}$; \textit{(b)} $n=4$, $r=2$ and $M$ is not uniform; \textit{(c)} $n=5$, $r=2$ and $M\backslash i$ is not uniform for every $i\in E$; \textit{(d)} $n=5$, $r=3$ and $M/i$ is not uniform for every $i\in E$.

\subsubsection{Matroids with exactly one embedded \texorpdfstring{$U^2_4$}{U^2_4}-minor}
\label{subsubsection: matroids with exactly one embedded U24-minor}

There are several isomorphism classes of matroids with exactly one $U^2_4$-minor, which we list in the following. 

Since the tip and cotip relations involve cross ratios from different embedded $U^2_4$-minors, they do not appear for matroids with only one embedded $U^2_4$-minor.

If $n=4$, then there is exactly one such matroid, namely $M=U^2_4$ itself, which has foundation $\U$ by Proposition \ref{prop: foundation of U24}.

\begin{prop}\label{prop: foundation of matroids with one embedded U24-minor}
 Let $M$ be a matroid on $5$ elements with exactly one embedded $U^2_4$-minor. Then $M$ is isomorphic to $U^2_4\oplus N$ where $N$ is a matroid on $1$ element. The foundation of $M$ is isomorphic to $\U$.
\end{prop}

\begin{proof}
 In order to have an $U^2_4$-minor, $M$ must have rank $2$ or $3$. Since the embedded minors $N\to M$ of $M$ correspond bijectively to the embedded minors $N^\ast\to M^\ast$ and since $U^2_4$ is self-dual, the matroids $M$ and $M^\ast$ have the same number of $U^2_4$-minors. Once we have shown that every rank $2$-matroid with exactly one embedded $U^2_4$-minor is isomorphic to $U^2_4\oplus N$ for a matroid $N$ on one element, which has to be of rank $0$, then we can conclude that $M^\ast$ is isomorphic to $U^2_4\oplus N^\ast$. To complete this reduction to the rank $2$-case, we note that the foundation of $M^\ast$ is canonically isomorphic to the foundation of $M$, cf.\ Proposition \ref{prop: foundation of the dual matroid}.
 
 Assume that the rank $2$-matroid $M$ on $E=\{1,\dotsc,5\}$ has an embedded $U^2_4$-minor. After a permutation of $E$, we can assume that this embedded $U^2_4$-minor is $M\backslash 5=M\backslash\{5\}$, i.e.\ that all of the following $2$-subsets 
 \[
  \{1,2\}, \quad \{1,3\}, \quad \{1,4\}, \quad \{2,3\}, \quad \{2,4\} \quad\text{and}\quad \{3,4\}  
 \]
 of $E$ are bases. If these are all bases of $M$, then $5$ is a loop and $M$ is isomorphic to $U^2_4\oplus N$, as claimed.
 
 We indicate why $M$ cannot have more bases of the form $\{i,5\}$. If $M$ has exactly one additional basis element, say $\{1,5\}$, then the basis exchange property is violated by exchanging $1$ by an element of the basis $\{3,4\}$. The same reason excludes the possibility that $M$ has exactly two additional basis elements, say $\{1,5\}$ and $\{2,5\}$. If $M$ has $9$ or more basis elements, say all $2$-subsets of $E$ but possibly $\{4,5\}$, then both minors $M\backslash 4$ and $M\backslash 5$ are isomorphic to $U^2_4$. Thus in this case, $M$ has at least two embedded $U^2_4$-minors.
 
 This shows that $M$ has to be isomorphic to $U^2_4\oplus N$. Since $5$ is a loop, the conditions for the tip relations are not satisfied, which means that all relations stem from the unique embedded $U^2_4$-minor $M\backslash 5$. This shows that the foundation of $M$ is isomorphic to $F_{M\backslash 5}\simeq\U$, as claimed. 
\end{proof}

\subsubsection{Matroids with exactly two embedded \texorpdfstring{$U^2_4$}{U^2_4}-minors}
\label{subsubsection: matroids with exactly two embedded U24-minors}

If $M$ has two embedded $U^2_4$-minors, then the ground set must be $E=\{1,\dotsc,5\}$. As explained in Section \ref{subsubsection: matroids with exactly one embedded U24-minor}, $M$ must have rank $2$ or $3$ if $M$ has an $U^2_4$-minor. We will show that if $M$ has exactly two embedded $U^2_4$-minors, then it must be isomorphic to the following matroid, or its dual.

\begin{df}
 We denote by $C_5$ the rank $3$-matroid on $E=\{1,\dotsc,5\}$ whose set of bases is $\binom E3-\{3,4,5\}$.
\end{df}

\begin{prop}\label{prop: foundation of C_5}
 A matroid $M$ on $5$ elements has exactly two embedded $U^2_4$-minors if and only if $M$ is isomorphic to either $C_5$ or its dual. The cross ratios of $C_5$ satisfy
 \[
  \cross ijk4{5} \ = \ \cross ijk5{4},
 \]
 and the cross ratios of $C_5^\ast$ satisfy
 \[
  \cross ijk4{} \ = \ \cross ijk5{}
 \]
 for all identifications $\{i,j,k\}=\{1,2,3\}$. The foundations of both $C_5$ and $C_5^\ast$ are isomorphic to $\U$.
\end{prop}

We illustrate all non-degenerate cross ratios of $C_5^\ast$ and their relations in Figure \ref{fig: cross ratios of C_5}.

\begin{figure}[htb]
 \centering
 \leavevmode
 \beginpgfgraphicnamed{tikz/fig1}
  \begin{tikzpicture}[x=1cm,y=1cm]
   \draw[line width=3pt,color=gray!25,fill=gray!25,bend angle=20] (60:3.4) to[bend left] (180:3.4) to[bend left] (300:3.4) to[bend left] cycle;
   \draw[line width=3pt,color=gray!30,fill=gray!30,bend angle=20] (0:3.4) to[bend left] (120:3.4) to[bend left] (240:3.4) to[bend left] cycle;
   \draw[line width=3pt,color=gray!25,dotted,bend angle=20] (60:3.4) to[bend left] (180:3.4) to[bend left] (300:3.4) to[bend left] cycle;
   \draw[line width=3pt,color=gray!40,fill=gray!40,bend angle=20] (60:1.4) to[bend left] (180:1.4) to[bend left] (300:1.4) to[bend left] cycle;
   \draw[line width=3pt,color=gray!50,fill=gray!50,bend angle=20] (0:1.4) to[bend left] (120:1.4) to[bend left] (240:1.4) to[bend left] cycle;
   \draw[line width=3pt,color=gray!40,dotted,bend angle=20] (60:1.4) to[bend left] (180:1.4) to[bend left] (300:1.4) to[bend left] cycle;
   \node (-1) at (0,0) {\small \textcolor{black!60}{{$\mathbf{-1\ }$}}};
   \node (1234) at (-1, 1.73) {$\cross 1234{}$};
   \node (1324) at ( 1, 1.73) {$\cross 1324{}$};
   \node (3124) at ( 2, 0   ) {$\cross 3124{}$};
   \node (3214) at ( 1,-1.73) {$\cross 3214{}$};
   \node (2314) at (-1,-1.73) {$\cross 2314{}$};
   \node (2134) at (-2, 0   ) {$\cross 2134{}$};
   \node (1235) at (-2, 3.46) {$\cross 1235{}$};
   \node (1325) at ( 2, 3.46) {$\cross 1325{}$};
   \node (3125) at ( 4, 0   ) {$\cross 3125{}$};
   \node (3215) at ( 2,-3.46) {$\cross 3215{}$};
   \node (2315) at (-2,-3.46) {$\cross 2315{}$};
   \node (2135) at (-4, 0   ) {$\cross 2135{}$};
   \path (1234) edge node[auto] {$+$} (1324);
   \path (1324) edge node[auto] {$*$} (3124);
   \path (3124) edge node[auto] {$+$} (3214);
   \path (3214) edge node[auto] {$*$} (2314);
   \path (2314) edge node[auto] {$+$} (2134);
   \path (2134) edge node[auto] {$*$} (1234);
   \path (1235) edge node[auto] {$+$} (1325);
   \path (1325) edge node[auto] {$*$} (3125);
   \path (3125) edge node[auto] {$+$} (3215);
   \path (3215) edge node[auto] {$*$} (2315);
   \path (2315) edge node[auto] {$+$} (2135);
   \path (2135) edge node[auto] {$*$} (1235);
   \path (1234) edge node[auto] {$=$} (1235);
   \path (1324) edge node[auto,swap] {$=$} (1325);
   \path (3124) edge node[auto] {$=$} (3125);
   \path (3214) edge node[auto] {$=$} (3215);
   \path (2314) edge node[auto,swap] {$=$} (2315);
   \path (2134) edge node[auto,swap] {$=$} (2135);
  \end{tikzpicture}
 \endpgfgraphicnamed
 \caption{The cross ratios of $C_5^\ast$ and their relations}
 \label{fig: cross ratios of C_5}
\end{figure}
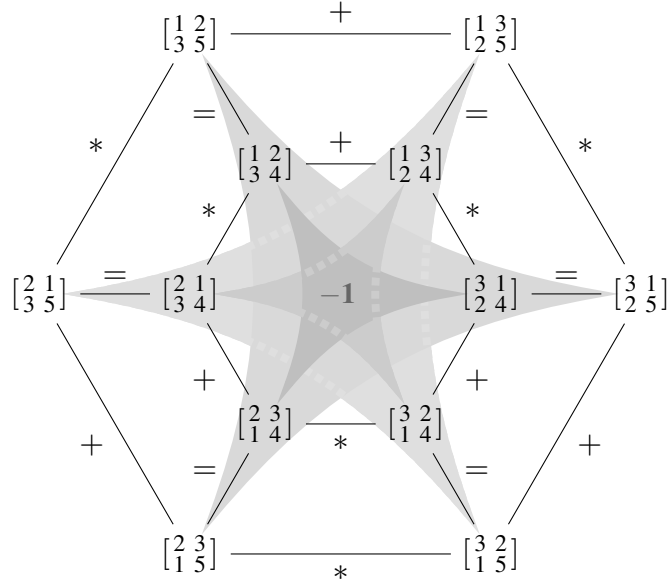

\begin{proof}
 In the proof of Proposition \ref{prop: foundation of matroids with one embedded U24-minor}, we saw that $C_5$ has at least two embedded $U^2_4$-minors, which correspond to the $U^2_4$-minors $C_5\backslash 4$ and $C_5\backslash 5$. All other minors of rank $2$ on $4$ elements of $C_5$ are of the form $C_5\backslash i$ for $i\in\{1,2,3\}$. But since $\{4,5\}$ is not a basis of $C_5$, none of these minors is isomorphic to $U^2_4$. This shows that $C_5$ has exactly two embedded $U^2_4$-minors, as has every matroid $M$ that is isomorphic to $C_5$.
 
 Conversely, assume that $M$ is a matroid on $5$ elements with exactly two embedded $U^2_4$-minors. Since duality preserves $U^2_4$-minors, can assume that $M$ is of rank $2$. After a permutation of $E$, we can assume that these two embedded $U^2_4$-minors are $M\backslash 4$ and $M\backslash 5$. Thus all of the $2$-subsets
 \[
  \{1,2\}, \ \{1,3\}, \ \{1,4\}, \ \{1,5\}, \ \{2,3\}, \ \{2,4\}, \ \{2,5\}, \ \{3,4\} \ \text{and} \ \{3,5\}
 \]
 are bases. If $\{4,5\}$ was also a basis of $M$, then $M$ would be the uniform matroid $U^2_5$, which has five $U^2_4$ minors $U^2_5\backslash i$ for $i=1,\dotsc,5$. Thus $M$ is isomorphic to $C_5$. This proves our first claim.
 
 Let us choose an identification $\{i,j,k\}=\{1,2,3\}$. The tip relation \eqref{R3} in Theorem \ref{thm: presentation of foundations in terms of bases} with tip $\{i,j\}$ and cyclic orientation $(k,4,5)$ for $C_5$ is
 \[
  \cross ijk4{} \ \cdot \ \cross ij45{} \ \cdot \ \cross ij5k{} \ = \ 1.
 \]
 Since $\cross ij45{}=1$ is degenerate, we obtain the claimed relation
 \[
  \cross ijk4{} \ = \ \crossinv ij5k{} \ = \ \cross ijk5{},
 \]
 where the second equality is relation \eqref{R1}. Similarly, the cotip relation \eqref{R3} with cotip $\{i,j\}$ and cyclic orientation $(k,4,5)$ for $C_5^\ast$ is
 \[
  \cross ijk4{5} \ \cdot \ \cross ij45{k} \ \cdot \ \cross ij5k{4} \ = \ 1.
 \]
 Since $\cross ij45{k}=1$ is degenerate, we obtain the claimed relation
 \[
  \cross ijk4{5} \ = \ \crossinv ij5k{4} \ = \ \cross ijk5{4}.
 \]
 
 Since $C_5^\ast$ is a parallel extension of $U^2_4$, the foundation of $C_5^\ast$ is $\U$ by Proposition \ref{prop: minor embeddings induce morphisms between foundations}, which concludes the proof.
\end{proof}

\subsubsection{Matroids with five embedded \texorpdfstring{$U^2_4$}{U^2_4}-minors}
\label{subsubsection: matroids with five embedded U24-minors}

The only matroids on at most five elements that do not appear among the previous cases with at most two embedded $U^2_4$-minors are the uniform matroids $U^2_5$ and $U^3_5$, which have five embedded $U^2_4$-minors.

For completeness, we describe their foundations. However, we postpone the proof to a sequel to this paper where we develop more sophisticated methods to calculate the foundations of matroids. Note that the results of this first part are independent from the following result since we consider matroids without large uniform minors.

\begin{prop}\label{prop: the foundation of U^2_5}
 The foundations of $U^2_5$ and $U^3_5$ are isomorphic to 
 \[
  \pastgen{\Funpm}{x_1,\dotsc,x_5}{x_i+x_{i-1}x_{i+1}-1 \, | \, i=1,\dotsc,5 }
 \]
 where $x_0=x_5$ and $x_6=x_1$.
\end{prop}


\subsection{Symmetry quotients}
\label{subsection: symmetry quotients}

The classification of foundations of matroids on up to five elements in section \ref{subsection: foundations of matroids on 5 elements} shows that in a matroid without large uniform minors, all relations between cross ratios of different embedded $U^2_4$-minors arise from minors of type $C_5$ or $C_5^\ast$. Proposition \ref{prop: foundation of C_5} shows that these types of minors identify the two hexagons of cross ratios, which implies an identification of two copies of the near-regular partial field $\U$; cf.\ Figure \ref{fig: cross ratios of C_5}. The same happens for relations of type \ref{R5}: they identify two copies of $\U$.

It can, and it will, happen that a matroid contains a chain of such minors, which creates a self-identification of the cross ratios belonging to an embedded $U^2_4$-minor of $M$. By Proposition \ref{prop: foundation of C_5}, this self-identification must respect the relations between the cross ratios in each hexagon, and induces an automorphism of $\U$. Therefore we are led to study the quotients of $\U$ by such automorphisms.

\subsubsection{Automorphisms of the near-regular partial field}
\label{subsubsection: automorphisms of the near-regular partial field}

In the following, we determine all automorphisms of the near-regular partial field $\U=\pastgen\Funpm{x,y}{x+y=1}$. By Lemma \ref{lemma: morphisms from U to P are determined by pairs of fundamental elements}, it suffices to determine the images of $x$ and $y$ to describe an automorphism of $\U$. A result equivalent to the following is also proved in \cite[Lemma 4.4]{Pendavingh-vanZwam10b}.

\begin{lemma}\label{lemma: fundamental elements of U}
 The elements of the form $z+z'-1$ in the nullset $N_\U$ of $\U$ with $z,z'\in\U^\times$ are 
 \[
  x+y-1, \qquad x^{-1}-x^{-1}y-1 \qquad \text{and} \qquad y^{-1}-xy^{-1}-1.
 \]
 Thus the fundamental elements of $\U$ are $x,\ y,\ x^{-1},\ -x^{-1}y,\ y^{-1},\ -xy^{-1}$.
\end{lemma}

\begin{proof}
 Note that the only element $z$ with $z+1-1=0$ is $z=0$. Thus to find all fundamental elements, it suffices to search for relations of the form $z+z'-1\in N_\U$ with $z,z'\in \U^\times$. Since $N_\U$ is generated by $1-1+0$ and $x+y-1$, and since all terms have to be nonzero and at least one term has to be equal to $-1$ to find a relation for fundamental elements, we find exactly three relations of the form $z+z'-1=0$, which are
 \[
  x+y-1, \qquad x^{-1}-x^{-1}y-1 \qquad \text{and} \qquad y^{-1}-xy^{-1}-1.
 \]
 Thus the claim of the lemma.
\end{proof}

\begin{prop}\label{prop: automorphisms of U}
 The associations
 \[
  \begin{array}{cccc}
   \rho: & \U  & \longrightarrow & \U  \\
         & x & \longmapsto     & y^{-1} \\
         & y & \longmapsto     & -xy^{-1}
  \end{array}
  \qquad 
  \begin{array}{c}
   \text{and}\\ \\ {}\ 
  \end{array}
  \qquad
  \begin{array}{cccc}
   \sigma: & \U  & \longrightarrow & \U,  \\
           & x & \longmapsto     & y \\
           & y & \longmapsto     & x 
  \end{array}
 \]
 define automorphisms of $\U$ that generate the automorphism group of $\U$ and satisfy the relations $\rho^3=\sigma^2=(\rho\sigma)^2=\id$. In particular, $\Aut(\U)\simeq S_3$.
\end{prop}

\begin{proof}
 By Lemma \ref{lemma: fundamental elements of U}, both $(y^{-1},-xy^{-1})$ and $(y,x)$ are pairs of fundamental elements in $\U$. Thus, by Lemma \ref{lemma: morphisms from U to P are determined by pairs of fundamental elements}, $\rho$ and $\sigma$ define morphisms from $\U$ to $\U$. Since $\rho^3(x)=x$ and $\rho^3(y)=y$, we conclude that $\rho$ defines a group automorphism of $\U^\times$ of order $3$. Similarly, $\sigma$ defines a group automorphism of $\U^\times$ of order $2$. The relation $(\rho\sigma)^2=\id$ can be easily verified by evaluation on $x$ and $y$.

 We conclude that the automorphism group of $\U$ contains $\gen{\rho,\sigma}\simeq S_3$ as a subgroup. By Lemma \ref{lemma: fundamental elements of U}, $\U$ contains precisely $6$ fundamental elements, which implies easily that $\Aut(\U)$ is generated by $\rho$ and $\sigma$.
\end{proof}

\begin{rem}\label{rem: the automorphisms of U are the symmetries of the hexagon}
 It follows from Lemma \ref{lemma: fundamental elements of U} that the isomorphism $F_{U^2_4}\to \U$ from Proposition \ref{prop: foundation of U24} maps the cross ratios of $U^2_4$ bijectively to the fundamental elements of $\U$. We can arrange these fundamental elements in a hexagon
 \[
 \beginpgfgraphicnamed{tikz/fig2}
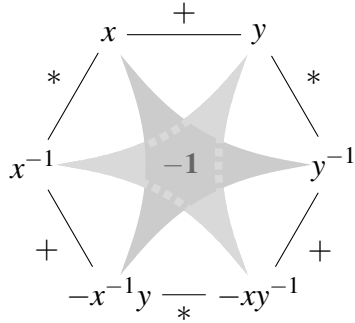
  \begin{tikzpicture}[x=1cm,y=1cm]
   \draw[line width=3pt,color=gray!30,fill=gray!30,bend angle=20] (60:1.4) to[bend left] (180:1.4) to[bend left] (300:1.4) to[bend left] cycle;
   \draw[line width=3pt,color=gray!40,fill=gray!40,bend angle=20] (0:1.4) to[bend left] (120:1.4) to[bend left] (240:1.4) to[bend left] cycle;
   \draw[line width=3pt,color=gray!30,dotted,bend angle=20] (60:1.4) to[bend left] (180:1.4) to[bend left] (300:1.4) to[bend left] cycle;
   \node (-1) at (0,0) {\small \textcolor{black!60}{{$\mathbf{-1\ }$}}};
   \node (1234) at (-1, 1.73) {$x$};
   \node (1324) at ( 1, 1.73) {$y$};
   \node (3124) at ( 2, 0   ) {$y^{-1}$};
   \node (3214) at ( 1,-1.73) {$-xy^{-1}$};
   \node (2314) at (-1,-1.73) {$-x^{-1}y$};
   \node (2134) at (-2, 0   ) {$x^{-1}$};
   \path (1234) edge node[auto] {$+$} (1324);
   \path (1324) edge node[auto] {$*$} (3124);
   \path (3124) edge node[auto] {$+$} (3214);
   \path (3214) edge node[auto] {$*$} (2314);
   \path (2314) edge node[auto] {$+$} (2134);
   \path (2134) edge node[auto] {$*$} (1234);
  \end{tikzpicture}
 \endpgfgraphicnamed
 \]
 in the same way as we arrange the cross ratios in Figure \ref{fig: hexagon}. It follows from Proposition \ref{prop: automorphisms of U} that the automorphisms of $\U$ correspond bijectively to the symmetries of this hexagon that preserve the edge labels and the inner triangles.
\end{rem}


\subsubsection{Classification of the symmetry quotients of \texorpdfstring{$\U$}{U}}
\label{subsubsection: classificaction of the symmetry quotients of U}

A \emph{symmetry quotient of $\U$} is the quotient of $\U$ by a group of automorphisms. More precisely, if $H$ is a subgroup of $\Aut(\U)$, then the \emph{quotient of $\U$ by $H$} is 
\[
 \U/H \ = \ \past\U{\{ \, x=\tau(x),y=\tau(y) \, | \, \tau\in H \, \}}.
\]
In fact, we have $\U/H=\past\U{\{x=\tau(x),y=\tau(y)|\tau\in S\}}$ if $S$ is a set of generators of $H$. 

Recall from section~\ref{subsubsection: examples} that $\F_3=\past{\Funpm}{\{1+1+1\}}$,
\[
 \D \ = \ \pastgen\Funpm{z}{z+z-1} \qquad \text{and} \qquad \H \ = \ \pastgen\Funpm{z}{z^3+1,\, z-z^2-1}.
\]
Note that this implies that $z^3=-1$ and $z^6=1$ in $\H$.

\begin{prop}\label{prop: symmetry quotients of U}
 The symmetry quotients of $\U$ are, up to isomorphism, 
 \[
  \U/\gen\id \ \simeq \ \U, \qquad \U/\gen\sigma \ \simeq \ \D, \qquad \U/\gen\rho \ \simeq \ \H, \qquad \U/\gen{\rho,\sigma} \ \simeq \ \F_3.
 \]
\end{prop}

\begin{proof}
 In the following, we show that the quotients of $\U$ by different subgroups $H$ of $\Aut(\U)\simeq S_3$ are exactly the pastures $\U$, $\D$, $\H$ and $\F_3$, up to isomorphism. Clearly $\U=\U/\gen\id$ is the quotient of $\U$ by the trivial subgroup.
 
 Note that if $H'$ is a subgroup conjugate to $H$, i.e.\ $H'=\tau H\tau^{-1}$ for some $\tau\in\Aut(\U)$, then the quotient of $\U$ by $H'$ equals the quotient of $\tau(\U)=\U$ by $H$. This means that it suffices to determine the isomorphism classes of the quotients of $\U$ by the groups $\gen\sigma$, $\gen\rho$ and $\Aut(\U)=\gen{\rho,\sigma}$, which represent all conjugacy classes of nontrivial subgroups of $\Aut(\U)$.
 
 Let $H=\gen \sigma$. We denote the residue classes of $x$ and $y$ in $\U/\gen\sigma$ by $\bar x$ and $\bar y$, respectively. We claim that the association
 \[
  \begin{array}{cccc}
     f: & \U/\gen\sigma & \longrightarrow & \D \\
        & \bar x        & \longmapsto     & z  \\
        & \bar y        & \longmapsto     & z 
  \end{array}
 \]
 defines an isomorphism of pastures. We begin with the verification that $f$ defines a morphism. The map $\hat f:\U\to\D$ with $\hat f(x)=\hat f(y)=z$ is a morphism, since the generator $x+y-1$ of the nullset of $\U$ is mapped to $z+z-1$, which is in the nullset of $\D$. Since $\hat f\big(\sigma(x)\big)=z=\hat f(x)$ and $\hat f\big(\sigma(y)\big)=z=\hat f(y)$, the morphism $\hat f$ induces a morphism $f:\U/\gen\sigma\to\D$ by the universal property of the quotient $\U/\gen\sigma=\past{\U}{\{\sigma(x)=y,\sigma(y)=x\}}$, cf.\ Proposition \ref{prop: universal property of algebras and quotients}.
 
 We define the inverse to $f$ as the association $g:z\mapsto \bar x$. This defines a multiplicative map since $\D^\times$ is freely generated by $z$. Since
 \[
  g(z)+g(z)-1 \ = \ \bar x+\bar x-1 \ = \ \bar x+\bar y-1
 \]
 is null in $\U/\gen\sigma$, this defines a morphism $g:\D\to\U/\gen\sigma$. It is obvious that $g$ is an inverse to $f$, which shows that $f$ is an isomorphism.
 
 We continue with the automorphism group $H=\gen\rho$. We claim that the association
 \[
  \begin{array}{cccc}
     f: & \U/\gen\rho & \longrightarrow & \H  \\
        & \bar x      & \longmapsto     & z   \\
        & \bar y      & \longmapsto     & -z^2 
  \end{array}
 \]
 defines an isomorphism of pastures. We begin with the verification that $f$ defines a morphism. The map $\hat f:\U\to\H$ with $\hat f(x)=z$ and $\hat f(y)=-z^2$ is a morphism, since the generator $x+y-1$ of the nullset of $\U$ is mapped to $z-z^2-1$, which is in the nullset of $\H$. Since $\hat f\big(\rho(x)\big)=\hat f(y^{-1})=z=\hat f(x)$ and $\hat f\big(\rho(y)\big)=\hat f(-xy^{-1})=-z^2=\hat f(y)$, the morphism $\hat f$ induces a morphism $f:\U/\gen\rho\to\D$ by the universal property of the quotient $\U/\gen\rho=\past{\U}{\{\rho(x)=y,\rho(y)=x\}}$.
 
 We define the inverse of $f$ as follows. Let $\hat g:\Funpm\gen z\to\U/\gen\rho$ be the morphism that maps $z$ to $\bar x$. The defining relations of $\U/\gen\rho$ are $\bar x=\bar y^{-1}$ and $\bar y=-\bar x\bar y^{-1}$. Thus
 \[
  \hat g(z^3)+\hat g(1) \ = \ \bar x^3 + 1 \ = \ \bar y^{-2}\bar x+1 \ = \ -\bar x^{-1}\bar y\bar y^{-1}\bar x+1 \ = \ -1+1,
 \]
 which is in the nullset of $\U/\rho$. Since $z^3=-1$ in $\H$, we have $-z^2=z^{-1}$ and thus
 \[
  \hat g(z)+\hat g(-z^2)-1 \ = \ \bar x+\bar x^{-1}-1 \ = \ \bar x+\bar y-1,
 \]
 which is also in the nullset of $\U/\gen\rho$. This shows that the morphism $\hat g$ defines a morphism $g:\H\to\U/\gen\rho$,
  which is obviously inverse to $f$.
 
 Finally we show that $\U/\gen{\rho,\sigma}$ is isomorphic to $\F_3$. Since $\U/\gen{\rho,\sigma}\simeq\big(\U/\gen\rho\big)/\gen\sigma$, it suffices to show that the association
 \[
  \begin{array}{cccc}
     f: & \H/\gen\sigma & \longrightarrow & \F_3     \\
        & \bar z        & \longmapsto     & -1 
  \end{array}
 \]
is an isomorphism. Since $\sigma(z)=\sigma(\bar x)=\bar y=z^{-1}$ and $f(\bar z)=f(\bar z^{-1})$, and since $f(z^6)=(-1)^6=1=f(1)$, the assignment $f(\bar z)=-1$ extends to a multiplicative map. Since $f(z^3)+f(1)=(-1)^3+1=-1+1$ and $f(z)+f(-z^2)-1=-1-1-1$ are null in $\F_3$, the map $f$ is a morphism. Note that in $\H/\gen\sigma$, we have $\bar z^3=-1$ and $\bar z=\bar z^{-1}$, and thus $\bar z=-1$. We conclude that the assignment $g:1\mapsto 1=-\bar z$ defines a morphism $g:\F_3\to\H/\gen\sigma$, since 
 \[
  g(1)+g(1)+g(1) \ = \ 1+1+1 \ = \ -\big(\bar z-\bar z^2-1\big)
 \]
 is null in $\H/\gen\sigma$. It is clear that $g$ is an inverse of $f$, which shows that $f$ is an isomorphism. This concludes the proof of the proposition.
\end{proof}


\subsection{The structure theorem for matroids without large uniform minors}
\label{subsection: the structure theorem for matroids without large uniform minors}

We are prepared to prove the central result of this paper. In the following, the empty tensor product stands for the initial object in $\Pastures$, which is $\Funpm$.

\begin{thm}\label{thm: structure theorem for matroids without large uniform minors}
 Let $M$ be a matroid without large uniform minors and $F_M$ its foundation. Then 
 \[
  F_M \ \simeq \ F_1\otimes \dotsb \otimes F_r
 \]
 for some $r\geq 0$ and pastures $F_1,\dotsc,F_r\in\{\U,\D,\H,\F_3,\F_2\}$.
\end{thm}

\begin{proof}
 Let $\cE$ be the collection of embedded minors $N$ of $M$ from Theorem \ref{thm: presentation of the foundation by embedded minors}. Then
 \[
  F_M \ \simeq \ \past{\bigg(\bigotimes_{N\in\cE} F_N \bigg)}S,
 \]
 where the set $S$ is generated by the relations $a=\iota_\ast(a)$ for every inclusion $\iota:N\to N'$ of embedded minors $N$ and $N'$ in $\cE$.

 From the analysis in section \ref{subsection: foundations of matroids on 5 elements}, it follows that the foundation $F_N$ of every embedded minor $N$ of $M$ with at most $5$ elements is either $\Funpm$ or $\U$, where we use the assumption that $M$ is without minors of types $U^2_5$ and $U^3_5$. A matroid with foundation $\Funpm$ is regular and has thus no minor of type $U^2_4$. We conclude that every embedded minor in $\cE$ on at most $5$ elements has foundation $\U$. 
 
 If an embedded minor $N$ in $\cE$ has $6$ elements, and thus two of them are parallel, then deleting one of these parallel elements yields an embedded minor $N'=N\backslash e$ of $N$, and the induced morphism $F_{N'}\to F_N$ is an isomorphism. Thus also every embedded minor in $\cE$ with $6$ elements has foundation $\U$.
 
 Since neither $F_7$ nor $F_7^\ast$ contains a minor of type $\U$, an embedded minor $N$ in $\cE$ with $7$ elements cannot contain another embedded minor $N'$ in $\cE$. Consequently the isomorphism of Theorem \ref{thm: presentation of the foundation by embedded minors} implies that 
 \[
  F_M \ \simeq \ \bigotimes_{N\in \cE_7} F_N \otimes \past{\bigg(\bigotimes_{N\in\cE'} F_N \bigg)}{S'},
 \]
 where $\cE_7$ is the subset of $\cE$ that contains all embedded minors with $7$ elements, $\cE'$ is the subset of $\cE$ that contains all embedded minors with at most $6$ elements and $S$ is the set generated by the relations $a=\iota_\ast(a)$ for every inclusion $\iota:N\to N'$ of embedded minors $N$ and $N'$ in $\cE'$.
 
 By what we have seen, an inclusion $N\to N'$ of embedded minors in $\cE'$ is an isomorphism, and either foundation is isomorphic to $\U$. Thus all identifications in $S'$ stem from isomorphisms between some factors $F_N$ of the tensor product. What can, and does, happen is that a chain of such isomorphisms imposes a self-identification of a factor $F_N\simeq\U$ with itself by a non-trivial automorphism. This leads to a symmetry quotient of $\U$, which is one of $\U$, $\D$, $\H$ and $\F_3$. Thus
 \[
  \past{\bigg(\bigotimes_{N\in\cE'} F_N \bigg)}{S'}
 \]
 is a tensor product of copies of $\U$, $\D$, $\H$ and $\F_3$. 
 
 This leaves us with the factors $F_N$ for $N\in\cE_7$. By Theorem \ref{thm: presentation of foundations in terms of bases}, we have $-1=1$, and all cross ratios are trivial since there are no $U^2_4$-minors. Thus $F_N\simeq \past\Funpm{\{1=-1\}}=\F_2$. This concludes  the proof of the theorem.
\end{proof}

Theorem \ref{thm: structure theorem for matroids without large uniform minors} can be reformulated as follows, which expresses the dependencies of the factors $F_i$ on $M$.

\begin{cor}\label{cor: alternative structure for matroids without large uniform minors}
 Let $M$ be a matroid without large uniform minors, $F_M$ its foundation. Then 
 \[
  F_M \ \simeq \ F_0\otimes F_1\otimes \dotsb \otimes F_r
 \]
 for a uniquely determined $r\geq0$ and uniquely determined pastures $F_0\in\{\Funpm,\F_2,\F_3,\K\}$ and $F_1,\dotsc,F_r\in\{\U,\D,\H\}$, up to a permutation of the indices $1,\dotsc,r$. We have $F_0=\F_2$ or $F_0=\K$ if and only if $M$ contains a minor of type $F_7$ or $F_7^\ast$.
\end{cor}

\begin{proof}
 By Theorem \ref{thm: structure theorem for matroids without large uniform minors}, the foundation $F_M$ of a matroid $M$ without large uniform minors is isomorphic to a tensor product of copies of $\U$, $\D$, $\H$, $\F_3$ and $\F_2$. 
 
 Since morphisms from $\F_2$ and $\F_3$ into other pastures are uniquely determined, if they exist, we conclude that $\F_2\otimes\dotsb\otimes\F_2=\F_2$ and $\F_3\otimes\dotsb\otimes\F_3=\F_3$. Thus the pasture
 \[
  \underbrace{\F_2\otimes\dotsb\otimes\F_2}_{r\text{ times}} \otimes \underbrace{\F_3\otimes\dotsb\otimes\F_3}_{s\text{ times}}
 \]
 is isomorphic to
 \[
  \Funpm\ \text{ if $r=s=0$;} \quad \F_2\ \text{ if $r>s=0$;} \quad \F_3\ \text{ if $s>r=0$;} \quad \F_2\otimes\F_3=\K\ \text{ if $r,s>0$;}
 \]
 cf.\ Example \ref{ex: product and coproduct of F2 and F3} for the equality $\F_2\otimes\F_3=\K$. This explains the list of possible isomorphism types for $F_0$. Since $\F_2$ appears as a factor of $F_M$ if and only if $M$ has a minor of type $F_7$ or $F_7^\ast$, this verifies the last claim of the corollary.
  
 It follows that $F_M$ is isomorphic to a tensor product of $F_0$ with pastures $F_1,\dotsc,F_r\in\{\U,\D,\H\}$. 
 
 We are left with establishing the uniqueness claims. To begin with, $F_0$ is uniquely determined by the presence or absence of the relations $1+1=0$ and $1+1+1=0$, which correspond to the relations $r>0$ and $s>0$, respectively, in our previous case consideration. Thus $F_0$ is uniquely determined.
 
 The factors $F_1,\dotsc,F_r$ are determined by the fundamental elements of $F_M$, as we explain in the following. Let $\iota_i:F_i\to\bigotimes F_j\simeq F_M$ be the canonical inclusion. By the construction of the tensor product, the nullset of $F_M$ consists of all terms of the form $d\iota_i(a)+d\iota_i(b)+d\iota_i(c)$ for some $i\in\{0,\dotsc,r\}$, $d\in \bigotimes F_j$ and $a,b,c\in F_i$ such that $a+b+c$ is in the nullset of $F_i$. The fundamental elements of $F_M$ stem from such equations for which $d\iota_i(a)$ and $d\iota_i(b)$ are nonzero and $d\iota_i(c)=-1$. Thus $d=-\iota_i(c)^{-1}=\iota_i(-c^{-1})$ is in the image of $\iota_i$, and therefore $d\iota_i(a)=\iota_i(-c^{-1}a)$ and $d\iota_i(b)=\iota_i(-c^{-1}b)$. Since $-c^{-1}a-c^{-1}b-1$ is in the nullset of $F_i$, we conclude that all fundamental elements in $F_M$ are of the form $\iota_i(z)$ for some $i$ and some fundamental element $z$ of $F_i$.
 
 To make a distinction between the different isomorphism types of the factors, we note that every fundamental element $x$ with relation $x+y-1=0$ gives rise to a set $\Big\{x,\, x^{-1},\, y,\, y^{-1},\, -x^{-1}y,\, -xy^{-1}\Big\}$ of fundamental elements. If these six fundamental elements come from a factor $F_i\simeq\U$, then they are pairwise different. If they come from a factor $F_i\simeq \D$, then 
 \[
  \Big\{x,\, x^{-1},\, y,\, y^{-1},\, -x^{-1}y,\, -xy^{-1}\Big\} \ = \ \Big\{x,\, y^{-1},\, -x^{-1}y\Big\}
 \]
 is a set with three distinct elements. If they come from a factor $F_i\simeq \D$, then 
 \[
  \Big\{x,\, x^{-1},\, y,\, y^{-1},\, -x^{-1}y,\, -xy^{-1}\Big\} \ = \ \Big\{x,y\Big\}
 \]
 is a set with two distinct elements. Note that if $F_0=\F_3$ or $F_0=\K$, then $x=-1$ is also a fundamental element, and in this case $x^{-1}=y=y^{-1}=-x^{-1}y=-xy^{-1}=-1$ are all equal. This shows that the number of factors of types $\U$, $\D$ and $\H$ are determined by the fundamental elements of $F_M$, which completes the proof of our uniqueness claims.
\end{proof}
 
 

\begin{rem}
 In a sequel to this paper, we will show that for all $r\geq0$ and $F_1,\dotsc,F_r\in\{\U,\D,\H,\F_3,\F_2\}$, there is a matroid $M$ without large uniform minors whose foundation is isomorphic to the tensor product $F_1\otimes\dotsb \otimes F_r$.
\end{rem}


\section{Applications}
\label{section: applications}

In this concluding part of the paper, we explain various applications of our central result Theorem \ref{thm: structure theorem for matroids without large uniform minors}. Along with some new results and strengthenings of known facts, we also present short conceptual proofs for a number of established theorems which illustrate the versatility of our structure theory for foundations.

The main technique in most of the upcoming proofs is the following. A matroid $M$ is representable over a pasture $P$ if and only there is a morphism from the foundation $F_M$ of $M$ to $P$. If $M$ is without large uniform minors, then we know by Theorem \ref{thm: structure theorem for matroids without large uniform minors} that $F_M$ is isomorphic to the tensor product of copies $F_i$ of $\U$, $\D$, $\H$, $\F_3$ and $\F_2$. Thus a morphism from $F_M$ to $P$ exists if and only there is a morphism from each $F_i$ to $P$, which in practice is quite easy to determine.

For reference in the later sections, we will provide some general criteria for such morphisms in the following result, and list the outcome for a series of prominent pastures in Table \ref{table: morphisms between pastures}. 

\begin{lemma}\label{lemma: conditions for morphisms from UDHF3F2 to a pasture}
 Let $P$ be a pasture.
 \begin{enumerate}
  \item \label{morphism1} There is a morphism $\U\to P$ if and only if $P$ contains a fundamental element. For a field $k$, this is the case if and only if $\# k\geq 3$.
  \item \label{morphism2} There is a morphism $\D\to P$ if and only if there is an element $u\in P^\times$ such that $u+u=1$. For a field $k$, this is the case if and only if $\char k\neq 2$.
  \item \label{morphism3} There is a morphism $\H\to P$ if and only if there is an element $u\in P^\times$ such that $u^3=-1$ and $u-u^2=1$. For a field $k$, this is the case if and only if $\char k=3$ or if $k$ contains a primitive third root of unity.
  \item \label{morphism4} There is a morphism $\F_3\to P$ if and only if $1+1+1=0$ in $P$. For a field $k$, this is the case if and only if $\char k=3$.
  \item \label{morphism5} There is a morphism $\F_2\to P$ if and only if $-1=1$ in $P$. For a field $k$, this is the case if and only if $\char k=2$.
 \end{enumerate}
 There exist morphisms from $\U$, $\D$, $\H$, $\F_3$ and $\F_2$ into the pastures $\U$, $\D$, $\H$, $\F_q$ for $q=2,\dotsc,8$, $\Q$, $\C$, $\S$, $\P$ and $\W$ where Table \ref{table: morphisms between pastures} contains a check mark---a dash indicates that there is no morphism.
\end{lemma}

\begin{table}[tb]
 \centering
 \caption{Existence of morphisms from $\U$, $\D$, $\H$, $\F_3$ and $\F_2$ into other pastures}
 \label{table: morphisms between pastures}
 \begin{tabular}{|c||c|c|c|c|c|c|c|c|c|c|c|c|c|c|c|}
  \hline 
         & $\U$     & $\D$     & $\H$     & $\F_2$   & $\F_3$   & $\F_4$   & $\F_5$   & $\F_7$   & $\F_8$   & $\Q$     & $\C$     & $\S$     & $\P$     & $\W$ \\
  \hline \hline 
  $\U$   & $\check$ & $\check$ & $\check$ & $-$      & $\check$ & $\check$ & $\check$ & $\check$ & $\check$ & $\check$ & $\check$ & $\check$ & $\check$ & $\check$ \\
  \hline
  $\D$   & $-$      & $\check$ & $-$      & $-$      & $\check$ & $-$      & $\check$ & $\check$ & $-$      & $\check$ & $\check$ & $\check$ & $\check$ & $\check$ \\
  \hline
  $\H$   & $-$      & $-$      & $\check$ & $-$      & $\check$ & $\check$ & $-$      & $\check$ & $-$      & $-$      & $\check$ & $-$      & $\check$ & $\check$ \\
  \hline
  $\F_3$ & $-$      & $-$      & $-$      & $-$      & $\check$ & $-$      & $-$      & $-$      & $-$      & $-$      & $-$      & $-$      & $-$      & $\check$ \\
  \hline
  $\F_2$ & $-$      & $-$      & $-$      & $\check$ & $-$      & $\check$ & $-$      & $-$      & $\check$ & $-$      & $-$      & $-$      & $-$      & $-$      \\
  \hline
 \end{tabular}
\end{table}

\begin{proof}
 We briefly indicate the reasons for claims \eqref{morphism1}--\eqref{morphism5}. We begin with claim \eqref{morphism1}. The universal property from Proposition \ref{prop: universal property of algebras and quotients} implies that there is a morphism from $\U=\pastgen\Funpm{x,y}{x+y-1}$ to $P$ if and only if there are $u,v\in P$ such that $u+v=1$. By definition, such elements are fundamental elements of $P$. If $P=k$ is a field, then a pair $(u,v)$ of fundamental elements is a point of the line $L=\{(w,1-w))|w\in k\}$ in $k^2$. Since $L$ contains precisely two points $(0,1)$ and $(0,1)$ with vanishing coordinates, the elements of $L\cap(k^\times)^2$ are in bijection with $k-\{0,1\}$. Thus $k$ has a fundamental element if and only if $\# k\geq 3$.
 
 We continue with claim \eqref{morphism2}. The first assertion follows at once from the universal property for $\D=\pastgen\Funpm{z}{z+z-1}$. A field $P=k$ contains an element $u$ with $u+u=1$ if and only if $1+1$ is invertible in $k$, which is the case if and only if $k$ is of characteristic different from $2$.
 
 We continue with claim \eqref{morphism3}. The first assertion follows at once from the universal property for $\H=\pastgen\Funpm{z}{z^3-1,z-z^2-1}$. In a field $P=k$ of characteristic $3$, the element $u=-1$ satisfies $u^3=-1$ and $u-u^2=1$. If $k$ has characteristic different from $3$, then $v=-u$ satisfies the equation $v^2+v+1=0$, which characterizes a primitive third root of unity. Note that we have automatically $u^3=-v^3=-1$ in a field if $v$ is a third root of unity.
 
 Claims \eqref{morphism4} and \eqref{morphism5} are obvious. The existence or non-existence of morphisms as displayed in Table \ref{table: morphisms between pastures} can be easily verified using \eqref{morphism1}--\eqref{morphism5}.
\end{proof}


\subsection{Forbidden minors for regular, binary and ternary matroids}

The techniques of this paper allow for short arguments to re-establish the known characterizations of regular, binary and ternary matroids in terms of forbidden minors, as they have been proven by Tutte in \cite{Tutte58a} and \cite{Tutte58b} for regular and binary matroids, and independently by Bixby in \cite{Bixby79} and by Seymour in \cite{Seymour79} for ternary matroids.

We spell out the following basic fact for its importance for many of the upcoming theorems.

\begin{lemma}\label{lemma: binary and ternary matroids are without large uniform minors}
 Binary matroids and ternary matroids are without large uniform minors.
\end{lemma}

\begin{proof}
 All minors of a binary or ternary matroid are binary or ternary, respectively. Since $U^2_5$ and $U^3_5$ are neither binary nor ternary, the result follows.
\end{proof}

Next we turn to the proofs of the excluded minor characterizations of regular, binary and ternary matroids.

\begin{thm}[Tutte '58]\label{thm: forbidden minors for binary matroids}
 A matroid is regular if and only if it contains no minor of types $U^2_4$, $F_7$ or $F_7^\ast$. A matroid is binary if and only if it contains no minor of type $U^2_4$.
\end{thm}

\begin{proof}
 By Corollary \ref{cor: representations and rescaling classes of U24}, $U^2_4$ is not binary and therefore also not regular. It follows from Theorem \ref{thm: presentation of foundations in terms of bases} that the foundations of $F_7$ and $F_7^\ast$ contain the relation $-1=1$, which means that they do not admit a morphism to $\Funpm$. Thus $F_7$ and $F_7^\ast$ are not regular. 
 
 We are left with showing that the respective lists of forbidden minors are complete. If a matroid $M$ does not contain a minor of type $U^2_4$, then Corollary \ref{cor: foundation as a quotient of a tensorproduct of Us} implies that the foundation $F_M$ of $M$ is equal to $\Funpm$ or $\past\Funpm{\{-1=1\}}=\F_2$. In either case, there is a morphism from $F_M$ to $\F_2$, which shows that $M$ is binary if it has no minor of type $U^2_4$. 
 
 If, in addition, $M$ has no minor of types $F_7$ or $F_7^\ast$, then Corollary \ref{cor: foundation as a quotient of a tensorproduct of Us} implies that $F_M=\Funpm$, and thus $M$ is regular.
\end{proof}

\begin{thm}[Bixby '79, Seymour '79]\label{thm: forbidden minors for ternary matroids}
 A matroid is ternary if and only if it does not contain a minor of type $U^2_5$, $U^3_5$, $F_7$ or $F_7^\ast$.
\end{thm}

\begin{proof}
 If $M$ is ternary, then it does not have a minor of type $U^2_5$ or $U^3_5$ by Lemma \ref{lemma: binary and ternary matroids are without large uniform minors}. Thus Theorem \ref{thm: presentation of foundations in terms of bases} applies, and since $-1\neq 1$ in $\F_3$, $M$ does not have a minor of type $F_7$ or $F_7^\ast$. This establishes all forbidden minors as listed in the theorem.
 
 To show that the list of forbidden minors is complete, we assume that $M$ contains no minors of these types. Then Corollary \ref{cor: alternative structure for matroids without large uniform minors} implies that the foundation of $M$ is isomorphic to $F_1\otimes \dotsb\otimes F_r$ with $F_i\in\{\U,\D,\H,\F_3\}$. Since each of $\U$, $\D$, $\H$, $\F_3$ admits a morphism to $\F_3$, there is a morphism $F_M\to\F_3$, which shows that $M$ is ternary.
\end{proof}


\subsection{Uniqueness of the rescaling class over \texorpdfstring{$\F_3$}{F3}}
\label{subsection: unique rescaling class over F3}

Brylawski and Lucas show in \cite{Brylawski-Lucas73} that a representation of a matroid over $\F_3$ is uniquely determined up to rescaling. Our method yields a short proof of the following generalization.

\begin{thm}\label{thm: unique rescaling class over pastures with at most one fundamental element}
 Let $P$ be a pasture with at most one fundamental element. Then every matroid has at most one rescaling class over $P$.
\end{thm}

\begin{proof}
 Let $M$ be a matroid with foundation $F_M$. Since the rescaling classes of $M$ over $P$ are in bijective correspondence with the morphisms $F_M\to P$, it suffices to show that there is at most one such morphism. 
 
 By Proposition \ref{prop: universal cross ratios are precisely the fundamental elements of the foundation}, every cross ratio of $F_M$ is a fundamental element of $F_M$, and thus must be mapped to a fundamental element $z$ of $P$. By the uniqueness of $z$ (if it exists), the image of every cross ratio is uniquely determined. Since $F_M$ is generated over $\Funpm$ by cross ratios, the result follows.
\end{proof}

\begin{rem}\label{rem: pastures with at most one fundamental element}
 Examples of pastures with at most one fundamental element are $\Funpm$, $\F_2$, $\F_3$ and $\K$. In fact it is not hard to prove that every pasture with at most one fundamental element contains one of these pastures as a subpasture, and that the fundamental element is $-1$ (if it exists). Note that Brylawski and Lucas's theorem concerns the case $P=\F_3$.
\end{rem}


\subsection{Criteria for representability over certain fields}
\label{subsection: criteria for realizability}

Our theory allows us to deduce at once that matroids without large minors that are representable over certain pastures are automatically representable over certain (partial) fields. For instance, we find such criteria in the cases of the sign hyperfield $\S$, the phase hyperfield $\P$ and the weak sign hyperfield $\W$.

Note that the proof of Criterion \eqref{criterion1} in the following theorem strengthens Lee and Scobee's result that every ternary and orientable matroid is dyadic; see \cite[Cor.\ 1]{Lee-Scobee99}. In fact, we further improve on this result in Theorem \ref{thm: oriented matroids without large minors are uniquely dyadic} where we show that every orientation is uniquely liftable to $\D$ up to rescaling.

In the statement of the following theorem, recall that a matroid is said to be {\em weakly orientable} if it is representable over $\W$.

\begin{thm}\label{thm: realizability criteria for S and P and W}
 Let $M$ be a matroid without large uniform minors. 
 \begin{enumerate}
  \item \label{criterion1} If $M$ is orientable, then it is representable over every field of characteristic different from $2$.
  \item \label{criterion2} If $M$ is representable over $\P$, then it is representable over fields of every characteristic except possibly $2$.
  \item \label{criterion3} If $M$ is weakly orientable, then it is ternary.
 \end{enumerate}
\end{thm}

\begin{proof}
 Let $F_M$ be the foundation of $M$ and $F_M\simeq F_1\otimes\dotsb\otimes F_r$ the decomposition from Theorem \ref{thm: structure theorem for matroids without large uniform minors} into factors $F_i\in\{\U,\D,\H,\F_3,\F_2\}$. If $M$ is representable over a pasture $P$, then there is a morphism $F_M\to P$, and thus there is a morphism $F_i\to P$ for every $i=1,\dotsc,r$. Conversely, if one of the building blocks $\U$, $\D$, $\H$, $\F_3$ and $\F_2$ does not map to $P$, we conclude that this building block does not occur among the $F_i$.
 
Claim \eqref{criterion1} follows since there are no morphisms from $\H$, $\F_3$ or $\F_2$ to $\S$, and both $\U$ and $\D$ map to every field of characteristic different from $2$. Claim \eqref{criterion2} follows since there are no morphisms from $\F_3$ or $\F_2$ to $\P$, and since each of $\U$, $\D$ and $\H$ maps to a field $k$ if its characteristic is $3$ or if it is different from $2$ and if $k$ contains a primitive third root of unity. Claim \eqref{criterion3} follows since there is no morphism from $\F_2$ to $\W$, and each of $\U$, $\D$, $\H$ and $\F_3$ maps to $\F_3$. 
\end{proof}

\begin{rem}
 The proof of Theorem \ref{thm: realizability criteria for S and P and W} shows that similar conclusions can be formulated for other pastures $P$ that do not receive morphisms from some of the building blocks of the foundation $F_M$ of a matroid $M$ without large uniform minors. If $M$ is representable over $P$, then we can conclude the following, for instance:
 \begin{itemize}
  \item if there is no morphism from $\D$ to $P$, then $M$ is quaternary;
  \item if there is no morphism from either $\F_2$ or $\D$ to $P$, then $M$ is hexagonal.
 \end{itemize}
\end{rem}


\subsection{Oriented matroids without large minors are uniquely dyadic}
\label{subsection: oriented matroids without large minors are uniquely dyadic}

Our techniques allow us to strengthen the result of Lee and Scobee (\cite[Thm.\ 1]{Lee-Scobee99}) that an oriented matroid is dyadic if its underlying matroid is ternary. At the end of this section, we deduce Lee and Scobee's result from ours.

An \emph{oriented matroid} is an $\S$-matroid, i.e.\ the class  $M=[\Delta]$ of a Grassmann-Pl\"ucker function $\Delta:E^r\to\S$, where $r$ is the rank of $M$ and $E$ its ground set. The \emph{underlying matroid of $M$} is the matroid $\underline{M}=t_{\S,\ast}(M)$, where $t_\S:\S\to\K$ is the terminal morphism, cf.\ section \ref{subsubsection: initial and final objects}. Recall that a reorientation class is a rescaling class over $\S$.

Let $\sign:\D\to\S$ be the morphism from the dyadic partial field $\D=\pastgen\Funpm{z}{z+z-1}$ to $\S$ that maps $z$ to $1$. An oriented matroid $M=[\Delta]$ is \emph{dyadic} if there is a $\D$-matroid $\widehat M$ such that $M=\sign_\ast(\widehat{M})$. We call $\widehat{M}$ a \emph{lift of $M$ along $\sign:\D\to\S$}.

\begin{thm}\label{thm: oriented matroids without large minors are uniquely dyadic}
 Let $M$ be an oriented matroid whose underlying matroid $\underline{M}$ is without large uniform minors. Then there is a unique rescaling class $[\widehat{M}]$ of dyadic matroids such that $\sign_\ast(\widehat{M})=M$.
\end{thm}

\begin{proof}
 Let $F_{\underline{M}}$ be the foundation of $\underline{M}$. The oriented matroid $M$ determines a reorientation class $[M]$ and thus a morphism $f:F_{\underline{M}}\to \S$. Since rescaling classes of $\underline{M}$ over $\D$ correspond bijectively to morphisms $F_{\underline{M}}\to \D$, we need to show that the morphism $f:F_{\underline{M}}\to \S$ lifts uniquely to $\D$, i.e.\ that there is a unique morphism $\hat f:F_{\underline{M}}\to \D$ such that the diagram
 \[
  \begin{tikzcd}[column sep=1.5cm, row sep=0.4cm]
   F_{\underline{M}} \ar[r,"\hat f"] \ar[rd,"f"'] & \D \ar[d,"\sign"] \\ & \S
  \end{tikzcd}
 \]
 commutes. 
 
 Note that this implies only that there is a unique rescaling class $[\widehat{M}]$ such that the reorientation classes $[\sign_\ast(\widehat{M})]$ and $[M]$ are equal. In order to conclude that we can choose $\widehat{M}$ such that $\sign_\ast(\widehat{M})=M$, we note that the morphism $\sign:\D\to\S$ is surjective, and thus any reorientation $M'=\sign_\ast(\widehat{M})$ of $M$ can be inverted by a rescaling of $\widehat{M}$ over $\D$. This shows that we have proven everything, once we show that $f$ lifts uniquely to $\D$.

 Since $\underline{M}$ is without large uniform minors, Theorem \ref{thm: structure theorem for matroids without large uniform minors} implies that $F_{\underline{M}}$ is isomorphic to $F_1\otimes\dotsb\otimes F_r$ for some $F_1,\dotsc,F_r\in\{\U,\D,\H,\F_3,\F_2\}$. Composing $f:F_{\underline{M}}\to \S$ with the canonical inclusions $\iota_i:F_i\to\F_{\underline{M}}$ yields morphisms $f_i=f\circ\iota_i:F_i\to \S$ for $i=1,\dotsc,r$. As visible in Table \ref{table: morphisms between pastures}, there are no morphisms from $\H$, $\F_3$ or $\F_2$ to $\S$. This means that $F_1,\dotsc,F_r\in\{\U,\D\}$.
 
 By the universal property of the tensor product, the morphisms $\F_{\underline{M}}\to\D$ correspond bijectively to the tuples of morphisms $f_i:F_i\to\D$. Thus there is a unique lift of $f$ to $\D$ if and only if for every $i$, there is a unique lift of $f_i$ to $\D$. This reduces our task to an inspection of the two cases $F_i=\D$ and $F_i=\U$.
 
 Consider the case $f_i:F_i=\D\to\S$. Since $z+z=1$ in $\D$, we must have $f(z)+f(z)=1$ in $\S$, which is only possible if $f(z)=1$. Thus $f_i=\sign$, which means that the identity morphism $\hat f_i=\id:\D\to\D$ lifts $f_i$, i.e.\
 \[
  \begin{tikzcd}[column sep=1.5cm, row sep=0.4cm]
   \D \ar[r,"\hat f_i=\id"] \ar[rd,"f_i"'] & \D \ar[d,"\sign"] \\ & \S
  \end{tikzcd}
 \]
 commutes. This lift is unique since $u+u=1$ is only satisfied by $u=z\in \D$, and thus $\hat f_i(z)=z$ is determined.
 
 We are left with the case $f_i:F_i=\U\to\S$, for which we inspect the possible images of the fundamental elements $x$ and $y$ of $\U$ in $\S$ and $\D$. The relations of the form $u+v-1=0$ in $\S$ are $1+1-1=0$ and $1-1-1=0$. Thus $f_i$ maps $(x,y)$ to one of $(1,1)$, $(1,-1)$ and $(-1,1)$. This means that there are precisely $3$ morphisms $\U\to\S$, and $f_i$ has to be one of them.
 
 The relations of the form $u+v-1=0$ in $\D$ are $z+z-1=0$ and $z^{-1}-1-1=0$. Thus the morphisms $\U\to \U$ correspond to a choice of mapping $(x,y)$ to one of $(z,z)$, $(z^{-1},-1)$ and $(-1,z^{-1})$. Considering the respective images $\sign(z)=\sign(z^{-1})=1$ and $\sign(-1)=-1$ in $\S$, we conclude that every morphism $f_i:\U\to\S$ lifts uniquely to a morphism $\hat f_i:\U\to \D$, i.e.\ 
 \[
  \begin{tikzcd}[column sep=1.5cm, row sep=0.4cm]
   \U \ar[r,"\hat f_i"] \ar[rd,"f_i"'] & \D \ar[d,"\sign"] \\ & \S
  \end{tikzcd}
 \]
 commutes. This completes the proof of the theorem.
\end{proof}

As an application, we show how Theorem \ref{thm: oriented matroids without large minors are uniquely dyadic} implies the result \cite[Thm.\ 1]{Lee-Scobee99} of Lee and Scobee.

\begin{thm}[Lee--Scobee '99]\label{thm: Lee-Scobee}
 An oriented matroid is dyadic if and only if its underlying matroid is ternary.
\end{thm}

\begin{proof}
 Let $M$ be an oriented matroid and let $\underline{M}$ be its underlying matroid.  If $\underline{M}$ is ternary, then it is without large uniform minors. Thus $M$ is dyadic by Theorem \ref{thm: oriented matroids without large minors are uniquely dyadic}.

 Conversely, assume that $M$ is dyadic, i.e.\ it has a lift $\widehat{M}$ along $\sign:\D\to\S$. Since there is a morphism $f:\D\to \F_3$, and since $t_{\F_3}\circ f=t_\S\circ\sign$, the $\F_3$-matroid $f_\ast(\widehat{M})$ is a representation of $\underline{M}=t_{\S,\ast} (M)$ over $\F_3$. Thus $\underline{M}$ is ternary.
\end{proof}


\subsection{Positively oriented matroids without large uniform minors are near-regular}
\label{subsection: positively oriented matroids without large minors are near-regular}

In their 2017 paper \cite{Ardila-Rincon-Williams17}, Ardila, Rinc\'on and Williams prove that every positively oriented matroid can be represented over $\R$ (and {\em a posteriori}, by a theorem of Postnikov, over $\Q$), which solves a conjecture from da Silva's thesis \cite{daSilva87} from 1987. A second proof has recently been obtained by Speyer and Williams in \cite{Speyer-Williams20}. Neither of these proofs yields information about the structure of the lifts of positive orientations to $\Q$ or $\R$.

With our techniques, we can recover and strengthen the result for positively oriented matroids whose underlying matroid is without large uniform minors. To begin with, let us recall the definition of positively oriented matroids.

\begin{df}
 Let $M$ be a matroid of rank $r$ on the ground set $E=\{1,\dotsc,n\}$. A \emph{positive orientation of $M$ (with respect to $E$)} is a Grassmann-Pl\"ucker function $\Delta:E^r\to\S$ such that $t_{\ast,\S}([\Delta])=M$ and such that $\Delta(j_1,\dotsc,j_r)\in\{0,1\}$ for every $(j_1,\dotsc,j_r)\in E^r$ with $j_1<\dotsc<j_r$.

 An oriented matroid $M$ of rank $r$ on $E$ is \emph{positively oriented} if its underlying matroid has a positive orientation $\Delta:E^r\to\S$ with respect to some identification $E\simeq\{1,\dotsc,n\}$ such that $M=[\Delta]$. 
\end{df}

A key tool for proof of Theorem \ref{thm: lifts of positive orientations} is the following notion.

\begin{df}
 Let $M$ be a matroid of rank $r$ on the ground set $E=\{1,\dotsc,n\}$. Let $V$ be the Klein $4$-group, considered as a subgroup of $S_4$. The \emph{$\Omega$-signature of $M$ (with respect to $E$)} is the map
 \[
  \Sigma: \ \Omega_M^\octa \ \longrightarrow \ S_4 / V
 \]
 that sends $(J;e_1,\dotsc,e_4)\in \Omega_M^\octa$ to the class $[\epsilon]\in S_4/V$ of the uniquely determined permutation $\epsilon\in S_4$ that
 \[
  \begin{array}{ccc}
   \{e_1,\dotsc,e_4\} & \longrightarrow & \{1,\dotsc,4\} \\
     e_i              & \longmapsto     &   \epsilon(i)
  \end{array}
 \]
 is an order-preserving bijection.
\end{df}

\begin{ex}
 The key example to understand the relevance of the $\Omega$-signature is the uniform matroid $M=U^2_4$, whose foundation is $F_M=\U$. In this case, $\Omega_M^\octa$ consists of the tuples $(\emptyset;e_1,\dots,e_4)$ for which $(e_1,\dotsc,e_4)$ is a permutation of $(1,\dotsc,4)$. Since the cross ratio $\cross{e_1}{e_2}{e_3}{e_4}{}\in F_M$ determines $(e_1,e_2,e_3,e_4)$ up to a permutation in $V$, which corresponds to a permutation of the rows and the columns of the cross ratio, the $\Omega$-signature induces a well-defined bijection
 \[
  \begin{array}{ccc}
   \Big\{\text{cross ratios in }F_M\Big\} & \longrightarrow & S_4/V \\
   \cross{e_1}{e_2}{e_3}{e_4}{}           & \longmapsto     & \Sigma(\emptyset;e_1,\dotsc,e_4).
  \end{array}
 \]
\end{ex}

\begin{lemma}\label{lemma: value of cross ratios under a positive valuation}
 Let $M$ be a matroid of rank $r$ on the ground set $E=\{1,\dotsc,n\}$ and let $\Delta:E^r\to\S$ be a positive orientation of $M$. Let $(J;e_1,\dotsc,e_4)\in\Omega^\octa_M$ and $\epsilon\in S_4$ be such that $[\epsilon]=\Sigma(J;e_1,\dotsc,e_4)$. Then
 \[
  \cross{e_1}{e_2}{e_3}{e_4}{\Delta,J} \ = \ (-1)^{\epsilon(1)+\epsilon(2)+1}.
 \]
\end{lemma}

\begin{proof}
Choose $\bJ=(j_1,\dotsc,j_{r-2})\in E^{r-2}$ so that $|\bJ|=J$. Since $\Delta$ is a positive orientation, we have for all $i\in\{1,2\}$ and $j\in\{3,4\}$ that $\Delta(\bJ e_ie_j)=\sign \pi_{i,j}$, where $\pi_{i,j}:Je_ie_j\to Je_ie_j$ is the unique permutation such that
 \[
  \pi_{i,j}(j_1) \ < \ \dotsc \ < \ \pi_{i,j}(j_{r-2}) \ < \ \pi_{i,j}(e_i) \ < \ \pi_{i,j}(e_j).
 \]
 Since the cross ratio $\cross{e_1}{e_2}{e_3}{e_4}{\Delta,J}$ is invariant under permutations of $J$, we can assume that $j_1<\dotsc<j_{r-2}$. Thus we can write $\pi_{i,j}=\sigma_{i,j}\circ\epsilon_{i,j}$ as the composition of $\sigma_{i,j}=\pi_{i,j}\circ\epsilon_{i,j}^{-1}$ with the permutation $\epsilon_{i,j}$ of $Je_ie_j$ that fixes $j_1,\dotsc,j_{r-2}$ and satisfies $\epsilon_{i,j}(e_i)<\epsilon_{i,j}(e_j)$. A minimal decomposition of $\sigma_{i,j}$ into transpositions is
 \[
  \sigma_{i,j} \ = \ (j_{k_j}\ e_j) \dotsb (j_{r-2}\ e_j) \  (j_{k_i}\ e_i) \dotsb (j_{r-2}\ e_i),
 \]
 where $k_{i}$ is such that $j_{k_{i}-1}<e_i<j_{k_i}$. Thus
 \[
  \sign(\sigma_{i,j}) \ = \ (-1)^{\big(r-1-k_i\big)+\big(r-1-k_j\big)} \ = \ (-1)^{k_i+k_j},
 \]
 and
 \begin{align*}
  \cross{e_1}{e_2}{e_3}{e_4}{\Delta,J} \ &= \ \frac{\Delta(\bJ e_1e_3)\Delta(\bJ e_2e_4)}{\Delta(\bJ e_1e_4)\Delta(\bJ e_2e_3)} \\[5pt] 
                                         &= \ \frac{\sign(\pi_{1,3})\sign(\pi_{2,4})}{\sign(\pi_{1,4})\sign(\pi_{2,3})} \\[5pt] 
                                         &= \ \frac{(-1)^{k_1+k_3}(-1)^{k_2+k_4}}{(-1)^{k_1+k_4}(-1)^{k_2+k_3}} \ \cdot \ \frac{\sign(\epsilon_{1,3})\sign(\epsilon_{2,4})}{\sign(\epsilon_{1,4})\sign(\epsilon_{2,3})} \\[5pt]
                                         & = \ {\sign(\epsilon_{1,3})\sign(\epsilon_{2,4})}{\sign(\epsilon_{1,4})\sign(\epsilon_{2,3})}.
 \end{align*}

 Since the parity of $\epsilon'(1)+\epsilon'(2)+1$ is even for every $\epsilon'\in V$, we can assume that $\epsilon$ is the representative that occurs in the definition of $\Sigma$, i.e.\ we can assume that $e_i\mapsto\epsilon(i)$ defines an order preserving bijection $\{e_1,\dotsc,e_4\}\to\{1,\dotsc,4\}$. Then $\epsilon_{i,j}$ is the identity if $\epsilon(i)<\epsilon(j)$ and $\epsilon_{i,j}=(e_i\ e_j)$ if $\epsilon(i)>\epsilon(j)$. Thus $\sign(\epsilon_{i,j})=1$ if $\epsilon(i)<\epsilon(j)$ and $\sign(\epsilon_{i,j})=-1$ if $\epsilon(i)>\epsilon(j)$.
 
 Since $\cross{e_1}{e_2}{e_3}{e_4}{\Delta,J}$ is invariant under exchanging rows and columns, we can assume that $e_1$ is the minimal element in $\{e_1,\dotsc,e_4\}$, i.e.\ $\epsilon(1)=1$ and $\sign(\epsilon_{1,j})=1$ for $j\in\{3,4\}$. We verify the claim of the lemma by a case consideration for the value of $\epsilon(2)$.
 
 If $\epsilon(2)=2$, then $e_2$ is minimal in $\{e_2,e_3,e_4\}$ and $\sign(\epsilon_{2,j})=1$ for all $j\in\{3,4\}$. Thus
 \[
  \cross{e_1}{e_2}{e_3}{e_4}{\Delta,J} \ = \ 1 \ = \ (-1)^{1+2+1} \ = \ (-1)^{\epsilon(1)+\epsilon(2)+1}.
 \]
 If $\epsilon(2)=3$, then $e_3<e_2<e_4$ or $e_4<e_2<e_3$. Thus $\sign(\epsilon_{2,3})\sign(\epsilon_{2,4})=-1$ and 
 \[
  \cross{e_1}{e_2}{e_3}{e_4}{\Delta,J} \ = \ -1 \ = \ (-1)^{1+3+1} \ = \ (-1)^{\epsilon(1)+\epsilon(2)+1}.
 \]
 If $\epsilon(2)=4$, then $e_2$ is maximal in $\{e_2,e_3,e_4\}$ and $\sign(\epsilon_{2,j})=-1$ for all $j\in\{3,4\}$. Thus
 \[
  \cross{e_1}{e_2}{e_3}{e_4}{\Delta,J} \ = \ (-1)^2 \ = \ (-1)^{1+4+1} \ = \ (-1)^{\epsilon(1)+\epsilon(2)+1},
 \]
 which completes the proof.
\end{proof}

Let $f:P\to \S$ be a morphism of pastures. A \emph{lift of $M$ to $P$ (along $f$)} is a $P$-matroid $\widehat{M}$ such that $f_\ast(\widehat M)=M$. In the following result, we will implicitly understand that a subfield $k$ of $\R$ comes with the sign map $\sign:k\to\S$. 

As explained in Corollary \ref{cor: representations and rescaling classes of U24}, the near-regular partial field $\U=\pastgen\Funpm{x,y}{x+y-1}$ admits three morphisms to $\S$. Since the automorphism group $\Aut(\U)$ acts transitively on these three morphisms, we can fix one of them without restricting the generality of our results. Thus we will implicitly understand that $\U$ comes with the morphism $\sign:\U\to\S$ given by $\sign(x)=\sign(y)=1$.

\begin{thm}\label{thm: lifts of positive orientations}
 Let $M$ be a positively oriented matroid whose underlying matroid $\underline{M}$ is without large uniform minors. Then $\underline{M}$ is near-regular and $F_{\underline M}\simeq \U^{\otimes r}$ for some $r\geq0$. Up to rescaling equivalence, there are precisely $2^r$ lifts of $M$ to $\U$, and for every subfield $k$ of $\R$, the lifts of $M$ to $k$ modulo rescaling equivalence correspond bijectively to $\Big((0,1)\cap k\Big)^r$.
\end{thm}

\begin{proof}
 By Theorem \ref{thm: structure theorem for matroids without large uniform minors}, the foundation $F_{\underline M}$ is isomorphic to a tensor product $F_1\otimes\dotsb\otimes F_r$ of copies $F_i$ of $\F_2$ and symmetry quotients of $\U$. The rescaling class of $M$ induces a morphism $F_{\underline M}\to\S$. Since there is no morphism from $\F_2$ to $\S$, each of the factors $F_i$ has to be a symmetry quotient of $\U$.
 
 From the proof of Theorem \ref{thm: structure theorem for matroids without large uniform minors}, it follows that each symmetry quotient $F_i=\U/H_i$ of $\U$ is the image of the induced morphism $\U\simeq F_N\to F_M$ of foundations for an embedded $U^2_4$-minor $N=M\minor IJ$ of $M$. This means that for every $\sigma\in H_i$ and every $(J;e_1,\dotsc,e_4)\in\Omega_M$, we have an identity of universal cross ratios
 \[
  \cross{e_1}{e_2}{e_3}{e_4}J \ = \ \cross{\sigma(e_1)}{\sigma(e_2)}{\sigma(e_3)}{\sigma(e_4)}J. 
 \]

 We claim that if $\cross{e_1}{e_2}{e_3}{e_4}J = \cross{e_1'}{e_2'}{e_3'}{e_4'}J$ then $\Sigma(e_1,\dotsc,e_4)=\Sigma(e_1',\dotsc,e_4')$, where $\Sigma:\Omega_M^\octa \to S_4/V$ is the $\Omega$-signature. We verify this in the following for all the defining relations of $F_{\underline M}$ that involve non-degenerate cross ratios, as they appear in Theorem \ref{thm: presentation of foundations in terms of bases}. 
 
 The relations \eqref{R-} and \eqref{R0} do not involve non-degenerate cross ratios (and \eqref{R-} does not occur in our case since neither the Fano matroid not its dual are orientable). The relations \eqref{Rs}, \eqref{R1}, \eqref{R2} and \eqref{R+} are already incorporated in $\U$ and can thus be ignored. For relation \eqref{R5}, it is obvious that both involved cross ratios have the same $\Omega$-signature.

 Thus we are left the relations \eqref{R3} and \eqref{R4}. Since $\underline M$ is without large uniform minors, each of these relations reduces to an identity of two universal cross ratios. We begin with the tip relation \eqref{R2}, which is of the form
 \[
  \cross{e_1}{e_2}{e_3}{e_4}J \ = \ \cross{e_1}{e_2}{e_3}{e_5}J
 \]
 in our case, where we use \eqref{R1} to express $\crossinv{e_1}{e_2}{e_5}{e_3}J$ as $\cross{e_1}{e_2}{e_3}{e_5}J$. After a permutation of $\{e_1,\dotsc,e_4\}$, we can assume that $e_1<e_4<e_2<e_3$, and thus
 \[
  \cross{e_1}{e_2}{e_3}{e_4}{\Delta,J} \ = \ (-1)^{1+3+1} \ = \ -1 
 \]
 by Lemma \ref{lemma: value of cross ratios under a positive valuation}. Therefore also $\cross{e_1}{e_2}{e_3}{e_5}{\Delta,J}=-1$, which means that the unique order preserving bijection $\pi:\{e_1,e_2,e_3,e_5\}\to\{1,\dots,4\}$ must satisfy $\pi(e_1)=\pi(e_2)$ according to Lemma \ref{lemma: value of cross ratios under a positive valuation}. Since $e_1<e_2<e_3$ by our assumptions, this implies that $e_1<e_5<e_2$. Thus $\Sigma(e_1,e_2,e_3,e_4)=\Sigma(e_1,e_2,e_3,e_5)$.
 
 The cotip relations \eqref{R3} are in our case of the form
 \[
  \cross{e_1}{e_2}{e_3}{e_4}{Je_5} \ = \ \cross{e_1}{e_2}{e_3}{e_5}{Je_4}.
 \]
 As before, we can assume that $e_1<e_4<e_2<e_3$ and thus $\cross{e_1}{e_2}{e_3}{e_4}{\Delta,Je_5}=-1$. By the same reasoning, this implies that $e_1<e_5<e_2<e_3$ and thus $\Sigma(e_1,e_2,e_3,e_4)=\Sigma(e_1,e_2,e_3,e_5)$. This establishes our claim that $\Sigma(e_1,\dotsc,e_4)=\Sigma(e_1',\dotsc,e_4')$ whenever $\cross{e_1}{e_2}{e_3}{e_4}J = \cross{e_1'}{e_2'}{e_3'}{e_4'}J$.
 
 In particular, if $\cross{e_1}{e_2}{e_3}{e_4}J=\cross{\sigma(e_1)}{\sigma(e_2)}{\sigma(e_3)}{\sigma(e_4)}J$ then $\Sigma(e_1,\dotsc,e_4)=\Sigma\Big({\sigma(e_1)},\dotsc,{\sigma(e_4)}\Big)$, which means that $\sigma$ is in $V$. These are precisely the relations in \eqref{Rs}, which are already satisfied in $\U$. We conclude that $\sigma$ is the identity on $\U$.
 
 This shows that every factor $F_i$ of $F_{\underline M}$ is a trivial quotient of $\U$ and thus $F_{\underline M}\simeq \U^{\otimes r}$, as claimed in the theorem. It also implies at once that $\underline M$ is near-regular.
 
 Let $\chi_M:F_{\underline M}\to \S$ be the morphism of pastures induced by the rescaling class of $M$. The lifts of $M$ to $\U$ and $k$, up to rescaling, correspond to the lifts of $\chi_M$ to $\U$ and $k$, respectively. We can study this question for each factor $F_i=\U$ of $F_{\underline M}$ individually. 
 
 A lift of $f:\U\to \S$ to $\U$ is a morphism $\hat f:\U\to\U$ such that $\sign\big(\hat f(x)\big)=\sign\big(\hat f(y)\big)=1$. This determines $\hat f$ up to a permutation of $x$ and $y$, which shows that there are precisely two lifts of $f:\U\to\S$ to $\U$. Thus there are precisely $2^r$ lifts of $M$ to $\U$ up to rescaling equivalence.
 
 A lift of $f:\U\to \S$ to $k$ is a morphism $\hat f:\U\to k$ such that $\sign\big(\hat f(x)\big)=\sign\big(\hat f(y)\big)=1$. Since $\hat f(y)=1-\hat f(x)$, this means that $\hat f(x)\in \Big((0,1)\cap k\Big)$ and, conversely, every choice of image $\hat f(x)\in\Big((0,1)\cap k\Big)$ determines a lift $\hat f$ of $f$ to $k$. Thus the lifts of $M$ to $k$ up to rescaling equivalence correspond bijectively to $\Big((0,1)\cap k\Big)^r$. This completes the proof of the theorem.
\end{proof}


\subsection{Representation classes of matroids without large uniform minors}
\label{subsection: representation classes of matroids without large uniform minors}

Given a matroid $M$, we can ask over which pastures $M$ is representable. This defines a class of pastures that we call the representation class of $M$.

For cardinality reasons, it is clear that not every class of pastures can be the representation class of a matroid. The theorems in Section \ref{subsection: characterization of classes of matroids} make clear that this fails in an even more drastic way---for example, a matroid that is representable over $\F_2$ and $\F_3$ is representable over all pastures; cf.\ Theorem \ref{thm: characterization of regular matroids}.

In this section, we determine the representation classes that are defined by matroids without large uniform minors. It turns out that there are only twelve of them; see Table \ref{table: representation classes} for a characterization.

\begin{df}
 Let $M$ be a matroid. The \emph{representation class of $M$} is the class $\cP_M$ of all pastures $P$ over which $M$ is representable. Two matroids $M$ and $M'$ are \emph{representation equivalent} if $\cP_M=\cP_{M'}$.
\end{df}

Note that the representation class $\cP_M$ of a matroid $M$ consists of precisely those pastures for which there is a morphism from the foundation $F_M$ of $M$ to $P$. This means that the representation class of a matroid is determined by its foundation. Evidently, $\cP_M=\cP_{M'}$ if $M$ and $M'$ are representation equivalent, which justifies the notation $\cP_C=\cP_M$ where $C$ is the representation class of $M$.

Often there are simpler pastures than the foundation that characterize representation classes in the same way, which leads to the following notion.

\begin{df}
 Let $M$ be a matroid with representation class $\cP_M$. A \emph{characteristic pasture for $M$} is a pasture $\Pi$ for which a pasture $P$ is in $\cP_M$ if and only if there is a morphism $\Pi\to P$. A matroid $M$ is \emph{strictly representable over a pasture $P$} if $P$ is a characteristic pasture for $M$.
\end{df}

By the existence of the identity morphism $\id:\Pi\to \Pi$, strictly representable implies representable. 
And the foundation of a matroid $M$ is clearly a characteristic pasture for $M$. The following result characterizes all characteristic pastures:

\begin{lemma}\label{lemma: characterization of characteristic pastures}
 Let $M$ be a matroid with foundation $F_M$. A pasture $\Pi$ is a characteristic pasture of $M$ if and only if there exist morphisms $F_M\to \Pi$ and $\Pi\to F_M$.
\end{lemma}

\begin{proof}
 Assume that $\Pi$ is a characteristic pasture for $M$. Since also $F_M$ is a characteristic pasture, we have $F_M,\Pi\in\cP_M$, and by the defining property of characteristic pastures, there are morphisms $F_M\to \Pi$ and $\Pi\to F_M$. 
 
 Conversely, assume that there are morphisms $F_M\to \Pi$ and $\Pi\to F_M$. If $P\in\cP_M$, then there is a morphism $F_M\to P$, which yields a morphism $\Pi\to F_M\to P$. If there is a morphism $\Pi\to P$, then there is a morphism $F_M\to \Pi\to P$, and thus $P\in\cP_M$. This shows that $\Pi$ is a characteristic pasture for $M$.
\end{proof}

The next result describes an explicit condition for representation equivalent matroids.

\begin{lemma}\label{lemma: representation equivalence in terms of characteristic pastures}
 Let $M$ and $M$ be two matroids with respective representation classes $\cP_M$ and $\cP_{M'}$ and respective characteristic pastures $\Pi$ and $\Pi'$. Then $\cP_{M'}$ is contained in $\cP_{M}$ if and only if there is a morphism $\Pi\to\Pi'$. In particular, $M$ and $N$ are representation equivalent if and only if there exist morphisms $\Pi\to\Pi'$ and $\Pi'\to\Pi$. 
\end{lemma}

\begin{proof}
 If there is a morphism $f:\Pi\to\Pi'$, then we can compose every morphism $\Pi'\to P$ with $f$, which implies that $\cP_{M'}\subset \cP_M$. Assume conversely that $\cP_{M'}\subset \cP_M$. Then $\Pi'\in\cP_M$, which means that there is a morphism $\Pi\to \Pi'$. The additional claim of the lemma is obvious.
\end{proof}

In the following, we say that a matroid $M$ is
 \begin{itemize}
  \item \emph{strictly binary} if $\F_2$ is a characteristic pasture for $M$;
  \item \emph{strictly ternary} if $\F_3$ is a characteristic pasture for $M$;
  \item \emph{strictly near-regular} if $\U$ is a characteristic pasture for $M$;
  \item \emph{strictly dyadic} if $\D$ is a characteristic pasture for $M$;
  \item \emph{strictly hexagonal} if $\H$ is a characteristic pasture for $M$;
  \item \emph{strictly $\D\otimes\H$-representable} if $\D\otimes\H$ is a characteristic pasture for $M$;
  \item \emph{idempotent} if $\K$ is a characteristic pasture for $M$.
 \end{itemize}
Note that an idempotent matroid $M$ is representable over a pasture $P$ if and only if $P$ is \emph{idempotent}, by which we mean that both $-1=1$ and $1+1=1$ hold in $P$.

\begin{thm}\label{thm: representation classes of matroids without large uniform minors}
 Let $M$ be a matroid without large uniform minors. Then $M$ belongs to precisely one of the $12$ classes that are described in Table \ref{table: representation classes}. The six columns of Table \ref{table: representation classes} describe the following information: 
 \begin{enumerate}
  \item\label{column1} a label for each class $C$; 
  \item\label{column2} a name (as far as we have introduced one); 
  \item\label{column3} a characteristic pasture $\Pi_C$ that is minimal in the sense that the foundation of every matroid $M$ in the class $C$ is of isomorphism type $F_M\simeq\Pi_C\otimes F_1\otimes \dotsb\otimes F_r$ for some $r\geq 0$ and $F_1,\dotsc,F_r\in\{\U,\D,\H\}$; 
  \item\label{column4} the type of factors $F_i$ that can occur in the expression $F_M\simeq\Pi_C\otimes F_1\otimes \dotsb\otimes F_r$ for $M$ in $C$; 
  \item\label{column5} a characterization of the pastures $P$ in the representation class $\cP_C$; 
  \item\label{column6} whether the matroids in this class are representable over some field.
 \end{enumerate}
 The left diagram in Figure \ref{figure: characteristic pastures and representation classes} illustrates the existence of morphisms between the different characteristic pastures $\Pi_C$ in Table \ref{table: representation classes}. The right diagram illustrates the inclusion relation between the representation classes $\cP_i=\cP_{C_i}$ (for $i=1,\dotsc,12$)---an edge indicates that the class on the bottom end of the edge is contained in the class at the top end of the edge.
\end{thm}

\begin{table}[tb]
 \centering
 \caption{The equivalence classes of matroids without large uniform minors}
 \label{table: representation classes}
 \begin{tabular}{|c|c|c|c|c|c|}
  \hline 
  $C$      & Name                      & minimal $\Pi_C$          & add.\ $F_i$       & $P\in \cP_C$ iff.\ $\exists u,v\in P^\times$ s.t.\ & field? \ \\
  \hline
  $C_1$    & regular                   & $\Funpm$                 &                   &                               & yes \\
  $C_2$    & str.\ near-regular        & $\U$                     &  $\U$             & $u+v=1$                       & yes \\
  $C_3$    & strictly dyadic           & $\D$                     &  $\U$, $\D$       & $u+u=1$                       & yes \\
  $C_4$    & str.\ hexagonal           & $\H$                     &  $\U$, $\H$       & $v-v^2=-v^3=1$                & yes \\
  $C_5$    & str.\ $\D\otimes\H$-repr. & $\D\otimes\H$            &  $\U$, $\D$, $\H$ & $u+u=v-v^2=-v^3=1$            & yes \\
  $C_6$    & strictly ternary          & $\F_3$                   &  $\U$, $\D$, $\H$ & $1+1=1$                       & yes \\
  $C_7$    & strictly binary           & $\F_2$                   &                   & $-1=1$                        & yes \\
  $C_8$    &                           & $\F_2\otimes\U$          &  $\U$             & $-1=u+v=1$                    & yes \\
  $C_9$    &                           & $\F_2\otimes\D$          &  $\U$, $\D$       & $-1=u+u=1$                    & no  \\
  $C_{10}$ &                           & $\F_2\otimes\H$          &  $\U$, $\H$       & $-1=v-v^2=v^3=1$              & yes \\
  $C_{11}$ &                           & $\F_2\otimes\D\otimes\H$ &  $\U$, $\D$, $\H$ & $-1=u+u=v-v^2=v^3=1$          & no  \\
  $C_{12}$ & idempotent                & $\F_2\otimes\F_3$        &  $\U$, $\D$, $\H$ & $-1=1+1=1$                    & no  \\
  \hline
 \end{tabular}
\end{table}

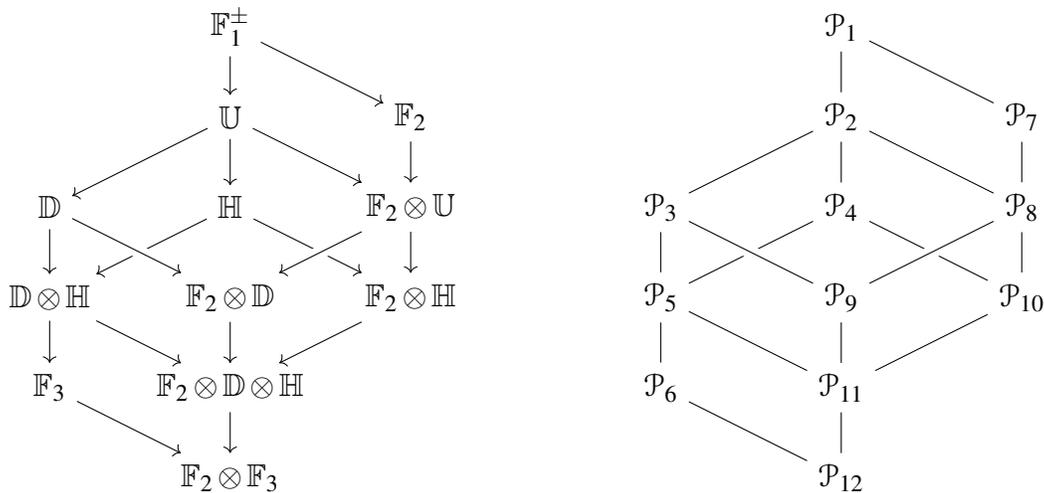
\begin{figure}[htb]
 \centering
 \leavevmode
 \beginpgfgraphicnamed{tikz/fig5}
  \begin{tikzpicture}[x=1.2cm,y=1.2cm]
   \node (1) at (3,6) {$\Funpm$};
   \node (2) at (3,5) {$\U$};
   \node (3) at (1,4) {$\D$};
   \node (4) at (3,4) {$\H$};
   \node (5) at (1,3) {$\D\otimes\H$};
   \node (6) at (1,2) {$\F_3$};
   \node (7) at (5,5) {$\F_2$};
   \node (8) at (5,4) {$\F_2\otimes\U$};
   \node (9) at (3,3) {$\F_2\otimes\D$};
   \node (10) at (5,3) {$\F_2\otimes\H$};
   \node (11) at (3,2) {$\F_2\otimes\D\otimes\H$};
   \node (12) at (3,1) {$\F_2\otimes\F_3$};
   \draw[->] (1) to (2);
   \draw[->] (1) to (7);
   \draw[->] (2) to (3);
   \draw[->] (2) to (4);
   \draw[->] (2) to (8);
   \draw[->] (4) to (5);
   \draw[->] (4) to (10);
   \draw[->] (3) to (5);
   \draw[-,line width=5pt,color=white] (3) to (9); \draw[->] (3) to (9);
   \draw[->] (5) to (6);
   \draw[->] (5) to (11);
   \draw[->] (6) to (12);
   \draw[->] (7) to (8);
   \draw[-,line width=5pt,color=white] (8) to (9); \draw[->] (8) to (9);
   \draw[->] (8) to (10);
   \draw[->] (9) to (11);
   \draw[->] (10) to (11);
   \draw[->] (11) to (12);
  \end{tikzpicture}
 \endpgfgraphicnamed
 \hspace{2cm}
 \beginpgfgraphicnamed{tikz/fig6}
  \begin{tikzpicture}[x=1.2cm,y=1.2cm]
   \node (1) at (3,6) {$\cP_1$};
   \node (2) at (3,5) {$\cP_2$};
   \node (3) at (1,4) {$\cP_3$};
   \node (4) at (3,4) {$\cP_4$};
   \node (5) at (1,3) {$\cP_5$};
   \node (6) at (1,2) {$\cP_6$};
   \node (7) at (5,5) {$\cP_7$};
   \node (8) at (5,4) {$\cP_8$};
   \node (9) at (3,3) {$\cP_9$};
   \node (10) at (5,3) {$\cP_{10}$};
   \node (11) at (3,2) {$\cP_{11}$};
   \node (12) at (3,1) {$\cP_{12}$};
   \draw[-] (1) to (2);
   \draw[-] (1) to (7);
   \draw[-] (2) to (3);
   \draw[-] (2) to (4);
   \draw[-] (2) to (8);
   \draw[-] (4) to (5);
   \draw[-] (4) to (10);
   \draw[-] (3) to (5);
   \draw[-,line width=5pt,color=white] (3) to (9); \draw[-] (3) to (9);
   \draw[-] (5) to (6);
   \draw[-] (5) to (11);
   \draw[-] (6) to (12);
   \draw[-] (7) to (8);
   \draw[-,line width=5pt,color=white] (8) to (9); \draw[-] (8) to (9);
   \draw[-] (8) to (10);
   \draw[-] (9) to (11);
   \draw[-] (10) to (11);
   \draw[-] (11) to (12);
  \end{tikzpicture}
 \endpgfgraphicnamed
 \caption{Morphisms between characteristic pastures and containment of the representation classes for matroids without large uniform minors}
 \label{figure: characteristic pastures and representation classes}
\end{figure}

\begin{proof}
 For the sake of this proof, we say that two pastures $P$ and $P'$ are \emph{equivalent}, and write $P\sim P'$, if there are morphisms $P\to P'$ and $P'\to P$. 
 
 If there is a morphism $P'\to P$, then there are morphisms $P\to P\otimes P'$ and $P\otimes P'\to P$, which means that $P\otimes P'\sim P$. This applies in particular to $P'=P$. This shows that $P_1\otimes \dotsb\otimes P_r\sim P_1\otimes\dotsb\otimes P_s$ for $s\leq r$ and pastures $P_1,\dotsc,P_r$ if, for every $i\in\{s+1,\dotsc,r\}$, there is a $j\in\{1,\dotsc,r\}$ and a morphism $P_i\to P_j$.
 
 Let $M$ be a matroid without large uniform minors and $F_M$ its foundation. By Theorem \ref{thm: structure theorem for matroids without large uniform minors}, $F_M\simeq F_1\otimes\dotsb\otimes F_r$ for some $F_1,\dotsc,F_r\in\{\U,\D,\H,\F_3,\F_2\}$, where we can assume that $\F_2$ appears at most once as a factor. By the previous considerations, $F_M\sim F_1 \otimes\dotsb\otimes F_s$
 for pairwise distinct $F_1,\dotsc,F_s\in\{\U,\D,\H,\F_3,\F_2\}$. Since there are morphisms
 
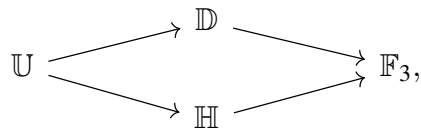
\[
  \begin{tikzcd}[column sep=50pt, row sep=0pt]
                    & \D \ar[dr]        \\
   \U\ar[ur]\ar[dr] &            & \F_3, \\
                    & \H \ar[ur]        
  \end{tikzcd}
 \]
 we have $\D\otimes\U\sim\D$, $\H\otimes\U\sim\H$ and $\F_3\otimes F\sim \F_3$ for $F\in\{\U,\D,\H\}$. Thus we can assume that in the expression $F_1 \otimes\dotsb\otimes F_s$ at most one of $\U$, $\D$, $\H$ and $\F_3$ appears, with the exception of $\D\otimes\H$. 
 
 Thus we are limited to the twelve different expressions for $F_1 \otimes\dotsb\otimes F_s$ that appear in Figure \ref{figure: characteristic pastures and representation classes}. We conclude that $F_M$ is equivalent to one of those and that $\Pi=F_1 \otimes\dotsb\otimes F_s$ is a characteristic pasture for $M$. 
 
 An easy case-by-case verification based on Table \ref{table: morphisms between pastures}, which we shall not carry out, shows that there is a morphism between two pastures if and only if there is a directed path between these pastures in the diagram on the left hand side of Figure \ref{figure: characteristic pastures and representation classes}. By Lemma \ref{lemma: representation equivalence in terms of characteristic pastures}, this diagram determines at once the inclusion behaviour of the associated representation classes $\cP_1$--$\cP_{12}$ as illustrated on the right hand side of Figure \ref{figure: characteristic pastures and representation classes}.
 
 Note that the way we found the twelve characteristic pastures $\Pi$ shows that they are minimal in the sense of part \eqref{column3} of the theorem, and it shows that the types of additional factors displayed in the forth column of Table \ref{table: representation classes} are correct. 
 The conditions in the fifth column of Table \ref{table: representation classes} follows at once from Lemma \ref{lemma: conditions for morphisms from UDHF3F2 to a pasture}. 
 
 For the verification of the last column, note that there is a morphism $\Pi_C\to\F_3$ for the classes $C\in\{C_1,\dotsc,C_6\}$ and that there is a morphism $\Pi\to\F_4$ for $C\in\{C_7,C_8,C_{10}\}$. Thus the matroids in the classes $C_1$--$C_8$ and $C_{10}$ are representable over a field. There is no morphism from $\F_2\otimes\D$ to any field since in a field only one of $1+1=0$ and $1+1=z^{-1}$ for some $z\neq 0$ can hold. Thus matroids in the classes $C_9$, $C_{11}$ and $C_{12}$ are not representable over any field, which concludes the proof of the theorem.
\end{proof}

As a sample application, we formulate the following strengthening of the result \cite[Thm.\ 3.3]{Whittle05} by Whittle. Recall that a matroid is called {\em representable} if it is representable over some field.

\begin{thm}
 Let $\cP_{\leq8}=\big\{\F_q\,\big|\,q\leq 8\text{ a prime power}\big\}$. Then two representable matroids $M$ and $M'$ without large uniform minors are representation equivalent if and only if $\cP_M\cap\cP_{\leq8}=\cP_{M'}\cap\cP_{\leq8}$. More precisely, for $i\in\{1,\dotsc,8,10\}$ and $p_i$ and $q_i$ as in Table \ref{table: primes for representable representation classes}, the class $\cP_{C_i}$ is the intersection of the representation classes $\cP_M$ of all matroids $M$ without large uniform minors that are representable over $\F_{p_i}$ and $\F_{q_i}$.
\end{thm}

\begin{table}[tb]
 \centering
 \caption{Prime powers such that $\cP_{C_i}=\bigcap\big\{\cP_M \, \big|\, M\text{ is representable over $\F_{p_i}$ and $\F_{q_i}$}\big\}$}
 \label{table: primes for representable representation classes}
 \begin{tabular}{|c||c|c|c|c|c|c|c|c|c|}
  \hline 
  $i$   & 1 & 2 & 3 & 4 & 5 & 6 & 7 & 8 & 10 \\ 
  \hline \hline 
  $p_i$ & 2 & 3 & 3 & 3 & 3 & 3 & 2 & 8 & 4  \\
  \hline
  $q_i$ & 3 & 8 & 5 & 4 & 7 & 3 & 2 & 8 & 4  \\
  \hline
 \end{tabular}
\end{table}

\begin{proof}
 For $i\in\{1,\dotsc,8,10\}$ and $M$ in $C_i$, let $\cU_i$ be the subset of $\{\U,\D,\H,\F_3,\F_2\}$ such that $\Pi_i=\bigotimes_{P\in\cU_i} P$ is a characteristic pasture for $M$, cf.\ Table \ref{table: representation classes}. Then we can read off from Table \ref{table: morphisms between pastures} that there are morphisms $P\to\F_{p_i}$ and $P\to\F_{q_i}$ for all $P\in\cU_i$, and that for all $P\in \{\U,\D,\H,\F_3,\F_2\}$ that are not in $\cU_i$, there is either no morphism from $P$ to $\F_{p_i}$ or no morphism from $P$ to $\F_{q_i}$. This shows that the existence of morphisms into $\F_{p_i}$ and $\F_{q_i}$ characterize the factors of the characteristic pasture $\Pi_i$ and establishes the claims of the theorem.
\end{proof}

\begin{rem}
 Note that the representation class $\cP_1$ of regular matroids contains all pastures and is therefore the largest possible representation class. The representation class $\cP_{12}$ of idempotent pastures is the smallest representation class, since every matroid is by definition representable over $\K$ and thus over every idempotent pasture. (Recall that a pasture $P$ is called idempotent if there is a morphism from $\K$ to $P$.) Every other representation class thus lies between $\cP_{12}$ and $\cP_1$.
\end{rem}

\begin{rem}\label{rem: existence of matroids in the classes C1-C12}
 We will show in a sequel to this paper that every tensor product of copies of the pastures $\U$, $\D$, $\H$, $\F_3$ and $\F_2$ occurs as the foundation of a matroid. Consequently each of the classes $C_1$--$C_{12}$ is nonempty. 
 
 Alternatively, we can use known results to deduce this. Since there are matroids that are regular, strictly near-regular (e.g.\ $U^2_4$), strictly dyadic (e.g.\ the non-Fano matroid $F_7^-$), strictly hexagonal (e.g.\ the ternary affine plane $AG(2,3)$), strictly ternary (e.g.\ the matroid $T_8$ from Oxley's book \cite{Oxley92}) and strictly binary (e.g.\ the Fano matroid $F_7$), the classes $C_1$, $C_2$, $C_3$, $C_4$, $C_6$ and $C_7$ are nonempty.
 
 Since the characteristic pastures of the remaining classes in Table \ref{table: representation classes} are tensor products of characteristic pastures of one of the aforementioned matroids, we can deduce that the other classes are also nonempty by observing that 
 \[
  \big\{ P \,\big|\, F_M\otimes F_{M'} \stackrel{\hspace{-2pt}\scalebox{0.7}{$\exists$}}{\rightarrow} P \big\} \ = \ \big\{ P \,\big|\, F_M \stackrel{\hspace{-2pt}\scalebox{0.7}{$\exists$}}{\rightarrow} P \big\}  \cap  \big\{ P \,\big|\, F_{M'} \stackrel{\hspace{-2pt}\scalebox{0.7}{$\exists$}}{\rightarrow} P \big\} \ = \ \cP_M\cap\cP_{M'} \ = \ \cP_{M\oplus M'}
 \]
 for two matroids $M$ and $M'$.
\end{rem}

\begin{rem}
 Since all binary and ternary matroids are without large uniform minors, all matroids in the classes $C_1$--$C_7$ are without large uniform minors. This is not true for all classes though. For instance the direct sum of an idempotent matroid with $U^2_5$ is also idempotent and thus in $C_{12}$, but has a minor of type $U^2_5$; cf.\ Remark \ref{rem: existence of matroids in the classes C1-C12} for the existence of idempotent matroids. 
 
 In fact, a similar construction yield matroids with $U^2_5$-minors in the classes $C_{10}$ and $C_{11}$. By contrast, all matroids in $C_8$ and $C_9$ are without large uniform minors. This latter fact can be proven as follows: a class $C_i$ contains a matroid $M$ with a $U^2_5$- or a $U^3_5$-minor if and only if there is morphism from the foundation of $U^2_5$ (cf.\ Proposition \ref{prop: the foundation of U^2_5}) to the minimal characteristic pasture for $M$. There is no morphism from the foundation of $U^2_5$ to $\F_2\otimes\U$ or to $\F_2\otimes\D$, but there are morphisms to $\F_2\otimes\H$ and $\F_2\otimes\D\otimes\H$.
\end{rem}


\subsection{Characterization of classes of matroids}
\label{subsection: characterization of classes of matroids}

In this section, we use our results to provide different characterizations of some prominent classes of matroids, such as regular, near-regular, binary, ternary, quaternary, dyadic, and hexagonal matroids. In particular, we find new proofs for results by Tutte, Bland and Las Vergnas, and Whittle, which we refer to in detail at the beginnings of the appropriate sections. Moreover, we obtain new characterizations, which often involve the pastures $\S$, $\P$ and $\W$.

All these characterizations are immediate applications of Theorem \ref{thm: structure theorem for matroids without large uniform minors} in combination with Table \ref{table: morphisms between pastures}. It is possible to work out additional descriptions for the classes of matroids under consideration, or to study other classes with the same techniques. 
For example, our technique allows for an easy proof of the following results found in Theorems 5.1 and 5.2 of Semple and Whittle's paper \cite{Semple-Whittle96b}.

\begin{thm}[Semple--Whittle '96]
 Let $\cC_P$ denote the class of matroids without large uniform minors that are representable over a pasture $P$. Then the following hold true.
 \begin{enumerate}
  \item $\cC_{\F_{2^r}}\cap\cC_{\F_3}=\cC_\U$ for odd $r\geq2$.
  \item $\cC_{\F_{2^r}}\cap\cC_{\F_3}=\cC_\H$ for even $r\geq2$.
  \item $\cC_k\subset \cC_{\F_3}$ for every field $k$ of characteristic different from $2$, and $\cC_k=\cC_\D$ if, in addition, $k$ does not contain a primitive sixth root of unity.
 \end{enumerate}
\end{thm}


\subsubsection{Regular matroids}
\label{subsubsection: regular matroids}

The following theorem extends a number of classical results that characterize regular matroids, namely as binary matroids that are representable over a field $k$ with $\char k\neq 2$ by Tutte in \cite{Tutte58a} and \cite{Tutte58b} (use $P=k$ in \eqref{regular5}) and as binary and orientable matroids by Bland and Las Vergnas in \cite{Bland-LasVergnas78} (use $P=\S$ in \eqref{regular5}). Up to the characterization \eqref{regular3}, the authors of this paper have proven Theorem \ref{thm: characterization of regular matroids} in its full generality in \cite[Thm.\ 7.33]{Baker-Lorscheid18} with a slightly different proof.

\begin{thm}\label{thm: characterization of regular matroids}
 Let $M$ be a matroid with foundation $F_M$. Then the following assertions are equivalent:
 \begin{enumerate}
  \item\label{regular1} $M$ is regular.
  \item\label{regular2} $F_M=\Funpm$.
  \item\label{regular3} $M$ belongs to $C_1$.
  \item\label{regular4} $M$ is representable over all pastures.
  \item\label{regular5} $M$ is representable over $\F_2$ and a pasture with $-1\neq 1$.
 \end{enumerate}
\end{thm}

\begin{proof}
 The logical structure of this proof is \eqref{regular1}\implies\eqref{regular3}\implies\eqref{regular4}\implies\eqref{regular5}\implies\eqref{regular2}\implies\eqref{regular1}. The implications  \eqref{regular2}\implies\eqref{regular1}\implies\eqref{regular3}\implies\eqref{regular4} follow from Theorem \ref{thm: representation classes of matroids without large uniform minors} and \eqref{regular4}\implies\eqref{regular5} is trivial. 
 
 We close the circle by showing \eqref{regular5}\implies\eqref{regular2}. If $M$ is binary, then it is without large uniform minors by Lemma \ref{lemma: binary and ternary matroids are without large uniform minors}. Thus, by Theorem \ref{thm: structure theorem for matroids without large uniform minors}, $F_M$ is a tensor product of copies of $\U$, $\D$, $\H$, $\F_3$ and $\F_2$. But none of $\U$, $\D$, $\H$ or $\F_3$ admits a morphism to $\F_2$, and $\F_2$ admits no morphism into a pasture $P$ with $-1\neq 1$. Thus $F_M=\Funpm$, as claimed. 
\end{proof}


\subsubsection{Binary matroids}
\label{subsubsection: binary matroids}

We find the following equivalent characterizations of binary matroids.

\begin{thm}\label{thm: characterization of binary matroids}
 Let $M$ be a matroid with foundation $F_M$. Then the following assertions are equivalent:
 \begin{enumerate}
  \item\label{binary1} $M$ is binary.
  \item\label{binary2} $F_M\simeq \Funpm$ or $F_M\simeq\F_2$.
  \item\label{binary3} $M$ belongs to $C_1$ or $C_7$.
  \item\label{binary4} $M$ is representable over every pasture for which $-1=1$.
  \item\label{binary5} All fundamental elements of $F_M$ are trivial.
 \end{enumerate}
\end{thm}

\begin{proof}
 We prove \eqref{binary1}\implies\eqref{binary3}\implies\eqref{binary2}\implies\eqref{binary5}\implies\eqref{binary2}\implies\eqref{binary4}\implies\eqref{binary1}. Steps \eqref{binary1}\implies\eqref{binary3}\implies\eqref{binary2} follow from Theorem \ref{thm: representation classes of matroids without large uniform minors}, step \eqref{binary5}\implies\eqref{binary2} follows from part \eqref{morphism1} of Lemma \ref{lemma: conditions for morphisms from UDHF3F2 to a pasture} and Corollary \ref{cor: alternative structure for matroids without large uniform minors}, and steps \eqref{binary2}\implies\eqref{binary5} and \eqref{binary2}\implies\eqref{binary4}\implies\eqref{binary1} are trivial.
\end{proof}


\subsubsection{Ternary matroids}
\label{subsubsection: ternary matroids}

We find the following equivalent characterizations of ternary matroids.

\begin{thm}\label{thm: characterization of ternary matroids}
 Let $M$ be a matroid with foundation $F_M$. Then the following assertions are equivalent:
 \begin{enumerate}
  \item\label{ternary1} $M$ is ternary.
  \item\label{ternary2} $F_M\simeq F_1\otimes\dotsb\otimes F_r$ for $r\geq0$ and $F_1,\dotsc,F_r\in\{\U,\D,\H,\F_3\}$.  
  \item\label{ternary3} $M$ belongs to one of $C_1$--$C_6$.
  \item\label{ternary4} $M$ is representable over every pasture for which $1+1+1=0$.
  \item\label{ternary5} $M$ is without large uniform minors and representable over a field of characteristic $3$.
  \item\label{ternary6} $M$ is without large uniform minors and weakly orientable.
  \item\label{ternary7} $M$ is without large uniform minors and there is no morphism from $\F_2$ to $F_M$.
 \end{enumerate}
\end{thm}

\begin{proof}
 We show \eqref{ternary2}\iff\eqref{ternary3}, \eqref{ternary1}\iff\eqref{ternary4} and \eqref{ternary2}\implies\eqref{ternary1}\implies\eqref{ternary5} / \eqref{ternary6} / \eqref{ternary7}\implies\eqref{ternary2}. The implications \eqref{ternary2}\implies\eqref{ternary1}\iff\eqref{ternary4} are trivial. The equivalence \eqref{ternary2}\iff\eqref{ternary3} follows from Theorem \ref{thm: representation classes of matroids without large uniform minors}. 
 
 Assuming \eqref{ternary1}, then $M$ is without large uniform minors by Lemma \ref{lemma: binary and ternary matroids are without large uniform minors}. Since there are morphisms $\F_3\to k$ for every field $k$ of characteristic $3$ and $\F_3\to\W$, this implies \eqref{ternary5} and \eqref{ternary6}. 
 
 If $M$ is without large uniform minors, then Theorem \ref{thm: structure theorem for matroids without large uniform minors} implies that $F_M$ is the tensor product of copies of $\U$, $\D$, $\H$, $\F_3$ and $\F_2$. Thus \eqref{ternary1} and the fact that $\F_2$ does not map to $\F_3$ implies \eqref{ternary7}. Conversely, each condition of \eqref{ternary5}, \eqref{ternary6} and \eqref{ternary7} implies that $\F_2$ cannot occur as a building block of $F_M$, and thus \eqref{ternary2}.
\end{proof}


\subsubsection{Quaternary matroids without large uniform minors}
\label{subsubsection: quarternary matroids}

We find the following equivalent characterizations of quaternary matroids without large uniform minors.

\begin{thm}\label{thm: characterization of quarternary matroids}
 Let $M$ be a matroid without large uniform minors and $F_M$ its foundation. Then the following assertions are equivalent:
 \begin{enumerate}
  \item\label{quarternary1} $M$ is quaternary.
  \item\label{quarternary2} $F_M\simeq F_1\otimes\dotsb\otimes F_r$ for $r\geq0$ and $F_1,\dotsc,F_r\in\{\U,\H,\F_2\}$.
  \item\label{quarternary3} $M$ belongs to $C_1$, $C_2$, $C_4$, $C_7$, $C_8$ or $C_{10}$.
  \item\label{quarternary4} $M$ is representable over every pasture for which $1+1=0$ and that contains an element $u$ for which $u^2+u+1=0$.
  \item\label{quarternary5} $M$ is representable over all field extensions of $\F_4$.
  \item\label{quarternary6} There is no morphism from $\D$ to $F_M$.
 \end{enumerate}
\end{thm}

\begin{proof}
 We show \eqref{quarternary2}\iff\eqref{quarternary3} and \eqref{quarternary2}\implies\eqref{quarternary4}\implies\eqref{quarternary1}\implies\eqref{quarternary5}\implies\eqref{quarternary6}\implies\eqref{quarternary2}. The equivalence \eqref{quarternary2}\iff\eqref{quarternary3} follows from Theorem \ref{thm: representation classes of matroids without large uniform minors}. The implications \eqref{quarternary2}\implies\eqref{quarternary4}\implies\eqref{quarternary1}\implies\eqref{quarternary5} are trivial. The implication \eqref{quarternary5}\implies\eqref{quarternary6} follows since there is no morphism from $\D$ to $\F_4$ by Lemma \ref{lemma: conditions for morphisms from UDHF3F2 to a pasture}. The implication \eqref{quarternary6}\implies\eqref{quarternary2} follows by Theorem \ref{thm: structure theorem for matroids without large uniform minors}, together with the fact that there is a morphism $\D\to\F_3$ but not to $\U$, $\H$ and $\F_2$, and thus only the latter three pastures can occur as factors of $F_M$.
\end{proof}


\subsubsection{Near-regular matroids}
\label{subsubsection: near-regular matroids}

In this section, we provide several characterizations of near-regular matroids. The descriptions \eqref{near-regular5} and \eqref{near-regular6} appear in Whittle's paper \cite[Thm.\ 1.4]{Whittle97}.

\begin{thm}\label{thm: characterization of near-regular matroids}
 Let $M$ be a matroid with foundation $F_M$. Then the following assertions are equivalent:
 \begin{enumerate}
  \item\label{near-regular1} $M$ is near-regular.
  \item\label{near-regular2} $F_M\simeq F_1\otimes\dotsb\otimes F_r$ for $r\geq0$ and $F_1=\dotsb=F_r=\U$.
  \item\label{near-regular3} $M$ belongs to $C_1$ or $C_2$.
  \item\label{near-regular4} $M$ is representable over all pastures with a fundamental element.
  \item\label{near-regular5} $M$ is representable over fields with at least $3$ elements.
  \item\label{near-regular6} $M$ is representable over $\F_3$ and $\F_8$.
  \item\label{near-regular7} $M$ is without large uniform minors and representable over $\F_4$ and $\F_5$.
  \item\label{near-regular8} $M$ is without large uniform minors and representable over $\F_4$ and $\S$.
  \item\label{near-regular9} $M$ is without large uniform minors and representable over $\F_8$ and $\W$.
  \item\label{near-regular10} $M$ is dyadic and hexagonal.
  \item\label{near-regular11} $M$ is without large uniform minors and there are no morphisms $\F_2 \to F_M$, $\D \to F_M$, or $\H \to F_M$.
 \end{enumerate}
\end{thm}

\begin{proof}
 We show \eqref{near-regular2}\iff\eqref{near-regular3}, \eqref{near-regular2}\implies\eqref{near-regular1}\implies\eqref{near-regular4}\implies\eqref{near-regular5}\implies\eqref{near-regular2} and the equivalence of \eqref{near-regular2} with each of \eqref{near-regular6}--\eqref{near-regular11}. The equivalence \eqref{near-regular2}\iff\eqref{near-regular3} follows from Theorem \ref{thm: representation classes of matroids without large uniform minors}, \eqref{near-regular2}\implies\eqref{near-regular1} and \eqref{near-regular4}\implies\eqref{near-regular5} are trivial and \eqref{near-regular1}\implies\eqref{near-regular4} follows from Lemma \ref{lemma: conditions for morphisms from UDHF3F2 to a pasture}. That \eqref{near-regular2} implies \eqref{near-regular6}--\eqref{near-regular11} can be read off from Table \ref{table: morphisms between pastures}. Conversely, each of \eqref{near-regular5}--\eqref{near-regular11} implies that $M$ is without large uniform minors and thus Theorem \ref{thm: structure theorem for matroids without large uniform minors} applies. In turn, each of \eqref{near-regular5}--\eqref{near-regular11} excludes that any of $\D$, $\H$, $\F_3$ and $\F_2$ occur as a factor $F_M$, and thus \eqref{near-regular2}. 
\end{proof}


\subsubsection{Dyadic matroids}
\label{subsubsection: dyadic matroids}

In this section, we provide several characterizations of dyadic matroids. Description \eqref{dyadic6} has been given by Whittle in \cite[Thm.\ 7.1]{Whittle95}. Descriptions \eqref{dyadic4} and \eqref{dyadic5} have been given by Whittle in \cite[Thm.\ 1.1]{Whittle97}.

\begin{thm}\label{thm: characterization of dyadic matroids}
 Let $M$ be a matroid with foundation $F_M$. Then the following assertions are equivalent:
 \begin{enumerate}
  \item\label{dyadic1} $M$ is dyadic.
  \item\label{dyadic2} $F_M\simeq F_1\otimes\dotsb\otimes F_r$ for $r\geq0$ and $F_1,\dotsc,F_r\in\{\U,\D\}$.
  \item\label{dyadic3} $M$ belongs to $C_1$, $C_2$ or $C_3$.
  \item\label{dyadic4} $M$ is representable over every pasture $P$ such that $1+1=u$ for some $u\in P^\times$.
  \item\label{dyadic5} $M$ is representable over every field of characteristic different from $2$.
  \item\label{dyadic6} $M$ is representable over $\F_3$ and $\F_q$, where $q$ is an odd prime power such that $q-1$ is not divisible by $3$.
  \item\label{dyadic7} $M$ is representable over $\F_3$ and $\Q$.
  \item\label{dyadic8} $M$ is representable over $\F_3$ and $\S$.
  \item\label{dyadic9} $M$ is without large uniform minors and there are no morphisms $\F_2 \to F_M$ or $\H \to F_M$.
 \end{enumerate}
\end{thm}

\begin{proof}
 We show \eqref{dyadic2}\iff\eqref{dyadic3}, \eqref{dyadic2}\implies\eqref{dyadic1}\implies\eqref{dyadic4}\implies\eqref{dyadic5}\implies\eqref{dyadic2} and the equivalence of \eqref{dyadic2} with each of \eqref{dyadic6}--\eqref{dyadic9}. The equivalence \eqref{dyadic2}\iff\eqref{dyadic3} follows from Theorem \ref{thm: representation classes of matroids without large uniform minors}, \eqref{dyadic2}\implies\eqref{dyadic1} and \eqref{dyadic4}\implies\eqref{dyadic5} are trivial and \eqref{dyadic1}\implies\eqref{dyadic4} follows from Lemma \ref{lemma: conditions for morphisms from UDHF3F2 to a pasture}. That \eqref{dyadic2} implies \eqref{dyadic6}--\eqref{dyadic9} follows from Lemma \ref{lemma: conditions for morphisms from UDHF3F2 to a pasture} and Table \ref{table: morphisms between pastures}. Conversely, each of \eqref{dyadic5}--\eqref{dyadic9} implies that $M$ is without large uniform minors and thus Theorem \ref{thm: structure theorem for matroids without large uniform minors} applies. In turn, each of \eqref{dyadic5}--\eqref{dyadic9} excludes that any of $\H$, $\F_3$ and $\F_2$ occur as a factor $F_M$, and thus \eqref{dyadic2}. 
\end{proof}


\subsubsection{Hexagonal matroids}
\label{subsubsection: hexagonal matroids}

In this section, we provide several characterizations of hexagonal matroids. Description \eqref{hexagonal5} has been given by Whittle in \cite[Thm.\ 1.2]{Whittle97}.

\begin{thm}\label{thm: characterization of hexagonal matroids}
 Let $M$ be a matroid with foundation $F_M$. Then the following assertions are equivalent:
 \begin{enumerate}
  \item\label{hexagonal1} $M$ is hexagonal.
  \item\label{hexagonal2} $F_M\simeq F_1\otimes\dotsb\otimes F_r$ for $r\geq0$ and $F_1,\dotsc,F_r\in\{\U,\H\}$.  
  \item\label{hexagonal3} $M$ belongs to $C_1$, $C_2$ or $C_4$.
  \item\label{hexagonal4} $M$ is representable over every pasture that contains an element $u$ with $u^3=-1$ and $u^2-u+1=0$.
  \item\label{hexagonal5} $M$ is representable over every field that is of characteristic $3$ or contains a primitive sixth root of unity.
  \item\label{hexagonal6} $M$ is representable over $\F_3$ and $\F_4$.
  \item\label{hexagonal7} $M$ is without large uniform minors, weakly orientable, and representable over $\F_4$.
  \item\label{hexagonal8} $M$ is without large uniform minors and there are no morphisms $\F_2 \to F_M$ or $\D \to F_M$.
 \end{enumerate}
\end{thm}

\begin{proof}
 We show \eqref{hexagonal2}\iff\eqref{hexagonal3}, \eqref{hexagonal2}\implies\eqref{hexagonal1}\implies\eqref{hexagonal4}\implies\eqref{hexagonal5}\implies\eqref{hexagonal2} and the equivalence of \eqref{hexagonal2} with each of \eqref{hexagonal6}--\eqref{hexagonal8}. The equivalence \eqref{hexagonal2}\iff\eqref{hexagonal3} follows from Theorem \ref{thm: representation classes of matroids without large uniform minors}, \eqref{hexagonal2}\implies\eqref{hexagonal1} and \eqref{hexagonal4}\implies\eqref{hexagonal5} are trivial and \eqref{hexagonal1}\implies\eqref{hexagonal4} follows from Lemma \ref{lemma: conditions for morphisms from UDHF3F2 to a pasture}. That \eqref{hexagonal2} implies \eqref{hexagonal6}--\eqref{hexagonal8} follows from Lemma \ref{lemma: conditions for morphisms from UDHF3F2 to a pasture} and Table \ref{table: morphisms between pastures}. Conversely, each of \eqref{hexagonal5}--\eqref{hexagonal8} implies that $M$ is without large uniform minors and thus Theorem \ref{thm: structure theorem for matroids without large uniform minors} applies. In turn, each of \eqref{hexagonal5}--\eqref{hexagonal8} excludes that any of $\D$, $\F_3$ and $\F_2$ occur as a factor $F_M$, and thus \eqref{hexagonal2}. 
\end{proof}


\subsubsection{\texorpdfstring{$\D\otimes\H$}{D tensor H}-representable matroids}
\label{subsubsection: ternary but not strictly ternary matroids}

Whittle describes in \cite[Thm.\ 1.3]{Whittle97} equivalent conditions that are satisfied by $\D\otimes\H$-representable matroids, which are conditions \eqref{ternary-not-strictly-ternary4}
and \eqref{ternary-not-strictly-ternary5} below. We augment Whittle's result with the following theorem.

\begin{thm}\label{thm: characterization of ternary but not strictly ternary matroids}
 Let $M$ be a matroid with foundation $F_M$. Then the following assertions are equivalent:
 \begin{enumerate}
  \item\label{ternary-not-strictly-ternary1} $M$ is $\D\otimes\H$-representable.
  \item\label{ternary-not-strictly-ternary2} $F_M\simeq F_1\otimes\dotsb\otimes F_r$ for $r\geq0$ and $F_1,\dotsc,F_r\in\{\U,\D,\H\}$.  
  \item\label{ternary-not-strictly-ternary3} $M$ belongs to one of $C_1$--$C_5$.
  \item\label{ternary-not-strictly-ternary4} $M$ is representable over $\F_3$ and $\C$.
  \item\label{ternary-not-strictly-ternary5} $M$ is representable over $\F_3$ and $\F_q$, where $q$ is an odd prime power congruent to $1$ modulo $3$.
  \item\label{ternary-not-strictly-ternary6} $M$ is representable over $\F_3$ and $\P$.
 \end{enumerate}
\end{thm}

\begin{proof}
 We show \eqref{ternary-not-strictly-ternary1}\implies\eqref{ternary-not-strictly-ternary2}\implies\eqref{ternary-not-strictly-ternary3}\implies\eqref{ternary-not-strictly-ternary1} and the equivalence of \eqref{ternary-not-strictly-ternary2} with each of \eqref{ternary-not-strictly-ternary4}--\eqref{ternary-not-strictly-ternary6}. The implications \eqref{ternary-not-strictly-ternary1}\implies\eqref{ternary-not-strictly-ternary2}\implies\eqref{ternary-not-strictly-ternary3}\implies\eqref{ternary-not-strictly-ternary1} follow from Theorem \ref{thm: representation classes of matroids without large uniform minors}. That \eqref{ternary-not-strictly-ternary2} implies \eqref{ternary-not-strictly-ternary4}--\eqref{ternary-not-strictly-ternary6} follows from Lemma \ref{lemma: conditions for morphisms from UDHF3F2 to a pasture} and Table \ref{table: morphisms between pastures}. Conversely, each of \eqref{ternary-not-strictly-ternary4}--\eqref{ternary-not-strictly-ternary6} implies that $M$ is without large uniform minors by Lemma \ref{lemma: binary and ternary matroids are without large uniform minors}, and thus Theorem \ref{thm: structure theorem for matroids without large uniform minors} applies. In turn, each of \eqref{ternary-not-strictly-ternary4}--\eqref{ternary-not-strictly-ternary6} excludes the possibility that either $\F_3$ or $\F_2$ occurs as a factor $F_M$, and thus \eqref{ternary-not-strictly-ternary2}. 
\end{proof}


\subsubsection{Representable matroids without large uniform minors}
\label{subsubsection: representable matroids without large uniform minors}

As a final application, we find the following equivalent characterization of matroids without large uniform minors which are representable over some field.

\begin{thm}\label{thm: characterization of representable matroids}
 Let $M$ be a matroid without large uniform minors and $F_M$ its foundation. Then the following assertions are equivalent:
 \begin{enumerate}
  \item\label{representable1} $M$ is representable over some field.
  \item\label{representable2} $F_M\simeq F_1\otimes\dotsb\otimes F_r$ for $r\geq0$ and either $F_1,\dotsc,F_r\in\{\U,\D,\H,\F_3\}$ or $F_1,\dotsc,F_r\in\{\U,\H,\F_2\}$.
  \item\label{representable3} $M$ belongs to one of $C_1$--$C_8$ or $C_{10}$.
  \item\label{representable4} $M$ is ternary or quaternary.
  \item\label{representable5} There is no morphism from $\F_2\otimes\D$ to $F_M$.
 \end{enumerate}
\end{thm}

\begin{proof}
 The equivalences \eqref{representable1}\iff\eqref{representable2}\iff\eqref{representable3} follow from Theorem \ref{thm: representation classes of matroids without large uniform minors}. The implications \eqref{representable2}\implies\eqref{representable4}\implies\eqref{representable5}\implies\eqref{representable2} can be derived by combining the implications  \eqref{ternary2}\implies\eqref{ternary1}\implies\eqref{ternary7}\implies\eqref{ternary2} from Theorem \ref{thm: characterization of ternary matroids} and \eqref{quarternary2}\implies\eqref{quarternary1}\implies\eqref{quarternary6}\implies\eqref{quarternary2} from Theorem \ref{thm: characterization of quarternary matroids}.
\end{proof}


\begin{small}
 \bibliographystyle{plain}
 \bibliography{matroid}
\end{small}

\end{document}